\documentclass[a4paper,11pt]{article}

\usepackage{amsmath}
\usepackage{amssymb}
\usepackage{amsthm}
\usepackage{graphicx}
\usepackage{mathabx}
\usepackage{verbatim}
\usepackage{color}
\usepackage[%
ocgcolorlinks,%
linkcolor=blue,%
filecolor=blue,%
citecolor=blue,%
urlcolor=blue]{hyperref}
\usepackage{lipsum}
       
\usepackage[font={small,it}]{caption}

\newcommand{\bbG}{\mathbb{G}}

\newcommand{\bbE}{\mathbb{E}}
\newcommand{\bbV}{\mathbb{V}}
\newcommand{\bbX}{\mathbb{X}}
\newcommand{\bbY}{\mathbb{Y}}
\newcommand{\Var}{\bbV{\rm ar}}

\newcommand{\bbP}{\mathbb{P}}

\newcommand{\bbN}{\mathbb{N}}
\newcommand{\bbZ}{\mathbb{Z}}
\newcommand{\bbR}{\mathbb{R}}

\newcommand{\eps}{\epsilon}

\newcommand{\ol}{\overline}

\newcommand{\wh}{\widehat}
\newcommand{\wc}{\widecheck}

\newcommand{\Argmin}{\text{Argmin}}
\newcommand{\argmax}{\text{argmax}}

\newcommand{\cA}{\mathcal A}

\newcommand{\cZ}{\mathcal Z}
\newcommand{\cC}{\mathcal C}

\newcommand{\cJ}{\mathcal J}
\newcommand{\cB}{\mathcal B}

\newcommand{\cT}{\mathcal T}

\newcommand{\cY}{\mathcal Y}

\newcommand{\cE}{\mathcal E}
\newcommand{\cN}{\mathcal N}
\newcommand{\cF}{\mathcal F}
\newcommand{\cM}{\mathcal M}
\newcommand{\cL}{\mathcal L}
\newcommand{\cR}{\mathcal R}
\newcommand{\cO}{\mathcal O}
\newcommand{\cQ}{\mathcal Q}

\newcommand{\rmc}{{\rm c}}
\newcommand{\rmF}{{\rm F}}
\newcommand{\rmd}{{\rm d}}
\newcommand{\rme}{{\rm e}}
\newcommand{\rmB}{{\rm B}}

\newcommand{\rmL}{{\rm L}}
\newcommand{\rmP}{{\rm P}}

\newcommand{\rmA}{{\rm A}}
\newcommand{\rmV}{{\rm V}}
\newcommand{\rmW}{{\rm W}}
\newcommand{\rmU}{{\rm U}}
\newcommand{\rmD}{{\rm D}}

\newcommand{\rmH}{{\rm H}}
\newcommand{\rmX}{{\rm X}}
\newcommand{\rmY}{{\rm Y}}
\newcommand{\rmZ}{{\rm Z}}
\newcommand{\rmQ}{{\rm Q}}
\newcommand{\rmO}{{\rm O}}

\newcommand{\rmE}{{\rm E}}

\newcommand{\rmM}{{\rm M}}
\newcommand{\rmS}{{\rm S}}

\newcommand{\bfA}{{\mathbf A}}
\newcommand{\bfG}{{\mathbf G}}
\newcommand{\bfW}{{\mathbf W}}
\newcommand{\bfF}{{\mathbf F}}
\newcommand{\bfN}{{\mathbf N}}
\newcommand{\bfL}{{\mathbf L}}

\newcommand{\bfT}{{\mathbf T}}
\newcommand{\bfR}{{\mathbf R}}
\newcommand{\bfX}{{\mathbf X}}

\newcommand{\bft}{{\mathbf t}}

\newcommand{\bfone}{{\mathbf 1}}
\newcommand{\bfXi}{{\bf \Xi}}
\newcommand{\bfLambda}{{\bf \Lambda}}
\newcommand{\bfGamma}{{\bf \Gamma}}
\newcommand{\bfS}{{\bf S}}

\newcommand{\rad}{{\rm rad}}
\newcommand{\Leb}{{\rm Leb}}
\newcommand{\supp}{{\rm supp}}

\newcommand{\wt}{\widetilde}

\newcommand{\la}{\langle}
\newcommand{\ra}{\rangle}

\newtheorem{theorem}{Theorem}[section]

\newtheorem{lem}[theorem]{Lemma}
\newtheorem{prop}[theorem]{Proposition}
\newtheorem{thm}[theorem]{Theorem}
\newtheorem*{rem*}{Remark}
\newtheorem{rem}[theorem]{Remark}

\newtheorem{mainthm}{Theorem}

\newtheorem{phasethm}{Theorem}

\numberwithin{equation}{section}
\graphicspath{{pic/}}

\newcommand{\textd}{\text{\rm d}\mkern0.5mu}

\newcommand{\texte}{\text{\rm e}\mkern0.7mu}

\newcommand{\1}{{1\mkern-4.5mu\textrm{l}}}
\renewcommand{\1}{\text{\sf 1}}

\newcommand{\E}{\mathbb E}

\newcommand{\G}{\mathbb G}

\newcommand{\N}{\mathbb N}

\newcommand{\R}{\mathbb R}

\newcommand{\Z}{\mathbb Z}

\newcommand{\TodoF}[1]{\footnote{{\bf Todo}: #1}}

\DeclareMathOperator{\Bigeo}{Bi-Geo}
\DeclareMathOperator{\Biexp}{Bi-Exp}
\DeclareMathOperator{\Poigeo}{Poi-Ge}

\title
{A Limit in Law for the Cover Time and Last Visited Vertex of Wired Planar Domains}
\author
{Oren Louidor\thanks{oren.louidor@gmail.com,  saglietti.sj@uc.cl.} 
\\ Technion, Israel \and Santiago Saglietti\footnotemark[1]\\ Pontificia Universidad\\ Cat\'olica de Chile}
\date{}

\begin{document}
\maketitle

\begin{abstract} 
We derive a scaling limit in law for the cover time of a simple random walk on a lattice version of a scaled-up planar domain with wired boundary conditions. The limiting distribution is that of a Gumbel Random Variable shifted randomly by an independent quantity which is equal to the full mass of a variant of the critical Liouville Quantum Gravity Measure on the same domain. We also derive a limit in law for the scaled location of the vertex visited last by the walk. Here the limit turns out to be precisely the critical Liouville Measure, normalized by its total mass. Both limits hold jointly with the limiting joint law explicitly described. 
These results resolve well known open problems in the field, in the case of wired boundary conditions. The proof is based on comparison with the extremal landscape of the discrete Gaussian Free Field, and in particular a version there-of obtained by conditioning the average value of the field to be zero.
\end{abstract}

\setcounter{tocdepth}{2}
\tableofcontents

\section{Introduction and results}

\subsection{Setup and main results}
\label{s:1}
Let $\rmD \subseteq \bbR^2$ be an non-empty bounded simply connected domain whose boundary is a finite union of connected components, each having positive Euclidean diameter.
For $N \geq 1$, consider its discretized scaled-up version
\begin{equation}
\label{e:1.1}
  \rmD_N := \Big\{ x \in \bbZ^2 :\: \rmd(x/N, \rmD^\rmc) > 1/N \Big\} \,. 
\end{equation} 
Above $\rmd(\cdot, \cdot)$ is the usual point-to-set distance with respect to the Euclidean norm on $\bbR^2$. Viewing subsets of $\bbZ^2$ as sub-graphs of the lattice, let $\wh{\rmD}_N$ be the graph obtained from $\rmD_N \cup \partial \rmD_N$, where $\partial \rmD_N$ is the outer boundary of $\rmD_N$, by contracting all vertices in $\partial \rmD_N$ into a single vertex to be denoted by $\partial$. 

On $\wh{\rmD}_N$ run a continuous time simple random walk $\bfX=(\bfX_u :\: u \geq 0)$, which starts from $\partial$ and transitions at (edge) rate $1$. The {\em cover time} of $\wh{\rmD}_N$ by $\bfX$ is defined as the first time by which all of the vertices of $\wh{\rmD}_N$ have been visited, namely
\begin{equation}
\label{e:1.1a}
	\wc{\bfT}_N := \min \big\{t \geq 0:\: \bigcup_{u \in [0,t]}\{ \bfX_u\} = \wh{\rmD}_N \big\} \,.
\end{equation}
The {\em last visited vertex} is
\begin{equation}
	\wc{\bfX}_N := \bfX_{\wc{\bfT}_N}.
\end{equation}

In this work, we show that the cover time and the last visited vertex obey a joint weak scaling limit at large $N$ and characterize their limiting law. The first of the main results asserts the existence of such limit.
\begin{mainthm}
\label{t:1.1}
Let $\wc{\bft}_N$ be defined via the relation:
\begin{equation}
\label{e:1.5}
\sqrt{\wc{\bft}_N} = \frac{1}{\sqrt{\pi}} \log N - \frac{1}{4 \sqrt{\pi}} \log \log N \,.
\end{equation}
Then, there exist random variables $\wc{\bfT}_\rmD$ and $\wc{\bfX}_\rmD$, defined on the same probability space and take values in $\bbR$ and $\rmD$ respectively, such that
\begin{equation}
\label{e:2.4}
\Big(\frac{\wc{\bfT}_N/|\rmD_N| - \wc{\bft}_N}{\log N}
\,,\,\, \frac{\wc{\bfX}_N}{N}\Big)
	\underset{N \to \infty}\Longrightarrow \, \big(\wc{\bfT}_\rmD,\, \wc{\bfX}_\rmD\big) \,.
\end{equation}
\end{mainthm}

Turning to describing the limiting law in~\eqref{e:2.4}, while far from being immediately apparent, the latter is intimately related to asymptotics of certain extreme order statistics of the {\em discrete Gaussian Free Field} (DGFF) on the same domain. We thus begin by recalling that the latter is a Gaussian random field $h_N = (h_N(x) :\: x \in \rmD_N)$ with zero mean and covariance $G_N$ given by 
\begin{equation}
\label{e:1.6i}
G_N:=\frac{1}{2}\G_{\rmD_N}^{\bfX} \,,
\end{equation} 
where $\G^{\bfX}_{\rmD_N}$ is the (discrete) Green Function on $\rmD_N$ associated with $\bfX$. Explicitly,
\begin{equation}\label{eq:defdgf}
\G^{\bfX}_{\rmD_N}(x,y):=\E_x \int_0^{\tau_\partial} 1_y(\mathbf{X}_u)\rmd u
\quad ; \qquad x,y \in \rmD_N \,,
\end{equation} 
where $\tau_\partial$ denotes the hitting time of $\partial$ and $\bbE_x$ denotes expectation with respect to the walk $\bfX$ started from $x$. (Note that, by our choice of rates, $\G^{\bfX}_{\rmD_N}$ is $\frac{1}{4}$ times the, more standard, Green Function
 for (discrete time) simple random walk on $\rmD_N$, which together with the choice of $1/2$ as the normalization constant in~\eqref{e:1.6i}, makes $h_N$ a $1/\sqrt{8}$ multiple of the DGFF on $\rmD_N$ under the common normalization. Such choices are arbitrary and made to simplify the constants in what follows.)

A standard way to record the min-extreme values of $h_N$ is via the so-called 
{\em multi-scale min-extremal process} of $h_N$ (henceforth extremal process for short). The latter is defined for $N, r \geq 0$ and a centering sequence $(m_N)_{N \geq 1}$ as
\begin{equation}\label{eq:defempgff}
\eta_{N,r}:=\sum_{x\in D_N}
\delta_{\,x/N}\otimes\delta_{\,h_N(x)+m_N}\otimes\delta_{\,\{h_N(x+y)-h_N(x)\colon y\in\Z^2\}} 
\1_{\{h_N(x)=\min_{z\in \rmB(x; r)} h_N(z)\}}\,.
\end{equation}
This is a point process on $\ol{\rmD} \times \bbR \times \bbR^{\bbZ_2}$ which records the {\em scaled} position, {\em centered} value and relative heights of all $r$-local minima of $h_N$. The sequence $(-m_N)_{N \geq 1}$ captures the typical value of the min-extremes of $h_N$.

Indeed, it was shown in~\cite[Theorem~2.1]{BL3} that with the choice 
\begin{equation}
m_N:= \frac{1}{\sqrt{\pi}} \log N - \frac{3}{8\sqrt{\pi}}  \log \log N \,,
\end{equation}
and whenever 
\begin{equation}
\label{e:1.11i}
r_N \wedge (N/r_N) \underset{N\to\infty}\longrightarrow \infty \,,
\end{equation} 
we have
\begin{equation}
\label{e:1.10i}
\eta_{N,r_N}
\underset{N \to \infty}\Longrightarrow \eta_\rmD 
\quad; \qquad
	\eta_\rmD \,\big|\, \cZ_\rmD \sim\  \text{\rm PPP}\bigl(\cZ_\rmD(\textd x)\otimes\texte^{-\alpha s}\textd s\otimes\nu(\rmd \omega)\bigr) \,,
\end{equation}
as (random) elements of the space of Radon measures on $\ol{\rmD} \times \bbR \times \bbR^{\bbZ^2}$ equipped with the vague topology. 
Above, ${\rm PPP}$ stands for {\em Poisson Point Process}, $\alpha = 4\sqrt{\pi}$ is a constant, $\nu$ is a deterministic distribution on $\bbR_+^{\bbZ_2}$ (whose law can be explicitly described) and $\cZ_\rmD$ is a random Radon measure on $\rmD$ known as the {\em critical Liouville Quantum Gravity Measure} (cLQGM) on~$\rmD$. The similarity between $m_N$ and $\sqrt{\wc{\bft}_N}$ is, of course, not a coincidence (see Section~\ref{sec:phaseA}).

A formal definition for the cLQGM $\cZ_\rmD$ can be given in terms of its (random) Radon-Nykodim Derivative w.r.t. the Lebesgue Measure on $\rmD$ (up-to an unimportant constant which we omit):
\begin{equation}\label{eq:deflqgm}
	\cZ_\rmD(\rmd x) = \big(\alpha \bbE h^2_\rmD(x) - h_\rmD(x)\big) \rme^{\alpha h_\rmD(x) - \frac{\alpha^2}{2}\bbE h^2_\rmD(x)}\, \rad_\rmD(x)^2 \rmd x \,,
\end{equation}
where $\alpha$ is as before, $\rad_\rmD(x)$ is the conformal radius of $x$ about $\rmD$, and $h_\rmD$ is the {\em continuous Gaussian Free Field} (CGFF) on $\rmD$. The CGFF can be defined as a proper $N \to \infty$ scaling-limit of $h_N$, or, more directly, as a centered Gaussian field on $\rmD$ with covariance $G_\rmD$ given by 
\begin{equation}\label{eq:defcgff}
G_\rmD:=\frac{1}{2}{\G}_\rmD \,.
\end{equation}
Above, $\G_\rmD$ denotes the (usual, continuous) Green Function on $\rmD$ which, under our normalization, 
is also the $N \to \infty$ limit of $\G^{\bf{X}}_{\rmD_N}\big(\lfloor N \cdot \rfloor, \lfloor N \cdot \rfloor\big)$ from~\eqref{e:1.6i}.
Due to the logarithmic singularity of the Green Function on the diagonal, the field $h_\rmD$ exists only as a random distribution. As such, the definition in~\eqref{eq:deflqgm} is only a formal one and a limiting procedure is required to make sense of it rigourosly (see, e.g.,~\cite{rhodes2014gaussian}).

Aside from the measure $\cZ_\rmD$, the description of the limiting law in Theorem~\ref{t:1.1} also requires the empirical average of the continuous field:
\begin{equation}
\label{e:1.15i}
	\ol{h}_\rmD := \frac{1}{\Leb(\rmD)} \int_{\rmD} h_\rmD(x) \rmd x \,.
\end{equation}
While again only a formal definition, $\ol{h}_\rmD$ can be naturally made sense of as
\begin{equation}
\label{e:1.7p}
\ol{h}_\rmD \sim \cN\big(0, \ol{\sigma}^2_\rmD\big)
\quad; \qquad 
		\ol{\sigma}^2_\rmD := \frac{1}{\big(\Leb(\rmD)\big)^2} \int_{\rmD \times \rmD} G_\rmD(x,y) \,\rmd x \rmd y \,.
\end{equation}
Equivalently, $\ol{h}_\rmD$ is the $N \to \infty$ weak limit of the discrete field averages:
\begin{equation}
\label{e:1.17i}
	\ol{h}_N := \frac{1}{|\rmD_N|} \sum_{x \in \rmD_N} h_N(x) \,.
\end{equation}

While it is elementary to see that discrete averages tend in law to their continuous analog, the joint occurrence of this limit 
together with that of the extremal process of the field~\eqref{e:1.10i}, is a non-trivial fact. This fact is crucially required both for showing the existence of the limit in Theorem~\ref{t:1.1}, and for establishing a coupling between $\cZ_\rmD$ and $\ol{h}_\rmD$, using which the limiting law can be defined. This leads to the second main result in this work.
\begin{mainthm}
\label{t:1.2i}
There exists a coupling between
$\eta_\rmD$ from~\eqref{e:1.10i}, $\cZ_\rmD$ from~\eqref{eq:deflqgm} and $\ol{h}_\rmD$ from~\eqref{e:1.7p}, such that $\eta_\rmD$ and $\cZ_\rmD$ 
are related as in~\eqref{e:1.10i}, $\cZ_\rmD$ is measurable w.r.t. $\eta_{\rmD}$ and $\eta_\rmD$ is conditionally independent of $\ol{h}_\rmD$ given  $\cZ_\rmD$. Under this coupling, if $(r_N)_{N \geq 1}$ satisfies~\eqref{e:1.11i} then
\begin{equation}
\label{E:1.9}
\big(\eta_{N,r_N},\, \ol{h}_N\,\big)
\underset{N \to \infty}\Longrightarrow \big(\eta_\rmD,\, \ol{h}_\rmD \big) \,,
\end{equation}
with the underlying space for the first component taken as in~\eqref{e:1.10i}. In particular, the joint law of $\eta_\rmD$, $\cZ_D$ and $\ol{h}_\rmD$ in any such coupling is the same.
\end{mainthm}
\noindent
%As $\cZ_\rmD$ is measurable w.r.t. $\eta_\rmD$, the joint law of $\ol{h}_\rmD$ and $\cZ_\rmD$ in the coupling of Theorem~\ref{t:1.2i} is uniquely determined by the joint law of $\eta_\rmD$ and $\ol{h}_\rmD$.

In order to prove the theorem, we write $h_N$ as the independent sum of its orthogonal projection (w.r.t. a suitable inner product) onto the space of mean zero fields on $\rmD_N$ (i.e. functions on $\rmD_N$ which sum-up to zero) and its complement. This gives the ``pointwise'' decomposition
\begin{equation}
\label{e:1.22i}
h_N = \big(\psi_N \ol{h}_N\big) \oplus \wh{h}_N\,,
\end{equation} 
where $\psi_N$ is an explicit deterministic field on $\rmD_N$, $\ol{h}_N$ is the average of the field as before and $\wh{h}_N$ is a Gaussian field on $\rmD_N$ whose law is that of $h_N$ conditioned on its average to be zero. We call the latter the {\em zero-average discrete Gaussian Free Field}. Here and after, we write $\oplus$ to indicate the sum of independent quantities.

The advantage of the above decomposition is that, thanks to the independence of the summands in~\eqref{e:1.22i}, in order to show the joint convergence in law of the empirical average and the extremal process of $h_N$, it suffices to separately establish the weak convergence of $\ol{h}_N$ and that of the extremal process of $\wh{h}_N$, namely
\begin{equation}
\label{e:1.14j}
\wh{\eta}_{N,r}:=\sum_{x\in D_N}
\delta_{\,x/N}\otimes\delta_{\,\wh{h}_N(x)+m_N}\otimes\delta_{\,\{\wh{h}_N(x+y)-\wh{h}_N(x)\colon y\in\Z^2\}} 
\1_{\{\wh{h}_N(x)=\min_{z\in \rmB(x; r)} \wh{h}_N(z)\}}\,.
\end{equation}
While obtaining asymptotics in law for $\ol{h}_N$ is straightforward, deriving a limit in law for $\wh{\eta}_{N,r}$ requires non-trivial work.
\begin{mainthm}	
\label{t:1.4i}
Assume $(r_N)_{N \geq 1}$ is as in~\eqref{e:1.11i}. Then,
with the underlying space is as in~\eqref{e:1.10i},
\begin{equation}
\label{E:1.9n}
\wh{\eta}_{N,r_N}
\underset{N \to \infty} \Longrightarrow \ \wh{\eta}_\rmD \,,
\quad ; \qquad
	\wh{\eta}_\rmD \,\big|\, \wh{\cZ}_\rmD \sim\  \text{\rm PPP}\bigl(\wh{\cZ}_\rmD(\textd x)\otimes\texte^{-\alpha s}\textd s\otimes\nu(\rmd \omega)\bigr) \,,
\end{equation}
where $\alpha$ and $\nu$ as in~\eqref{e:1.10i}, and $\wh{\cZ}_\rmD$ is a random Radon measure on $\rmD$.
The law of $\wh{\cZ}_\rmD$ is uniquely defined via the relation:
\begin{equation}
\label{e:1104.14n}
	\wh{\cZ}_\rmD \rme^{\alpha \psi_\rmD \ol{h}_\rmD} \overset{\rmd} = \cZ_\rmD \,,
\end{equation}
where $\ol{h}_\rmD$ is as in~\eqref{e:1.7p} and independent of $\wh{\cZ}_\rmD$, $\psi_\rmD: \rmD \to \bbR$ is the deterministic field:
\begin{equation}
\label{e:1.9p}
	\psi_\rmD(\cdot) : =  \Leb(\rmD) \frac{\int_\rmD G_\rmD(\cdot,z)\, \rmd x}{\int_{\rmD \times \rmD} G_\rmD(x,y)\, \rmd x \rmd y} \,,
\end{equation}
and $\cZ_\rmD$ is as in~\eqref{eq:deflqgm}. The joint law of $\ol{h}_\rmD$ and $\cZ_\rmD$ as determined by~\eqref{e:1104.14n} is the same as that in the coupling of Theorem~\ref{t:1.2i}.
\end{mainthm}

The theorem, which (we believe) is of a stand-alone interest, shows that the extremal process associated with the zero-average DGFF admits a similar limit as that of the unconditional field. The centering sequence is the same, as well as the exponential density and the cluster law in the last two components of the intensity measure. The only difference lies in the first component, where a different random measure $\wh{\cZ}_\rmD$ takes the place of $\cZ_\rmD$. Relation~\eqref{e:1104.14n} provides another characterization of the coupling between $\ol{h}_\rmD$ and $\cZ_\rmD$, as used in Theorem~\ref{t:1.3i} to define the limiting law.

We now have all ingredients to describe the limiting law in Theorem~\ref{t:1.1}. Henceforth, for a measure $\mu$ we write $\|\mu\|$ and $\mu^\circ$ for its total mass and normalized version, respectively, so that $\mu^\circ = \mu/\|\mu\|$. We also recall that ${\rm Gumbel}(\mu, \beta)$ denotes a probability distribution on $\bbR$ whose CDF is 
given by $F(x) = \exp(-\rme^{-(x-\mu)/\beta})$. 
\begin{mainthm}
\label{t:1.3i}
The joint law of $\wc{\bfT}_\rmD$, $\wc{\bfX}_\rmD$ from Theorem~\ref{t:1.1} is given by
\begin{equation}
\label{e:1.6q}
\big(\wc{\bfT}_\rmD, \wc{\bfX}_\rmD\big) \,\big|\, \cZ_\rmD, \ol{h}_\rmD\, \sim  {\rm Gumbel}\,\Big(\frac{1}{2\pi} (\log (C_\star\|\cZ_\rmD\|) + 4\sqrt{\pi}\,\ol{h}_\rmD\big) \,,\,\, \frac{1}{2\pi} \Big) \otimes \cZ_\rmD^\circ \,,
\end{equation}
where $C_\star \in (0,\infty)$ is a universal constant, $\cZ_\rmD$ is as in~\eqref{eq:deflqgm}, $\ol{h}_\rmD$ is as in~\eqref{e:1.7p} and they are coupled together as in 
Theorem~\ref{t:1.2i} (equivalently, Theorem~\ref{t:1.4i}).
Explicitly,
\begin{equation}\label{eq:alternativeformulation}
\bbP \big(\wc{\bfT}_\rmD \leq s,\,\, \wc{\bfX}_\rmD \in \rmd x\big)
= \bbE\,\rme^{-C_\star\|\cZ_\rmD\| \rme^{-2\pi s + 4\sqrt{\pi} \,\ol{h}_\rmD}}\cZ^\circ_\rmD(\rmd x) \,.
\end{equation}	
\end{mainthm}
Theorem~\ref{t:1.3i} shows that the limiting law of the cover time is that of a Gumbel, shifted by an independent random quantity, which is the sum of the logarithm of the total mass of the cLQGM on $\rmD$ and a dependent Gaussian random variable. This ``randomly shifted Gumbel'' form is a well-known characteristic of limiting laws of extreme order statistics in logarithmically correlated systems (see Subsection~\ref{s:1.2}). The limiting law of the last visited vertex is that of the cLQGM itself, normalized to become a random probability measure on $\rmD$. 

\subsection{Context and open problems}
\label{s:1.2}

\subsubsection{Local time extreme order statistics, logarithmically correlated fields and Gaussian multiplicative chaos}
The cover time can be alternatively expressed in terms of the time spent by the walk at each vertex. Indeed, if 
we define the {\em local time} at $x \in \wh{\rmD}_N$ by {\em real time} $u \in \bbR_+$ as
\begin{equation}
	\bfL_{t}(x) := \int_{0}^t 1_x\big(\bfX_{u}\big) \rmd u\,,
\end{equation}
then, the cover time of $\wh{\rmD}_N$ can be equivalently defined as
\begin{equation}
	\label{e:1.4a}
	\wc{\bfT}_N := \min \big\{t \geq 0:\: \min_{x \in \wh{\rmD}_N} \bfL_{t}(x) > 0 \big\} \,. 
\end{equation}
It is thus the first real time by which the minimal local time is positive. 

Display~\eqref{e:1.4a} recasts the cover time as an extreme order statistic of the field of local times $\bfL_t := (\bfL_t(x) :\: x \in \wh{\rmD}_N)$. This field (and its square root) is known to have approximate logarithmic correlations in two dimensions. Indeed, the covariance matrix of $\bfL_t$ (as well as that of $\sqrt{\bfL_t}$) is closely related to the discrete Green Function $\bbG_{\rmD_N}^{\bfX}$ on $\rmD_N$, which in our normalization takes the form
\begin{equation}
	\bbG^{\bfX}_{\rmD_N}(x,y) = \frac{1}{2 \pi} \log N - \frac{1}{2 \pi}\log \big(\|x-y\| \vee 1\big) + O(1) \,,
\end{equation} 
for $x$,$y$ in the bulk of $\rmD_N$ (see, e.g. \cite[Lemma~3.1]{Tightness} for a more precise statement).
%\ToS{Related to the Green Function, not the covariance of the field. Hence change of consts.}

Extreme order statistics of fields with (approximately) logarithmic correlations have been the subject of extensive research over the past two decades. 
Other notable examples which have been studied include Branching Brownian Motion, the discrete Gaussian Free Field on planar domains or tree-like graphs, the (logarithm of the) modulus of the Riemann zeta function on a uniformly chosen interval on the critical line and the (logarithm of the) modulus of the characteristic polynomial of a CUE matrix on the unit circle. For more information, a recent survey on this subject can be found in~\cite{arguin2016extrema}.

In all known examples, statistics of extreme values of such fields exhibit similar asymptotic features. These include a negative log-log correction to the typical height of the maximum/minimum compared to the i.i.d. case, randomly shifted Gumbel fluctuations in the limit, clustering of extreme values and a Cox-type Law for their joint limiting distribution and connection to Gaussian Multiplicative Chaos measures at criticality. The results here therefore provide additional evidence to support the conjectural universality of such features among all logarithmically correlated fields. 

The appearance of the cLQGM in the limiting law of the cover time and last visited vertex ties the cover time asymptotics to the theory of {\em Gaussian Multiplicative Chaos} (GMC). The latter are random measures obtained, formally, by taking the exponential of a log-correlated Gaussian Field as their Radon-Nykodim derivative w.r.t. a deterministic reference measure. The case of the cLQGM corresponds to the CGFF as the underlying field. The theory of GMC, which goes back to the works of H{\o}egh-Krohn~\cite{hoegh1971general} and Kahane~\cite{kahane1985chaos}, has seen considerable advances in the past decades, with far reaching connections to random geometry, combinatorics and conformal field theory. See~\cite{rhodes2014gaussian} and, more recently, also~\cite{Berestycki_Powell_2025}, for excellent sources on the subject.

\subsubsection{Historical account}
Let us now briefly survey the historical development of this problem. One of the earliest mentioning of the search for asymptotics of the cover time can be found in the editorial piece of Hebert Wilf~\cite{wilf1989editor} who, for self-amusement, simulated a random walk on his television screen, acting as a two dimensional discrete torus $\bbZ^2/(N\bbZ^2)$ of side length $N \geq 1$. In this casual-reading article, the author wondered about the time it took for the screen to become ``completely white''. 

An upper bound of the correct first order was found by Aldous~\cite{aldous1991threshold} in '89. A non-matching first order lower bound (in expectation) was then given by Lawler~\cite{lawler1993covering} in '92. This was later improved in the seminal work of Dembo, Peres, Rosen and Zeitouni~\cite{dembo2004cover} in '04, who settled the question of first order asymptotics. Key to the proof in~\cite{dembo2004cover} were sharp estimates for the first and second moments of the number of unvisited vertices, subject to a ``truncation'' event, which was added in order to make this random quantity concentrated around its mean.

Turning to the second order term, Ding~\cite{ding2012cover} in '12 showed that it is of order $O(\log \log N)$. This result relied on an earlier work by Ding, Lee and Peres~\cite{ding2011cover} in `11 who exploited, for the first time (as far as we know), the connection between the cover time and the maximum/minimum of the discrete Gaussian Free Field on the same underlying graph. As our proof is more in line with this approach, further details on this connection could be found in Subsection~\ref{sec:phaseA}. 

The full derivation for the second order term was done by Belius and Kistler~\cite{belius2017subleading} in `17 for the analogous problem of the cover time by a Brownian Motion (sausage) and then by Abe~\cite{abe2021second} in '20 for a random walk. These works, however, stopped short of showing tightness. For the case of wired planar domains, as in the present paper, this was shown recently by the authors  in~\cite{Tightness}. For the case of Brownian Motion on the two dimensional sphere, this was recently settled by Belius, Rosen and Zeitouni~\cite{belius2020tightness}.

The case of the discrete torus $\bbZ^d/(N\bbZ^d)$ in dimension $d \geq 3$ was settled by Belius~\cite{belius2013gumbel} in `12, which showed that the limiting law of the cover time is that of a non-shifted Gumbel. The vanishing of the random shift is a consequence of the fast decay, away from the diagonal, of the Green Function in this case, putting the local time field in the same extreme-universality class as that of fields with independent entries. A more general result was recently given in~\cite{berestycki2022universality}, showing that i.i.d.-type asymptotics are obtained as soon as the diameter of the graph is $o(\sqrt{n/\log n})$, where $n$ is the number of its vertices. Random geometric graphs on $\bbR^d$ were considered in~\cite{martinez2025jump}.

A more related graph is that of the rooted regular tree of depth $n$, where correlations are hierarchical and thus approximately logarithmic in a dyadic embedding of the leaves on the real line. Here, analog results to those in Theorem~\ref{t:1.1} and Theorem~\ref{t:1.3i} (albeit for the cover time only) were shown by the authors and Cortines in~\cite{cortines2021scaling}, and also by Dembo, Rosen and Zeitouni in~\cite{dembo2021limit}. The limiting law of the cover time is again a randomly shifted Gumbel, with the random shift directly related to the so-called derivative martingale, which is the tree analog of the total mass of the cLQGM.

Lastly, it is worth mentioning the substantial work that has been conducted on the ``opposite side'' of extremality in local times of random walks and Brownian Motion. Here one typically studies the so-called most favorite (or thick) and $\lambda$-favorite points of the underlying motion, namely points where the local time is maximal or at least a $\lambda \in (0,1)$ fraction of the typical value of the maximum. 

Research in this area started in the pioneering work of Erd\"os and Taylor~\cite{ErdosTaylor} who derived non-matching upper and lower bound on the order of $(\log n)^2$ for the highest local time of a planar simple random walk run until time $n$. The lower bound was ``corrected'' more than forty years later, in the work of Dembo, Peres, Rosen and Zeitouni~\cite{DPJZ_Thick}, establishing $\pi^{-1} (\log n)^2$ as the leading order term in the asymptotics for the local time at the most favorite point (for both the simple random walk and Brownian Motion). The second order term along with tightness, in the case of Brownian Motion, was recently shown by Rosen~\cite{Rosen_Thick}.

$\lambda$-favorite points were also treated in~\cite{DPJZ_Thick}, where it was showed that with high probability there are $n^{1-\lambda + o(1)}$ many points whose local time is at least $\lambda \pi^{-1} (\log n)^2$ for $\lambda \in (0,1)$. These asymptotics were sharpened considerably by Jego~\cite{Jego1, Jego2, Jego3}, who recently established a weak scaling limit for a random-measure encoding of these points. The limiting law, describing the asymptotics of both the number and spatial distribution of the $\lambda$-favorite points turns out to be the so-called (sub-critical) {\em Brownian Multiplicative Chaos} - an analog, in some sense, to the GMC but driven by the local time of Brownian Motion.

A similar derivation in the case where the underlying walk is run for times comparable with the cover time of the domain (with periodic boundary conditions), was  by Abe and Biskup~\cite{AbeBiskup1, AbeBiskup2}. In this very different order of time scales, the limit turns out to be closely related to the (sub-critical) GMC associated with the CGFF on the domain. Similar problems when the underlying graph is the rooted regular tree were studied by Abe~\cite{Abe_Tree} and by Biskup and Louidor~\cite{BL_FavoritePoint}.

\subsubsection{Open problems and conjectures}
Lastly, we would like to address the imposed boundary conditions. The reason for considering wired boundary conditions, stems from our use of a comparison tool, know as the Generalized Second Ray-Knight Theorem (Theorem~\ref{t:103.1}) which relates the law of the local time field to that of the DGFF. Putting boundary conditions on the boundary of the domain for the DGFF, translates under this theorem, to a wiring of the boundary for the walk with respect to which the local time is defined. For our method of proof, this is thus absolutely essential. 

It thus remains an open problem to derive asymptotics for the cover time when different boundary conditions are assumed, e.g. periodic (the torus case) or free (the no-boundary case). While second order asymptotics are available for the mean value of the cover time in the case of the torus (as mentioned above), showing weak convergence or even tightness are still open problems. We conjecture that similar results as in Theorems~\ref{t:1.1} and~\ref{t:1.3i} hold. In particular, we believe that the limiting law takes the same form as in~\eqref{e:1.6q} only with a different random measure and (dependent) Gaussian variable appearing in place of $\cZ_\rmD$ and $\ol{h}_\rmD$.

\subsection*{Organization of the Paper}
The rest of the paper is organized as follows. In the following Section~\ref{s:2} we provide a top-level proof for the convergence of the cover time and the last visited vertex, namely  Theorems~\ref{t:1.1} and~\ref{t:1.3i}. The proof builds on a key idea from~\cite{Tightness}, whereby the running time of the walk is divided into two consecutive phases: A and B of prescribed lengths. The analysis of these phases are summarized in Theorem~\ref{t:2.1} and Theorem~\ref{t:2b}, which are stated in Section~\ref{s:2} and then used to prove the main convergence results. The proofs of these two theorems are deferred to the sections which follow.

The proof of Theorem~\ref{t:2.1}, which constitutes one of the main efforts in the manuscript, takes up Sections~\ref{sec:phaseA} through~\ref{s:barrier}, as well as Section~\ref{s:5n}. We also use a top-down approach for the presentation of the proof of Theorem~\ref{t:2.1}, with Section~\ref{sec:phaseA} providing a top-level proof of the theorem, using results which are stated in this section but proved only in Sections~\ref{sec:conseq},~\ref{sec:iid},~\ref{s:barrier} and~\ref{s:5n}. The proof of Theorem~\ref{t:2b} then follows in Section~\ref{sec:proofb}. 

Section~\ref{s:5n} includes the proof for the joint convergence of the extremal process and the field average, as well as that of the convergence of the extremal process of the zero-averaged DGFF. In other words, it is where Theorems~\ref{t:1.2i} and~\ref{t:1.4i} are proved. Lastly, Appendix~\ref{sec:auxproof} includes proofs of miscellaneous auxiliary results which are used throughout the article. These results are either standard, shown before or technical but not illuminating and thus their proofs were deferred to the end of this work.

\section{Existence and identification of the limit: Top-level proofs of Theorems~\ref{t:1.1},~\ref{t:1.3i}}
\label{s:2}
In this section we provide a top-level proof for the joint weak convergence of the cover time and last visited vertex. Unlike the separation in the exposition of the results, the proof takes care of both the existence of the limit and its characterization.

\subsection{Notation}
Let us introduce some notation. As the proofs follow a multi-scale analysis along an exponential scale, we shall henceforth treat all length parameters as if given on such scale. Therefore, from now on and with a slight override of notation, we shall henceforth write $\rmD_n$ to mean 
$\rmD_N$ with $N=\rme^n$, as was defined in~\eqref{e:1.1} (throughout $n > 0$ need not be an integer).

The Euclidean and Supremum norms will be denoted by $\|\cdot\|$ and $|\cdot|$, respectively. Given $x \in \bbZ^2$ and $r \geq 0$, we let
$\rmB(x;r) := \{y \in \Z^2 :\: \|x-y\| \leq \lfloor \rme^r \rfloor \}$ and $\rmQ(x;r) := \{y \in \Z^2 :\: x-y \in  (-\frac{\lfloor \rme^r \rfloor}{2},\frac{\lfloor \rme^r \rfloor}{2}]^2 \}$ be the discrete ball (resp. box) of radius (resp. side-length) $\lfloor \rme^r\rfloor$ about $x$, which will be omitted from the notation whenever it is the origin.
%Given $\rmU \subset \bbZ^2$ non-empty, the \textit{diameter} of $\rmU$ will be as usual the supremal (Euclidean) distance between any two points in $\rmU$. For $r \geq 0$, we will define the $r$-\textit{bulk} of $\rmU$ as $\rmU^r := \{x \in \bbZ^2 :\: \rmd(x, \rmU^\rmc) > \rme^r \}$. Whenever $U=\rmD_n$ is the $\rme^n$-scale of some domain $\rmD \subseteq \R^2$, we will write $\rmD_n^r$ as short for $(\rmD_n)^r$.
%Given $\rmU \subset \bbZ^2$ non-empty, the \textit{diameter} of $\rmU$ will be as usual the supremal (Euclidean) distance between any two points in $\rmU$. For $r \geq 0$, we will define the $r$-\textit{bulk} of $\rmU$ as $\rmU^r := \{x \in \bbZ^2 :\: \rmd(x, \rmU^\rmc) > \rme^r \}$. Whenever $U=\rmD_n$ is the $\rme^n$-scale of some domain $\rmD \subseteq \R^2$, we will write $\rmD_n^r$ as short for $(\rmD_n)^r$.
% \Maybe{Later a calligraphic font is used for continuum domain. This is inconsistent with our main domain $D$. Also, I'm not sure I like the visual outcome. Maybe $\cU$ but $U_N$. \textcolor{red}{Yes, we used that because $\rmD$ is not just any continuum domain, but THE continuum domain from the main result. In the other paper we were using several continuum domains at once ($\rmD$, unit balls, etc.), so this extra notation was perhaps necessary. Here, however, it may not be.}}
 As in the introduction, we shall view subsets $\rmU \subset \bbZ^2$ as sub-graphs of the square lattice, inheriting from it all edges which connect two any vertices in $\rmU$. For such sets, $\partial \rmU$ shall denote its \textit{outer boundary}, namely the set of vertices in $\bbZ^2 \setminus \rmU$ which share an edge of $\bbZ^2$ with at least one vertex in $\rmU$, and $\ol{\rmU}$ will be defined as $\rmU \cup \partial \rmU$.

Finally, all constants throughout will be assumed to be finite and positive, with their value possibly changing from one use to the next, and, for $f:\mathcal{X} \rightarrow \mathcal{Y}$ and $\mathcal{Z} \subseteq \mathcal{X}$, we will write $f(\cZ_\rmD)$ to denote the set $\{ f(z) : z \in \cZ_\rmD\}$.

\subsection{Time parametrization and rescaling}
The first step in the proof is to ``reparameterize'' time so that time is ``measured in terms of the local time spent at the boundary''. Formally, for $x \in \wh{\rmD}_n$ and $t \geq 0$, we define the \textit{local time}~at~$x$ by \textit{real time} $t$ as
\begin{equation}
\bfL_{n,t}(x) \equiv \bfL_t(x):=\int_0^t 1_x(\bfX_u)\,\rmd u\,,	
\end{equation} where we omit the dependence in $n$ from the notation for simplicity.
In addition, recalling that $\partial$ is the boundary vertex, we let
\begin{equation}
\bfT_{n,t} \equiv \bfT_t := \bfL^{-1}_t(\partial)                                                                                                                                                                                                                                                                                                                                                                                                                                                                                                                                                                                                                                                                                                                                                                                                                                                                                                                                                                                                                                                                                                                                                                                                                                                                                                                                                                                                                                                                                                                                                                                                                                                                                                                                                                                     = \inf \big\{u \geq 0 :\: \bfL_{u}(\partial) > t \big\} \,.
\end{equation}
be the generalized inverse of $u \mapsto \bfL_{u}(\partial)$. We shall loosely refer to this quantity as the real time by which the local time at $\partial$ is $t$, or $\partial$\textit{-time} $t$ for short. 

To distinguish between quantities in which time means real time and those in which time means $\partial$-time, we shall denote the former quantities in bold. Thus, while $\bfL_t(x)$ is the local time at $x$ by real time $t$, the local time at $x$ by $\partial$-time $t$ is 
\begin{equation}
L_{t}(x) := \bfL_{\bfT_{t}}(x)\,.
\end{equation}
Similarly, the cover time in the $\partial$-time parameterization, or {\em cover $\partial$-time}, is then
\begin{equation}
\label{e:1.4}
\wc{T}_n := \min \big\{t \geq 0:\: \min_{x \in \rmD_n} L_{t}(x) > 0 \big\} \,.
\end{equation}
The $\sigma$-algebra corresponding to $\bfT_t$ as a stopping time w.r.t. the natural filtration of the walk $\bfX$, will be denoted by
$\cF_t$.

Next, we also re-tune the rates of the underlying random walk $\bfX$, so that the (edge) transition rates are now $1/(2\pi)$ instead of $1$. By the scaling properties of the exponential distribution, this modification only amounts to a change in the (scaling of the) cover time by a factor of $2\pi$. At the same time, this makes all constants a bit simpler and, as such, reduces visual complexity. 
As we shall see, the $\partial$-time analog of $\wc{\bft}_n$ under this time rescaling and reparameterization is $\wc{t}_n$, which we define via
\begin{equation}
\label{e:2.4n}
\sqrt{\wc{t}_n} = \sqrt{(2\pi)\wc{\bft}_n } = 
\sqrt{2} n - \frac{1}{2\sqrt{2}} \log n \,.
\end{equation}

This retuning of rates will also produce modifications to other objects defined in the introduction. Indeed, henceforth both the discrete and continuous Green Function will increase by a factor of $2\pi$. Consequently, all fields (e.g., the continuum and discrete GFFs, as well as LQGM-type random measures, will be $\sqrt{2\pi}$ times their value as given by their definition in the introduction. On the other hand, the value of the critical parameter decreases by the same factor and, thus, henceforth reset to
\begin{equation}
\alpha := 2\sqrt{2} \,.
\end{equation}
%\ToS{You wrote $1/\sqrt{2\pi}$. The field is now larger and so is the measure. Not sure that this retuning is very wise. Maybe we should remove it in the next version to avoid creating unnecessary confusion. \textcolor{red}{Are you sure? If you scale the DGFF $h_n$ by $c$ and derive the asymptotics of the maximum for $ch_n$ from those corresponding to $h_n$, don't you get that $\alpha$ and $\cZ_\rmD$ are modified by exactly the same scaling $\frac{1}{c}$ (and not $\frac{1}{c}$ and $c$, respectively)? If so, then the field increases but the measure and $\alpha$ decrease. As for the retuning, agree but let's discuss later....}}

\subsection{Separation into Phases A and B}
Given $n > 0$, for $t,u \geq 0$ we define
\begin{equation}
\rmF_{n,t}(u) := \big\{ x \in \rmD_n :\: L_t(x) \leq u \big\} \,,
\end{equation}
so that $\rmF_{n,t}(0)$ denotes the set of unvisited vertices by $\partial$-time $t$. Since $\{\wc{T}_n \leq t\} = \{\rmF_{n,t}(0) = \emptyset\}$, we need to study this set at times $t$ where $\sqrt{t} = \sqrt{\wc{t}_n} + \Theta(s)$. 

To this end, we split the running $\partial$-time of the walk into two consecutive phases: A and B. The $\partial$-time at these phases will be $t_n^A$ and $t^B_n+sn$, respectively, where $s \in \R$ and  
\begin{equation}
\label{e:2.4.1}
\sqrt{t_n^A} := \sqrt{2} n - \frac{3}{4\sqrt{2}}  \log n 
\quad, \qquad
t_n^B : = \frac{1}{2} n\log n \,.
\end{equation}
Then, the total running $\partial$-time of the walk $t_n(s)$ will satisfy,
\begin{equation}
\label{e:2.5}
\sqrt{t_n(s)} \equiv \sqrt{t_n^A + t_n^B + sn} = \sqrt{\wc{t}_n} + \frac{1}{2\sqrt{2}} s + o(1) \,,
\end{equation}
where $o(1) \to 0$ as $n \to \infty$ for fixed $s$. 

We observe that
\begin{equation}
\label{e:2.16}
\rmF_{n,t_n(s)}(0) = \rmF_{n,t_n^A}(0) \cap \rmF'_{n,t_n^B+sn}(0) \,,
\end{equation}
where $\rmF'_{n,t_n^B+sn}(0)$ denotes the set of unvisited vertices during phase B. We shall therefore study the set of unvisited sites during each phase separately.

\subsubsection{Phase A}
Let us begin with phase A, which takes up most of the work in the proof of Theorem~\ref{t:1.1}. Here we show that by the end of the phase, except for a negligible set of vertices whose size is $o(\sqrt{n})$, the remaining $O(1)$ local time vertices are arranged in clusters of microscopic size whose mutual distance is macroscopic. The number of such clusters divided by $\sqrt{n}$ admits a limit in law, which holds jointly with that of the total (real) running time of the walk at this phase.

For a precise statement, for $0 < r < R$, a set $\rmA \subset \bbZ^2$ will be called $(r,R)$-\textit{clustered} if 
\begin{equation}
\label{e:2.9h}
	\log \|x-y\| \notin (r,R) \quad ; \qquad \forall x,y \in \rmA \,.
\end{equation}
A \textit{minimal $r$-cover} of $\rmA$ is any element $\bfA$ from
\begin{equation}
\label{e:2.10t}
	\Argmin \Big\{ |\bfA| :\: \rmA \subseteq \bigcup_{z \in \bfA} \rmB(z;r)\,,\,\, \bfA \subseteq \bbZ^2  \Big\} \,.
\end{equation}
While it is not unique, unless otherwise stated, all results in the sequel will hold uniformly with respect to the choice of such cover. Moreover, it is not difficult to see that there always exists a choice of $\bfA$ such that $\bfA \subseteq \rmA$, and we shall often assume that our cover is as such. 

The \textit{$r$-cluster process} associated with such $\rmA$, denoted $\bfXi_{\rmA,r}$, is a point measure encoding the points of a minimal $r$-cover of $\rmA$ after scaling, i.e.,
\begin{equation}
\label{e:2.9n}
	\bfXi_{\rmA,r} = \sum_{z \in \bfA} \delta_{\rme^{-n} z}.
\end{equation}
It inherits its non-uniqueness from the non-uniqueness of $\bfA$.

For what follows, we define the {\em clustering scale length} as 
\begin{equation}
\label{e:2.3}
	r_n = r_n(\eta_0) := n^{1/2-\eta_0} \,,
\end{equation} 
for some $\eta_0 \in (0,1/2)$ (not to be confused with $\eta_{n,r}$ from \eqref{eq:defempgff}), whose precise value is immaterial for the validity of the results to follow. Lastly, we define the bulk of $\rmD_n$ as
\begin{equation}
	 \rmD_n^\circ:= \{ x \in \bbZ^2 : \rmd(x,\rmD_n^\rmc) > n -2\log n\},
	 \end{equation}
where, as before, $\rmd(\cdot,\cdot)$ denotes the point-to-set distance with respect to the Euclidean norm.

Having given all these definitions, we can now state the result which describes the statistics of the vacant set at the end of phase A.
\begin{phasethm}[Phase A]
\label{t:2.1}
For each $n \geq 1$ there exists an $\cF_{t_n^A}$-measurable, $(r_n, n-r_n)$-clustered random subset $\wt{\rmF}_{n,t_n^A}(0) \subseteq \rmD_n^\circ$ such that
\begin{equation}
\label{e:303.21}
\bigg( \frac{1}{\sqrt{n}}\Big|\rmF_{n, t_n^A}(0) \,\Delta\, \wt{\rmF}_{n,t_n^A}(0) \Big| \,,\,\,
	\frac{1}{\sqrt{n}} \bfXi_{\wt{\rmF}_{n,t_n^A}(0),r_n} \,,\,
	\frac{\bfT_{n,t_n^A}/|\rmD_n| - t_n^A}{2\sqrt{t_n^A}} \bigg)
	\,\underset{n \to \infty} \Longrightarrow \, \big(0 ,\, C_\star \cZ_\rmD ,\,\ol{h}_\rmD \big) \,,
\end{equation}
where $\cZ_\rmD$ is as in~\eqref{eq:deflqgm}, $\ol{h}_\rmD$ is as in~\eqref{e:1.7p} and $C_\star \in (0,\infty)$ is a constant.
\end{phasethm}

A general outline of the proof of Theorem~\ref{t:2.1} will be given in Section~\ref{sec:phaseA} below, and the full proof will then be carried out in the subsequent sections.

\subsubsection{Phase B}

Next we address the set of unvisited sites during phase B. Here, we show that the cover time and last visited vertex of these remaining sites admit a joint scaling limit. Henceforth, in alignment with the notation before, given $\rmA \subset \rmD_n$, we let $\wc{T}_A$ and $\wc{X}_A$ respectively denote the cover $\partial$-time of $\rmA$ and the last visited vertex thereof. 
\begin{phasethm}[Phase B]
\label{t:2b}
Uniformly in $s \in \bbR$ on compacts, $(r_n, n-r_n)$-clustered sets $\rmA \subset \rmD^\circ_n$ and subsets $\rmW \subseteq \rmD$,
\begin{equation}
\label{e:2.15m}
\bigg| \bbP \big( \wc{T}_\rmA \leq t_n^B + sn\,,\,\, \wc{X}_\rmA \in \rme^{n} \rmW \big) -
	\exp \Big(-\rme^{-(s-\log \frac{1}{\sqrt{n}} \|\bfXi_{\rmA, r_ n}\|)}\Big)\bfXi^\circ_{\rmA, r_n}(\rmW) \bigg| \underset{n \to \infty} \longrightarrow 0 \,.
\end{equation}
Moreover, for any fixed $s \in \bbR$, if $n$ is large enough, then, uniformly in all $\rmA \subset \rmD_n$ (not necessarily clustered),
\begin{equation}
\label{e:2.10a}
\bbP \big( \wc{T}_{\rmA} > t_n^{B} + sn \big) \leq 2 \rme^{-s}\frac{|\rmA|}{\sqrt{n}} \,.
\end{equation}
Lastly, for any fixed $s \in \bbR$,
\begin{equation}
	\frac{\bfT_{n,t_n^B+sn}/|\rmD_n| - (t_n^B+s n)}{n} \overset{\bbP}{\underset{n \to \infty}\longrightarrow} 0 \,.
\end{equation}
\end{phasethm}
The proof of Theorem~\ref{t:2b} is given in Section~\ref{sec:proofb}.

\subsection{Proof of the main theorem}
Combining Theorem~\ref{t:2.1} and Theorem~\ref{t:2b}, the proof of Theorems~\ref{t:1.1} and~\ref{t:1.3i} are fairly straightforward.

\begin{proof}[Proof of Theorems~\ref{t:1.1} and~\ref{t:1.3i}]
Let us define the quantities 
\begin{equation}
\wc\bfT_n^B := (\wc{\bfT}_n - \bfT_{n,t_n^A}) \vee 0 \qquad \text{ , }\qquad \wc{T}_n^B := (\wc{T}_n - t_n^A) \vee 0.
\end{equation} If $\rmD_n$ has not been covered by the end of phase A, then these quantities respectively correspond to the real and $\partial$-time until the full domain is covered. Accordingly, let also 
\begin{equation}
\wc{\bfX}_n^B \equiv \wc{X}_n^B:=\mathbf{X}_{\wc{\bfT}_n \vee \bfT_{n,t^A_n}} 	
\end{equation}
 be the last visited vertex after phase A (until coverage) in any time parametrization. Note~that, on the event that $\wc{\bfT}^B_n>0$ (that is, the domain is not fully covered yet by the end of phase A), we have $\wc{\bfX}_n^B=\wc{\bfX}_n$. Next, abbreviate:
\begin{equation}
 \ol{\bfT}_n^A \equiv \frac{\bfT_{n,t_n^A}}{|\rmD_n|} 
 \quad, \qquad 
	\ol{\bfT}_n^B \equiv \frac{\wc{\bfT}_n^B}{|\rmD_n|}
\quad, \qquad 
	\ol{\bfT}_n \equiv \frac{\wc{\bfT}_n}{|\rmD_n|}
	\quad , \qquad 
		\wh{\bfT}_{n,t} \equiv	\frac{\bfT_{n,t}/|\rmD_n| - t}{2\sqrt{t}} \,,
\end{equation}
and also,
\begin{equation}
	\ol{\bfXi}_n \equiv \frac{1}{\sqrt{n}} \bfXi_{\wt{\rmF}_{n,t_n^A}(0),r_n}
\quad, \qquad
	\ol{\bfX}_n^B \equiv \frac{1}{\rme^{n}} \wc{\bfX}_n^B \quad, \qquad
		\ol{\bfX}_n \equiv \frac{1}{\rme^{n}} \wc{\bfX}_n \,.
\end{equation}

We first claim that uniformly in $s \in \bbR$ on compacts and subsets $\rmW \subseteq \rmD$,
\begin{equation}
\label{e:1103.30}
\bigg|\bbP \big(\ol{\bfT}_n^B \leq t_n^B + sn \,,\,\, \ol{\bfX}_n^B \in W 
\big|\, \cF_{t_n^A}\big) - \exp \big( -\rme^{-(s - \log \|\ol{\bfXi}_n\|)}\big) \ol{\bfXi}^\circ_n(W)\bigg|
	\underset{n \to \infty}{\overset{\bbP}\longrightarrow} 0\,.
\end{equation}
Indeed, we first observe that, in view of the last part of Theorem~\ref{t:2b} and the continuity in $s$ of the exponent in~\eqref{e:1103.30}, it is enough to show~\eqref{e:1103.30} with the probability therein replaced by
\begin{equation}
	\bbP \Big(\wc{T}^B_{\rmF_{n,t_n^A}(0)} \leq t_n^B + sn \,,\,\, \wc{X}^B_{\rmF_{n,t_n^A}(0)} \in \rme^n \rmW \,\Big|\, \cF_{t_n^A}\Big) \,,
\end{equation}
Replacing $\rmF_{n,t_n^A}(0)$ by $\wt{\rmF}_{n,t_n^A}(0)$ changes the last probability by at most
\begin{equation}
\label{e:2.22r}
	2\bbP \Big(\wc{T}^B_{\rmF_{n,t_n^A}(0) \Delta \wt{\rmF}_{n,t_n^A}(0)}
> \wc{T}^B_{\rmF_{n,t_n^A}(0) \cap \wt{\rmF}_{n,t_n^A}(0)} \,\Big|\, \cF_{t_n^A}\Big) \,,
\end{equation}
which tends to zero in probability thanks to the convergence of the first two components on the left-hand side of \eqref{e:303.21} and the first two assertions in Theorem~\ref{t:2b}. The random variables in~\eqref{e:2.22r} are respectively the cover $\partial$-time of the sets $\rmF_{n,t_n^A}(0) \Delta \wt{\rmF}_{n,t_n^A}(0)$ and $\rmF_{n,t_n^A}(0) \cap \wt{\rmF}_{n,t_n^A}(0)$ by the walk viewed in phase B only.
But then the first assertion of Theorem~\ref{t:2b} gives~\eqref{e:1103.30}.

Now, on the one hand, for any $s \in \bbR$ and subset $\rmW \subseteq \rmD$ we have
\begin{equation}
\big|\bbP \big(\ol{\bfT}_n \leq \wc{t}_n + sn\,,\,\, \ol{\bfX}_n \in \rmW \big) 
 - \bbP \big(\ol{\bfT}_n^B \leq \wc{t}_n - \ol{\bfT}_n^A + sn\,,\,\, \ol{\bfX}_n \in \rmW\big)\big| \leq \bbP\big( \overline{\bfT}_n \leq \overline{\bfT}^A_n\big) 
\end{equation} which tends to $0$ as $n \to \infty$ by the convergence of the second component in the left-hand side of \eqref{e:303.21} to a non-degenerate measure. On the other hand, if $\rmW$ is stochastically continuous w.r.t. Lebesgue, then $\rmW$ is a.s. stochastically continuous w.r.t. $\cZ_\rmD$ as well (see \cite[Exercise~10.10]{biskuppims}) and therefore, by using~\eqref{e:1103.30}, we obtain
\begin{equation}
\begin{split}
	\bbP \big(\ol{\bfT}_n^B \leq \wc{t}_n - \ol{\bfT}_n^A + sn\,,\,\, \ol{\bfX}_n & \in \rmW \big)
	= \bbP \Big(\ol{\bfT}_n^B \leq t_n^B - 2\sqrt{t_n^A}\; \wh{\bfT}_{n,t_n^A} -\tfrac{5}{32}\log^2 n + sn\,,\,\, \ol{\bfX}_n \in \rmW \big)\\
	& = \bbE \bbP \Big(\ol{\bfT}_n^B \leq t_n^B - 2\sqrt{t_n^A}\, \wh{\bfT}_{n,t_n^A} -\tfrac{5}{32}\log^2 n+ sn\,,\,\, \ol{\bfX}_n \in \rmW \,\Big|\, \cF_{t_n^A} \Big) \\
	& = \bbE \Big(\exp \Big( -\rme^{-(s - (2\sqrt{2}+o(1)) \wh{\bfT}_{n,t_n^A} - \log \|\ol{\bfXi}_n\|)} \Big)\ol{\bfXi}_n^\circ(\rmW)\Big) + o(1) \\
	& = \bbE \Big(\exp \Big( -C_\star \rme^{-(s - 2\sqrt{2}\, \ol{h}_\rmD - \log \|\cZ_\rmD\|)}\Big) \cZ^\circ_\rmD(\rmW) \Big) + o(1) \,,
\end{split}
\end{equation}
where $o(1) \to 0$ as $n \to \infty$, and the last equality above follows by the convergence of the last two components in the left-hand side of \eqref{e:303.21}. Note that above we have also used the tightness of $\wh{\bfT}_{n,t_n^A}$ together with the uniform convergence in~\eqref{e:1103.30} to obtain the one before last equality. 

In view of~\eqref{e:2.4n} and recalling that we have scaled-down the random walk rates by a factor of $2\pi$, this gives~\eqref{eq:alternativeformulation} and thus completes the proof for both theorems.
\end{proof}

\section{Phase A: Top-level proof of Theorem~\ref{t:2.1}}
\label{sec:phaseA}
 In this section we prove Theorem~\ref{t:2.1}. As in the case of the overall result, the exposition will be top-down, with the key steps outlined as propositions which are first used to construct a top-level proof of the theorem and only proved after, in what constitutes most of the remainder of this manuscript.
 
The proof of the theorem, and the reason for the definition of a Phase A, builds on the idea in~\cite{Tightness}, whereby the law of the set of unvisited vertices $\rmF_{n,t_n^A}(0)$ can be compared to the law of the min-extreme values of the DGFF via the so-called Isomorphism Theorem. This makes it possible to use recently derived estimates and asymptotics for the extreme values of the DGFF to deduce analogous results in the ``local time world'' and, ultimately, to prove Theorem~\ref{t:2.1}.

We begin by recalling the Isomorphism Theorem, also known as the Generalized
 Second Ray-Knight Theorem, due to~\cite{Gang}. 
\begin{thm}[Generalized second Ray-Knight Theorem]
\label{t:103.1}
Given $t \geq 0$ and $n \geq 0$, there exists a coupling of $L_{t} = (L_{t}(x) : x \in \rmD_n)$ with two copies of the DGFF on $\rmD_n$, $h_n = (h_n(x) : x \in \rmD_n)$
and $h'_n = (h'_n(x) : x \in \rmD_n)$, such that $L_{t}$ and $h_n$ are independent of each other and almost-surely,
\begin{equation}
\label{e:3.1n}
L_{t}(x) + h^2_n(x) = 
\big(h'_n(x) + \sqrt{t})^2 
\quad , \ x \in \rmD_n \,.
\end{equation}
\end{thm}
We recall that the DGFFs above are related to the DGFF presented in the introduction by a multiplicative factor of $\sqrt{2\pi}$. This scaling is due to the retuning of the random walk rates, and the absorption of a $\frac12$ factor that usually multiplies $h_n$ and $h'_n$ in~\eqref{e:3.1n}.
We shall use the coupling in the theorem at time $t = t_n^{A}$, in which case we shall write the right-hand side of~\eqref{e:3.1n} as $f^2_n(x)$, where
\begin{equation}
f_n(x) := h'_n(x) + m_n 
\quad , \qquad m_n := \sqrt{t_n^{A}} \,.
\end{equation}
Uncoincidentally, $-m_n$ is precisely the typical height of the DGFF's min-extreme values, so that the $O(1)$-level sets of $f_n$ correspond to the min-extremal level sets of the DGFF. Accordingly, for $u \geq 0$, we define the (sub-)level set of the field $f^2_n$ at height $u$ by
\begin{equation}
\label{e:1103.3o}
\rmS_n(u) := \big \{ x \in \rmD^\circ_n :\: f^2_n(x) \leq u \big \} \,.
\end{equation}
We observe that, if we define 
\begin{equation}
\rmW_n(u):=\{ x \in \rmD^\circ_n : L_t(x) \leq u\}\,,	
\end{equation}
then, under the coupling in Theorem~\ref{t:103.1}, we have
\begin{equation}
\label{e:3.3}
\rmS_n(u) = \big\{ x\in \rmW_n(u) :\: h^2_n(x) \leq u - L_t(x) \big\} \subseteq \rmW_n(u) \,,
\end{equation}since the second term on the left-hand side of~\eqref{e:3.1n} is always positive. In view of this, we shall often say that a vertex $x \in \rmW_n(u)$ ``survived the isomorphism'' if $x$ is also in $\rmS_n(u)$. Clearly, this happens whenever $h_n(x) = O(1)$, which occurs with probability $\Theta(1/\sqrt{n})$ since the latter is a centered Gaussian with variance $\Theta(n)$. We see that $\rmS_n(u)$ is a (random) thinning of $\rmW_n(u)$ by a factor $\Theta(1/\sqrt{n})$. The idea is to use knowledge about the limiting law of $\rmS_n(u)$ and the well-known random properties of $h_n$ to study the asymptotics of the law of $\rmW_n(u)$, and then of the law of the vacant set $\rmF_{n,t_n^A}(0)$.

To make this idea work, several steps need to be taken. First, one has to treat the global and local scales separately. As usual in many log-correlated fields, non-trivial correlations exist at all scales, and give rise to different asymptotic stochastic behavior at different scales. In our case, these multi-scale correlations result in the set $\rmS_n(u)$ being clustered, in the sense of~\eqref{e:2.9h}, and in the field $h_n$ having strong local correlations but almost no correlations at a macroscopic scale. 

To take care of the global scale, one works at the level of clusters. Formally, for $r \geq 0$, we let $\bbX_r := \lfloor \rme^r\rfloor \bbZ^2$ be the $r$-scaled lattice, so that the $r$-boxes $(\rmQ(z;r) : z \in \bbX_r)$ form a tiling of $\bbZ^2$. Recalling $r_n$ from~\eqref{e:2.3}, we then define
\begin{equation}
	\bfW_n(u) := \Big \{ z \in \bbX_{n-r_n}  :\:  \rmW_n(u)  \cap \rmQ(z ; n-r_n) \neq \emptyset \Big \} \,.
\end{equation}
This is the set of $n-r_n$ boxes which contain $O(1)$-local-time vertices. We shall occasionally refer to $\rmW_n(u)  \cap \rmQ(z ; n-r_n)$ for $z \in \bfW_n(u)$ as the cluster of $O(1)$-local-time vertices centered at $z$. The analog quantity for the DGFF is \
\begin{equation}
	\bfS_n(u) := \Big \{ z \in \bbX_{n-r_n}  :\:  \rmS_n(u)  \cap \rmQ(z ; n-r_n) \neq \emptyset \Big \} \,.
\end{equation}
Using the well-known upper tightness of $|\bfS_n(u)|$, one can then fairly easily get upper tightness of the quantity $|\bfW_n(u)|/\sqrt{n}$ (see \cite[Lemma~4.2]{Tightness}). To show lower tightness (i.e. asymptotic almost-sure strict positivity), however, it is not enough to use the known non-triviality of $\rmS_n(u)$ in the limit, and one has to address the local structure of clusters as well.

To this end, for $k \in [0, n-r_n]$ we further let
\begin{equation}
\label{e:3.30.1}
\begin{split}
	\bfW^{k}_n(u) := \Big \{ z \in \bfW_n(u) :\: & \exists! x=x(z)\in \bbX_k \ \text{ s.t. }
	\\
	& \rmQ(x;k) \subset \rmQ(z;n-r_n) \,,\,\,
	\rmW_n(u) \cap \rmQ(x;k) \neq \emptyset 
	\Big\} \,.
\end{split}
\end{equation}
This is (essentially) the collection of center points of boxes containing clusters whose $\ell^\infty$-diameter is at most $\rme^k$. 
To rule out vertices which are not sufficiently clustered, we define for $K \subset \bbR_+$,
\begin{equation}
\label{e:3.7n}
	\rmW_n^K(u) := \Big\{ x \in \rmW_n(u) :\: \exists y \in \rmW_n(u) \ \text{ s.t. } \log |x-y| \in K \Big\} \,.
\end{equation}

The following result was shown (in a slightly different but still equivalent formulation)  in~\cite{Tightness} and constitutes one of the main efforts in that work.
\begin{thm}[Theorem 5.1 and Theorem 6.1 in~\cite{Tightness}]
\label{p:300.3}
Let $u \geq 0$.
\begin{equation}
\label{e:5.3a}
\lim_{k \to \infty}
\limsup_{n \to \infty} \bbP \big( \bfG_n(u) \setminus 
\bfW_n^k(u) \neq \emptyset \big) = 0 \,.
\end{equation}
Morevoer,
\begin{equation}
\label{e:5.3}
\frac{ \big| \bfW_n(u) \setminus 
\bfW_n^k(u) \big|}{\sqrt{n}} \overset{\bbP}\longrightarrow 0 \,,
\end{equation}
as $n \to \infty$ followed by $k \to \infty$, and 
\begin{equation}
\label{e:3.10n}
\frac{\big|\rmW_n^{[r_n, n-r_n]}(u)  \big|}{\sqrt{n}}
	\overset{\bbP}{\underset{n \to \infty}\longrightarrow} 0 \,.
\end{equation}
as $n \to \infty$.
\end{thm}
\noindent
 Thus, most of the clusters in $\bfW_n(u)$ are of $O(1)$ diameter and that those which are larger, do not survive the isomorphism. Furthermore, outside a set of size $o(\sqrt{n})$ in probability, all vertices are $(r_n, n-r_n)$-clustered.
Using the above theorem and the non-vanishing of $\bfS_n(u)$ in the limit, the analysis of Phase A in~\cite{Tightness} culminated in the full tightness of $|\bfW_n(0)|/\sqrt{n}$, namely,
\begin{thm}[Theorem~2.3 in~\cite{Tightness}]
\label{t:2.1o}
\begin{equation}
\lim_{\delta \downarrow 0} \limsup_{n \to \infty}
\bbP \Big(\frac{1}{\sqrt{n}} \big|\bfW_n(0)\big| \notin \big(\delta, \delta^{-1} \big) \Big) = 0 \,.
\end{equation}
\end{thm}
\noindent
Taking $\wt{\rmF}_{n,t_n^A}(0)= \rmW_n(0) \setminus \rmW^{[r_n, n-r_n]}_n(0)$ gives the necessary tightness underlying the convergence of the first two components in the desired weak limit~\eqref{e:303.21}.

To go beyond tightness of the normalized number of low local time clusters, several additional steps need to be taken. First, as we also want to keep track of the location of these clusters, we naturally define 
\begin{equation}
\label{e:3.13l}
\bfXi_{n,u} :=
	\sum_{z \in \bfW_n(u)}
	\delta_{\rme^{-n}z} 
\qquad ; \qquad \qquad 
\bfXi^k_{n,u} :=
	\sum_{z \in \bfW_n^k(u)}
	\delta_{\rme^{-n}z} \,. 
\end{equation}
The analog quantity for the DGFF on the right-hand side of the isomorphism is
\begin{equation}
	\bfGamma_{n,u} := \sum_{z \in \bfS_n(u)} \delta_{\rme^{-n} z} \,.
\end{equation}
We shall refer to these two point processes respectively as the cluster process of low local time vertices and that of the DGFF min-extremes.

Second, we need to utilize the full weak convergence of $\bfGamma_{n,u}$ not just the tightness of $|\bfS_n(u)|$. Thanks to recent work, asymptotics for the law of $\bfGamma_{n,u}$ are well known. Indeed, this is a straight-forward consequence of the so-called convergence of the extremal process of the DGFF. This is not enough, however, as in order to control also the conversion between $\partial$-time and real-time, we shall need the joint convergence together with the average of the field $\ol{h}_n$, as defined in~\eqref{e:1.17i}.
We shall thus prove,
\begin{prop}
\label{p:300.4}
If $u > 0$ is sufficiently large, there exist a constant $C_u \in (0,\infty)$ and a coupling of $\cZ_\rmD$ from~\eqref{eq:deflqgm}, $\ol{h}_\rmD$ from~\eqref{e:1.7p} and a random Radon point measure $\bfGamma_{\infty,u}$ on $\rmD$ satisfying that
\begin{equation}
\bfGamma_{\infty,u} \,\big|\, \cZ_\rmD, \ol{h}_\rmD\  \sim\  {\rm PPP}\big(C_u \cZ_\rmD(\rmd x) \big) \,,
\end{equation}
such that 
\begin{equation}
\label{e:3.14n}
\big(\bfGamma_{n,u}\,,\,\, \ol{h}_n \big)
	\,\underset{n \to \infty}\Longrightarrow\, \big(\bfGamma_{\infty,u},\, \ol{h}_\rmD\big) \,.
\end{equation}
\end{prop}
\noindent
This proposition is a straightforward consequence of the joint convergence of the extremal process of the DGFF and its average, namely Theorem~\ref{t:1.2i}, the proof of which is the subject of Section~\ref{s:5n}.

Having identified the limiting law of $\bfGamma_{n,u}$, we wish to derive an asymptotic relation in law between this measure and $\bfXi_{n,u}$ using~\eqref{e:3.3}. Thanks to Theorem~\ref{p:300.3}, we know that ``modulo'' a negligible fraction, the set $\rmW_n(u)$ is $(O(1), n-r_n)$-clustered, in the sense of the definition given below~\eqref{e:2.3}. The covariance structure of the DGFF then entails that $h_n$ has essentially the same law on any such cluster, and that its values on different clusters are effectively independent. Henceforth we shall say that such a field is {\em cluster-i.i.d.} on $\rmW_n(u)$. 

As the value of $h_n$ at a given vertex is Gaussian with mean zero and variance $n + O(1)$ and the values at nearby vertices are $O(1)$ apart, it follows that $h_n$ acts in~\eqref{e:3.3} on the clusters of $\rmW_n(u)$ by thinning them according to some i.i.d. law with a ``survival rate'' of $\Theta(1/\sqrt{n})$.
While $h_n$ is essentially cluster-i.i.d. on $\rmW_n(u)$, it is not clear a-priori that 
the same holds for $L_{t_n^A}$, namely that the the local time field after $\partial$-time $t_n^A$ has asymptotically the same joint law on each of the clusters of $\rmW_n(u)$, and that values at different clusters are independent of each other. This cluster-i.i.d.-ness for $L_{t_n^A}$ is essential for having $\bfGamma_{n,u}$ related to $\bfXi_{n,u}$ via an i.i.d. thinning, and in turn for relating the laws of these two random measures.

 In Section~\ref{sec:iid} we show that this is indeed the case. As an LLN-type-consequence of this i.i.d. structure, we get the existence of a limiting law for the empirical distribution of the local time field on the clusters. The latter is defined as
\begin{equation}
\label{e:300.312n}
\bfLambda^k_{n,u} :=	
	\sum_{z \in \bfW_n^k(u)}
	\delta_{\rme^{-n}z} \otimes \delta_{L_{t_n^A}(\rmQ(x(z);k))} 
\end{equation}
This is a (random) probability measure on $\rmD \times \bbR^{\rmQ(k)}$ whose 
$\rmD$-marginal is $\bfXi^k_{n,u}$.

The following shows that $\bfLambda^k_{n,u}$ is asymptotically, with high probability, a product measure with a deterministic second component. We recall that $\mu^\circ$ denotes its normalized version of the measure $\mu$.
\begin{prop}
\label{p:300.1a}
Fix $u \geq 0$. For all $k \geq 0$, there exists a probability measure $\rmP^k_u$ on $\bbR^{\rmQ(k)}$ such that, uniformly in all measurable functions $f: \rmD \times \bbR^{\rmQ(k)} \to \bbR$ with $\|f\|_\infty \leq 1$,
\begin{equation}
\frac{1}{\sqrt{n}} \Big|\int f(x,\omega) \bfLambda^{k}_{n,u}(\rmd x \times \rmd \omega) -
	\int f(x,\omega) \, \big(\bfXi^{k}_{n,u}(\rmd x) \otimes \rmP^k_u (\rmd \omega) \big)\Big|
 \underset{n \to \infty}{\overset{\bbP} \longrightarrow} 0 \,.
\end{equation}
\end{prop}
The proof of proposition~\ref{p:300.1a}, which takes up a substantial part of this manuscript, is the subject of Sections~\ref{sec:iid} and~\ref{s:barrier}.

The i.i.d. thinning of the clusters of $\bfW_n(u)$ can now be stated as follows. Here and after, we define the intersection of two counting measures as the counting measure whose support is the intersection of their supports.
\begin{prop}
\label{p:300.2}
Fix $u \geq 0$. For all $k \geq 0$, there exists a bounded function $\varphi^{k}_u :
\bbR^{\rmQ(k)} \to \bbR_+$ such that, for any measurable positive function $f: \rmD \to \bbR_+$ with compact support,
%\Oren{Check that we use compact support, as $\bfW_n$ is defined in terms of $\rmD_n^\circ$.}
\begin{equation}
\label{e:300.3.16}
\begin{split}
\lim_{n \to \infty}
\bigg|\bbE \bigg( \xi_n \exp \bigg\{ -\int f(x) & \big(\bfXi_{n,u}^k \cap \bfGamma_{n,u}\big)
(\rmd x)\bigg\} \bigg) \\
& - \bbE \bigg( \xi_n \exp \bigg\{-\frac{1}{\sqrt{n}} \int \varphi^k_u(\omega) \big(1-\rme^{-f(x)}\big)\bfLambda_{n,u}^{k}(\rmd x \times \rmd \omega)
  \bigg\}\bigg) \bigg| = 0 \,,
\end{split}
\end{equation}
uniformly over all $cF_{t^A_n}$-measurable random variables $\xi_n$ with $\rmL^\infty$-norm at most $1$.
\end{prop}
\noindent
The inclusion of the extra random variable $\xi_n$ in the statement will be useful for treating the conversion between $\partial$-time and real time. Proposition~\ref{p:300.2} is proved in Section~\ref{sec:conseq}.

Assuming the convergence of the integral, which will be a consequence of Proposition~\ref{p:300.1a}, whenever $\xi_n=1$ 
the right-hand side of~\eqref{e:300.3.16} can be identified as the Laplace Transform of a Poisson Point Process  with a random intensity given by a constant times $\frac{1}{\sqrt{n}}\bfXi_{n,u}^k$. As this is asymptotically the Laplace Transform of the cluster process of DGFF min-extremes, we get via Proposition~\ref{p:300.4} an asymptotic equivalence between the laws of $\frac{1}{\sqrt{n}}\bfXi_{n,u}^k$ and a constant multiple of $\cZ_\rmD$. This is almost the desired convergence of the second component in~\eqref{e:303.21}, except that we need to replace $u > 0$ with $u = 0$, as we are after the vacant set. This is the purpose of the next proposition, whose proof, which is a straightforward application of Proposition \ref{p:300.1a}, is included in Section~\ref{sec:iid}.
\begin{prop}
\label{p:300.5}
For each $u \geq 0$, there exists a constant $C'_u \in (0,1]$ such that, uniformly over all measurable $f: \rmD \to \bbR$ with $\|f\|_\infty \leq 1$,
\begin{equation}
\label{e:300.3.18}
\frac{1}{\sqrt{n}}\Big|\int f(x) \bfXi_{n,0}(\rmd x) - C'_u \int f(x) \bfXi_{n,u}(\rmd x)\Big|  \underset{n \to \infty}{\overset{\bbP}{\longrightarrow}} 0 \,.
\end{equation}
\end{prop}
\noindent

To close the argument for the convergence of the first two components in~\eqref{e:303.21}, we need two more steps.
First, we need to show that $\bfXi_{n,0}$ is asymptotically equivalent (in the Vague metric) to $\bfXi_{\rmW_n(0)}$, as defined in~\eqref{e:2.9n}. This may fail if there are clusters which fall on the boundary of a box $\rmQ(z;n-r_n)$. To show that the fraction of such clusters is negligible, for $u \geq 0$, $n \geq k \geq 1$ and $\delta > 0$, we let
\begin{equation}
\label{e:3.19n}
	\bfW^k_{n,\delta}(u) := \Big \{ z \in \bfW^k_n(u)  :\:   \rmW_n(u) \cap \rmQ(z ; n-r_n) \subseteq \rmQ(z ; n-r_n-\delta) \Big \} \,.
\end{equation}
%with $\bfW^k_{n,\delta}(u) := \bfW_{n,\delta}(u) \cap \bfW^k_{n}(u)$.
Then,
\begin{prop}
\label{p:3.8nn}
For all $u \geq 0$, as $n \to \infty$ followed by $k \to \infty$ and finally $\delta \to 0$,
\begin{equation}
	\frac{|\bfW_{n}(u) \setminus \bfW^k_{n,\delta}(u)|}{\sqrt{n}} \overset{\bbP} \longrightarrow 0 \,.
\end{equation}
\end{prop}
\noindent
This result is implied by the stronger Proposition~\ref{prop:rest3} below and is a consequence of Lemma~\ref{l:1103.7}, which shows more generally that any subset of the domain of vanishing proportion can be ignored in the count of clusters of low local time vertices. 

Second, as we are avoiding low local time vertices near the boundary, we need to show that their number is negligible as well.
\begin{prop}[Lemma 7.5 in~\cite{Tightness}]
	\label{p:3.9nn}
	\begin{equation}
		\label{e:7.21}
		\lim_{n \to \infty} \bbP \big(\rmF_{n,t_n^A}(0) \setminus \rmD_n^\circ \neq \emptyset \big) = 0 \,.
	\end{equation}
\end{prop}

It only remains to get the simultaneous convergence of the running time fluctuations, namely the third component in~\eqref{e:303.21}.
Using the Isomorphism Theorem again, the key idea here is that the scaled and centered fluctuations in the running time are asymptotically equivalent (in probability) to the average of the DGFF $h_n'$ in the isomorphism. Formally, setting
\begin{equation}
\label{e:3.23n}
	\wh{\bfT}_{n,t} := 	\frac{\bfT_{n,t}/|\rmD_n| - t}{2\sqrt{t}} \,,
\end{equation}
and letting $\ol{h}_n'$ denote the average of $h'_n$ over all its vertices as in~\eqref{e:1.17i}for $h_n$, we have
\begin{prop}
\label{p:1103.6}
If $t_n \to \infty$ but $\frac{\sqrt{t_n}}{|\rmD_n|} \to 0$ 
as $n \to \infty$, then, under the coupling of Theorem~\ref{t:103.1}, 
\begin{equation}
\label{e:1103.19}
	\wh{\bfT}_{n,t_n} -	\ol{h}'_n
	\overset{\bbP}{\underset{n \to \infty}
	\longrightarrow} 0 \,.
\end{equation}	
\end{prop}
\noindent
Thus, to include $\wh{\bfT}_{n,t}$ in the convergence statement of the theorem, we need to keep track of $h'_n$ in the argument. This explains its inclusion in the limiting statement of Proposition~\ref{p:300.4} as well as the addition of the extra variable $\xi_n$ in Proposition~\ref{p:300.2}. The proof of Proposition~\ref{p:1103.6} can be found in Section~\ref{sec:conseq}.

With all the ingredients ready, we can now easily present the top-level proof of Theorem~\ref{t:2.1}.
\begin{proof}[Proof of Theorem~\ref{t:2.1}]
Let $f : \rmD \to \bbR_+$ be a compactly supported non-negative continuous~function and $\mu \in \bbR$.  
 Upon combining Proposition~\ref{p:300.1a} and Proposition~\ref{p:300.2} with $\xi_n = \rme^{i\mu \wh{\bfT}_{n,t_n^A}}$, we see that the difference
\begin{multline}
\label{e:3.27m}
\Big|\bbE \exp \Big(-\int f(x) (\bfXi_{n,u}^k \cap
\bfGamma_{n,u})(\rmd x) + i \mu \wh{\bfT}_{n,t_n^A}  \Big) \\
 - \bbE \exp \Big(-\tfrac{1}{\sqrt{n}}\int \varphi^k_u(\omega) \rmP^k_u(\rmd \omega) \int \big(1-\rme^{-f(x)}\big) \bfXi_{n,u}^k (\rmd x)
 + i \mu \wh{\bfT}_{n,t_n^A} \Big) \Big|  \,,
\end{multline}
tends to $0$ as $n \to \infty$. Thanks to Theorem~\ref{p:300.3}, Proposition~\ref{p:300.5} and Proposition~\ref{p:1103.6} and since the exponents are bounded from above, we may drop the intersection with $\bfXi_{n,u}^k$ and 
replace $\wh{\bfT}_{n,t_n^A}$ by $\ol{h}'_n$ in the first exponent, and also replace $\bfXi_{n,u}^k$ by $(C_u')^{-1} \bfXi_{n,0}$ in the second exponent with the difference still going to $0$, provided that now we also take $k \to \infty$ after $n \to \infty$. 

But then, by Proposition~\ref{p:300.4}, if we take $u$ sufficiently large the first expectation 
in~\eqref{e:3.27m} tends as $n \to \infty$ to
\begin{equation}
\bbE \exp \Big(- \int \big(1-\rme^{-f(x)}\big) C_u \cZ_\rmD(\rmd x) + i\mu \ol{h}_\rmD  \Big) 
\end{equation}
As the convergence of joint Laplace-Fourier Transform implies joint weak convergence (c.f.~\cite{Fitz2010}), we thus get
\begin{equation}
\label{e:1103.25}
\Big((C_u C_u')^{-1}
\int \varphi^k_u(\omega) \rmP^k_{u}(\rmd \omega)
\frac{1}{\sqrt{n}} \bfXi_{n,0} \,,\,\, \wh{\bfT}_{n,t_n^A} \Big)
\Longrightarrow
\big(\cZ_\rmD, \ol{h}_\rmD \big) \,,
\end{equation}
as $n \to \infty$ followed by $k \to \infty$. 

Since the first coordinate on the left-hand side is the product of two sequences which are indexed by different parameters, the integral must converge. Setting
\begin{equation}
	C_\star := C_u C'_u \bigg(\lim_{k \to \infty} \int \varphi^{k}_u(
	\omega) \rmP^{k}_{u}(\rmd \omega)\bigg)^{-1} \,,
\end{equation}
we thus get
\begin{equation}
\label{e:3.30n}
	\Big(\frac{1}{\sqrt{n}}\bfXi_{n,0} \,,\,\, \wh{\bfT}_{n,t_n^A}\Big)
	\underset{n \to \infty}\Longrightarrow	\big(C_\star \cZ_\rmD(D) ,\,\ol{h}_\rmD \big) \,.
\end{equation}

To obtain the desired convergence in the theorem, we define
\begin{equation}
\label{e:3.32nn}
	\wt{\rmF}_{n,t_n^A}(0) := \rmW_n(0) \setminus \rmW_n^{[r_n, n-r_n]}(0) \,.
\end{equation}
By definition, $\wt{\rmF}_{n,t_n^A}(0)$ is a $(r_n, n-r_n)$-clustered and, in addition, by the third part of Theorem~\ref{p:300.3} and Proposition~\ref{p:3.9nn}, we get the convergence of the first component in~\eqref{e:303.21}. In view of~\eqref{e:3.30n}, it thus remains to show that  
\begin{equation}
\label{e:3.33n}
\rmd \Big(\frac{1}{\sqrt{n}} \bfXi_{\wt{\rmF}_{n,t_n^A}(0), r_n},\,\frac{1}{\sqrt{n}} \bfXi_{n,0}\Big) \overset{\bbP}{\underset{n \to \infty}\longrightarrow} 0 \,,
\end{equation}
under the Vague metric.

To this end, let $0 < \eta' < \eta_0 < \eta'' < 1/2$, and let $r_n'$, $r_n''$ be defined as $r_n$ from~\eqref{e:2.3} only with $\eta'$, $\eta''$ respectively in place of $\eta_0$. Define also $\wt{\rmF}_{n,t_n^A}'(0)$, $\wt{\rmF}_{n,t_n^A}''(0)$ as in~\eqref{e:3.32nn} only with $r_n'$, $r_n''$ respectively.
Clearly, $\wt{\rmF}''_{n,t_n^A}(0) \subseteq \wt{\rmF}_{n,t_n^A}(0) \subseteq \wt{\rmF}'_{n,t_n^A}(0)$. It is also not difficult to see that we may take minimal $r_n$, $r_n'$, $r_n''$-covers, $\bfF$, $\bfF'$, $\bfF''$ of $\wt{\rmF}_{n,t_n^A}(0)$, $\wt{\rmF}'_{n,t_n^A}(0)$, $\wt{\rmF}''_{n,t_n^A}(0)$ respectively, such that
\begin{equation}
	\bfF \subseteq \bfF' \subseteq \bfF \cup \rmW_n^{[r_n, n-r_n]}(0) 
	\quad ;\qquad
	\bfF'' \subseteq \bfF \subseteq \bfF \cup \rmW_n^{[r_n'', n-r_n'']}(0). 
\end{equation}
It thus follows by the third part of Theorem~\ref{p:300.3} (which holds for all choice of $\eta$) that 
\begin{equation}
\label{e:3.34n}
\frac{1}{\sqrt{n}}\rmd\Big(\bfXi_{\wt{\rmF}_{n,t_n^A}'(0), r'_n}\, ,\,\, \bfXi_{\wt{\rmF}_{n,t_n^A}(0), r_n}\Big) \overset{\bbP}{\underset{n \to \infty}\longrightarrow} 0 
\quad , \qquad 
\frac{1}{\sqrt{n}}\rmd\Big(\bfXi_{\wt{\rmF}_{n,t_n^A}''(0), r''_n}\, ,\,\, \bfXi_{\wt{\rmF}_{n,t_n^A}(0), r_n}\Big) \overset{\bbP}{\underset{n \to \infty}\longrightarrow} 0 \,. 
\end{equation}
in the Vague distance.

Now fix a positive continuous function $f:\rmD \to \bbR_+$ with compact support. If $\bfF''$ is a minimal $r''_n$-cover of $\wt{\rmF}''_{n,t_n^A}(0)$ (in the sense of the definition above~\eqref{e:2.9h}), then for any $x \in \bfF''$ there is at least one $z \in \bfW^{r_n}_n(0)$ such that $\rmQ(z;n-r_n) \cap \rmB(x;r''_n) \neq \emptyset$. Moreover, if $n$ is large enough then $\rmQ(z;n-r_n) \cap \rmB(y;r''_n)=\emptyset$ for all $y \in \bfF'' \setminus \{x\}$. It follows from the uniform continuity and positivity of $f$ that for any $\epsilon > 0$,
\begin{equation}
	\int f(x) \bfXi_{\wt{\rmF}''_{n,t_n^A}(0), r_n''}(\rmd x) \leq 	\int f(x) \bfXi_{n,0}(\rmd x) + \epsilon |\bfW_n(0)| \,,
\end{equation}
provided $n$ is large enough.

Similarly, let $\bfF'$ be a minimal $r'_n$ cover of $\wt{\rmF}'_{n,t_n^A}(0)$. Then, for any $\delta > 0$, $k > 0$, if $z \in \bfW^k_{n,\delta}(0)$ then there exists a unique $x \in \bfF' $ such that $\rmB(x;r'_n) \cap \rmQ(z;n-r_n) \neq \emptyset$ if $n$ is large enough. Thus for any $\epsilon > 0$, if $n$ is large enough,
\begin{equation}
\int f(x) \bfXi_{n,0,\delta}^k(\rmd x) \leq \int f(x) \bfXi_{\wt{\rmF}'_{n,t_n^A}(0)), r_n'}(\rmd x)  + \epsilon |\bfW_n(0)| \,,
\end{equation}
where $\bfXi^k_{n,u,\delta}$ is defined as $\bfXi^k_{n,u}$ in~\eqref{e:3.13l} only with $\bfW^k_{n,\delta}(u)$ in place of $\bfW^k_{n}(u)$.
Then~\eqref{e:3.33n} follows from the second part of Theorem~\ref{p:300.3} combined with Proposition~\ref{p:3.8nn} and finally also~\eqref{e:3.34n}.
\end{proof}

\section{Consequences of The Isomorphism}\label{sec:conseq}
In this section we show how the Isomorphism Theorem can be used to deduce a relation in law between relevant observables for the local time field and the DGFF.

\subsection{DGFF min-extrema as Poissonization of low local time vertices}

First, we use the isomorphism to prove Proposition~\ref{p:300.2}. To this end, we shall need the following auxiliary lemma.

\begin{lem}
\label{l:300.3.7}
For each $k \geq 1$ there exists a measure $\nu_k$ on $\bbR^{\rmQ(k)}$ with Radon marginals such that, for any $M \in (0,\infty)$,  
\begin{multline}
\label{e:300.3.25}
\lim_{n \to \infty} \sup_{\rmX_{n,k}}
 \sup_{(f_{k,x})_{x \in \rmX_{n,k}}} \bigg| \bbE \exp \Big(-\sum_{x \in \rmX_{n,k}} f_{k,x} \big(h_n(\rmQ(x;k))\big) \Big)  \\
- \exp \Big(- \tfrac{1}{\sqrt{n}} \sum_{x \in \rmX_{n,k}} \int_{\bbR^{\rmQ(k)}} \big(1-\rme^{-f_{k,x}(\sigma)}\big) \nu_k(\rmd \sigma) \Big) \bigg| = 0 \,,
\end{multline} 
where the first supremum is taken over all sets $\rmX_{n,k} \subset \rmD_n$ such that
$\log \|x-x'\| > n-r_n-M$ for all $x \neq x' \in \rmX_{n,k}$ and the second is over all families $(f_{k,x}: x \in \rmX_{n,k})$ of measurable positive functions on $\bbR^{\rmQ(k)}$ satisfying
\begin{equation}
\label{e:300.3.26}
\|f_{k,x}\|_\infty \leq M 
\quad ; \quad
\supp\big(f_{k,x}\big) \subseteq \big\{ \sigma \in \bbR^{\rmQ(k)} :\:
\min_{x \in \rmQ(k)} |\sigma(x)| \leq M \big\} \,.
\end{equation}
\end{lem}
\begin{proof}
Fix $k$ and $M$ as above and note that all constants and asymptotic terms will depend only on these quantities. Let $n \geq k$. For a measurable positive function $f$ on $\bbR^{\rmQ(k)}$ satisfying~\eqref{e:300.3.26} and $\varphi \in \bbR^{\rmQ(k)}$, define
\begin{equation}
\label{e:300.3.27}
	\Psi_{n,f}(\varphi) := \bbE f\big(h_{n}(\rmQ(k)) + \varphi\big)
\end{equation}
where $h_{n}$ is a DGFF on $\rmD_{n}$. 
We first claim that, if $f$ satisfies~\eqref{e:300.3.26}, and $\varphi$ satisfies $|\varphi(0)| \leq n^{1/4}$ and 
	$\sup_{\rmQ(k)} |\varphi-\varphi(0)| \leq n^{-1}$, then as $n \to \infty$,
\begin{equation}
\label{e:300.3.27n}
	\Psi_{n,f}(\varphi) = (1+o(1)) n^{-{1/2}} \int f(\sigma) \nu_k(\rmd \sigma)  \,,
\end{equation}
where $\nu_k$ is the measure
\begin{equation}
	\nu_k(\rmd w) = \frac{1}{\sqrt{\pi}} \int_{s=-\infty}^\infty \bbP \big (h^0(\rmQ(k)) + s \in \rmd w) \rmd s \,,
\end{equation}
and $h^0$ is the pinned DGFF on $\bbZ^2$ at zero tilt (i.e. a Gaussian field on $\bbZ^2$ with mean $0$ and covariance given by the Green Function on $\bbZ^2\setminus \{0\}$ under our usual normalization). In particular, that $\nu_k$ has Radon marginals follows since $h^0(y)$ is a centered Gaussian with bounded variance for all $y \in \rmQ(k)$.

Indeed, by conditioning on $h_n(0)$ and using the Gibbs-Gaussian decomposition for $h_n$ (see \cite[Section~3.2.1]{Tightness}), we may write~\eqref{e:300.3.27} as
\begin{equation}
\label{e:300.3.30}
	\int_{s=-\infty}^\infty \bbP\big(h_n(0) + \varphi(0) \in \rmd s\big)
		\bbE f\Big(h_n^0(\rmQ(k)) + s + (s-\varphi(0))\psi + O(n^{-1})\Big)  \,,
\end{equation}
where $h_n^0$ is the DGFF on $\rmQ(n) \setminus \{0\}$ (with zero boundary conditions), and $\psi :=\frac{1}{h_n(0)} \bbE(h_n - h_n(0) \,|\, h_n(0) )$ satisfies
\begin{equation}
	0 \geq \psi(y) = \frac{G_{n}(y,0) - G_{n}(0,0)}{G_{n}(0,0)} \geq -C \frac{k}{n}
	\quad ; \qquad 
	y \in \rmQ(k) \,
\end{equation} by standard Green's function asymptotics, see for example \cite[Lemma~3.1]{Tightness}.
Now, from the condition on the support of $f$, the restriction on $\varphi(0)$, the above bound on $\psi$, and the union bound, the probability that the quantity inside the expectation 
in \eqref{e:300.3.30} is non zero, is bounded from above by
\begin{equation}
\label{e:4.8n}
	\sum_{y \in \rmQ(k)} \bbP\big(|h_n^0(y)| > |s|/2 - M\big)
	\leq C\rme^{-c s^2} \,,
\end{equation}
with $C,c \in (0,\infty)$, since $h_n^0(y)$ is a centered Gaussian with bounded variance. Since $f$ is also bounded, the last right hand side is also a bound on the expectation in~\eqref{e:300.3.30} itself, up to different $(k,M)$-dependent constants. 

Also, by standard computation for all $v \in \bbR^{\rmQ(k)}$,
\begin{equation}
\lim_{\substack{n \to \infty\\ \epsilon \to 0}}
	\frac{\rmd \bbP(h_n^0|_{\rmQ(k)} + \varepsilon \in \rmd v)}{\rmd v} = \frac{\rmd \bbP(h^0|_{\rmQ(k)} \in \rmd v)}{\rmd v} \,,
\end{equation}
with the density on the left hand side decaying uniformly ``Gaussianly'' for all $n \geq 1$ and $\|\epsilon\|_2 \leq 1$.
Together with the bound on $f$, it follows from the Dominated Convergence Theorem, that for fixed $s$ the expectation in~\eqref{e:300.3.30} tends to
\begin{equation}
	\bbE \Big(f\big(h^0(\rmQ(k)) + s \big) \Big) \,,
\end{equation}
as $n \to \infty$. Using also that $h_n(0)$ is a centered Gaussian with variance $n/2 + O(1)$, it follows from the Dominated Convergence Theorem again, that multiplied by $\sqrt{n}$ the integral in~\eqref{e:300.3.30} converges to
\begin{equation}
	\frac{1}{\sqrt \pi} \int_{s=-\infty}^\infty \bbE \Big(f\big(h^0(\rmQ(k)) + s \big) \Big) \rmd s \,,
\end{equation}
which by definition of $\nu_k$ is equal to the integral on the right hand side of~\eqref{e:300.3.27n} and bounded uniformly by an $(k,M)$-dependent quantity thanks to~\eqref{e:4.8n}.

Next, setting $n' := n-r_n - 2M$, by the Gibbs-Gaussian property, we may decompose $h_n$ as the independent sum 
\begin{equation}
	\varphi_n + \sum_{x \in \rmX_{n,k}} h_{n',x}(x+\cdot ) \,,
\end{equation}
where $\varphi_n$ is the conditional expectation of $h_n$ given its values on $\cup_{x \in \rmX_{n,k}} \partial \rmQ(x;n')$ (which is a disjoint union by choice of $n'$) and 
$h_{n',x}$ is distributed as a DGFFs on $\rmQ(n')$ with zero boundary conditions outside.

Using this decomposition, we write the expectation in~\eqref{e:300.3.25} as
\begin{equation}
\label{e:300.3.37}
\bbE \exp \bigg(\sum_{x \in \rmX_{n,k}} \log \Big(1-\Psi_{n', F_x}\big(\varphi_n (x+\cdot)\big)\Big)\bigg) \,,
\end{equation}
where $\Psi$ is as in~\eqref{e:300.3.27} with $n'$ and $F_x$ in place of $n$ and $f$, where 
\begin{equation}
	F_x(\sigma) = 1-\rme^{-\lambda f_{x,k}(\sigma)} \,.
\end{equation}
Observe that $F_x$ also satisfies~\eqref{e:300.3.26} with the same $k$ and $M$ for all $x$ and thus that the corresponding integral in~\eqref{e:300.3.27n} can be bounded by a uniform $(k,M)$-dependent constant.

Now, by direct computation, $\bbE \varphi^2_n(x) \lesssim r_n $. Also, by Fernique Majorization Lemma and the Borell-TIS inequality, $\max_{y \in \rmQ(x;k)} |\varphi_n(y) - \varphi_n(x)|$ has mean $O(k\rme^{-n'})$ and Gaussian tails of the same order. Since $|\rmX_{n,k}|$ is at most $\rme^{3r_n}$ by the sparsity requirement for $\rmX_{n,k}$, it follows from the union bound that with probability tending to $1$ with $n$,
\begin{equation}
\max_{x \in \rmX_{n,k}} \big|\varphi_n(x)\big| \leq r_n \log r_n 
\quad ; \qquad
\max_{x \in \rmX_{n,k}} \max_{y \in \rmQ(x;k)}
	\big|\varphi_n(y) - \varphi_n(x)\big| \leq \rme^{-n/4} \,.
\end{equation}

When these two inequalities hold,~\eqref{e:300.3.27n} is in force, and then by Taylor approximation, the exponent in~\eqref{e:300.3.37} is equal to the second term in~\eqref{e:300.3.25}, up to multiplicative error in the exponent which tends to $1$ as $n \to \infty$. Since the quantity inside the exponential is negative, this can be turned into an additive error which tends to $0$ with $n$.
\end{proof}

We are now ready to give the proof of Proposition~\ref{p:300.2}.

\begin{proof}[Proof of Proposition~\ref{p:300.2}]
Fix $u \geq 0$ and a measurable function $f$ as in the statement, and let us first assume that $\xi_n \equiv 1$. Take $\delta > 0$ and recall the definition of $\bfW^k_{n,\delta}(u)$ from~\eqref{e:3.19n}.  
Define also $\bfLambda^{k,\circ}_{n,u, \delta}$ as in~\eqref{e:300.312n} only with 
$\bfW_{n,\delta}^k$ in place of $\bfW_n^{k}$. It follows from the isomorphism that
\begin{equation}
\label{e:300.3.41}
\bfW_{n,\delta}^k \cap \bfG_n(u) 
= \Big\{ z \in \bfW_{n,\delta}^k(u) :\:
\min \Big(L_{t_n^A}(\rmQ(x(z);k) + h^2_n(\rmQ(x(z);k) \big) \Big) \leq u \Big\} \,,
\end{equation}
where $x(z)$ is as in Definition~\ref{e:3.30.1}.
Thus, the quantity in the first expectation ~in~\eqref{e:300.3.16}, with 
$\bfXi_{n,u,\delta}^k$ in place of $\bfXi_{n,u}^k$ is equal to
\begin{equation}
\label{e:300.3.40}
\exp \Big(- \sum_{z \in \bfW_{n,\delta}^k(u)} f\big(\rme^{-n}z\big)
1_{\cA_{k,u}(L_{t_n^A}(\rmQ(x(z);k))}\big(h_n(\rmQ(x(z);k)\big) \Big)\,,
\end{equation}
where the indicator function in the sum above is for the event
\begin{equation}
\label{e:300.3.43}
\cA_{k,u}(\omega) := \Big\{\sigma \in \bbR^{\rmQ(k)} :\: \min \big(\omega + \sigma^2\big) \leq u \Big\} \,,
\end{equation}
with $\omega = L_{t_n^A}(\rmQ(x(z);k))$.

Noting that these indicator functions satisfy the conditions of Lemma~\ref{l:300.3.7} with $M=\sqrt{u} \vee 1$, we may use the latter to equate the expectation of~\eqref{e:300.3.40}, conditional on $\cF_{t_n^A}$, with that of
\begin{equation}
\exp \Big(- \tfrac{1}{\sqrt{n}} \sum_{x \in \bfW_{n,\delta}^k(u)} 
\big(1-\rme^{-f(\rme^{-n}z)}\big)
\nu_k\big(\cA_{k,u}\big(L_{t_n^A}(\rmQ(x(z);k))\big)\Big) + o(1) \,,
\end{equation}
with $\nu_k$ as in Lemma~\ref{l:300.3.7}, and $o(1) \to 0$ with $n \to \infty$. Taking expectation, and using the Bounded Convergence Theorem, we recover~\eqref{e:300.3.16} with $\varphi_u^k(\omega) := \nu_k\big(\cA_{k,u}(\omega)\big)$, but with $\bfW_{n,\delta}^k(u)$ and $\wh{\bfGamma}_{n,u, \delta}^k$ in place of $\bfW_n^k(u)$ and $\wh{\bfGamma}_{n,u}^k$. 

To get rid of the restriction to the $\delta$-bulk of the $n-r_n$ boxes, we use Proposition~\ref{p:3.8nn}. For the second expectation in~\eqref{e:300.3.16} this is sufficient as $\varphi_u^k$ is uniformly bounded in $\omega$, thanks to the Radon marginals of $\nu_k$. 
For the first expectation, since~\eqref{e:300.3.41} holds also 
with $\bfW^k_n(u) \setminus \bfW_{n,\delta}^k(u)$ in place of $\bfW^k_n(u)$, conditional on the local time field, the mean of $|(\bfW^k_n(u) \setminus \bfW_{n,\delta}^k(u)) \cap \bfG_n(u)|$ is at most $|\bfW^k_n(u) \setminus \bfW_{n,\delta}^k(u)|$ times 
\begin{equation}
\label{e:300.3.45}
	\sup_{x \in \bbX_k} \bbP \big(h_n(\rmQ(x;k) \in \cA_{k,u}(0)) \leq 
|\rmQ(k)| \sup_{x \in \rmD_n} \bbP(h_n(x)^2 \leq u) \leq C_{k,u} n^{-1/2} \,,
\end{equation}
by the Union Bound and as the variance of $h_n(x)$ is bounded from below for all $n$ large enough by $cn$ for some constant $c>0$ and all $x \in \rmD_n$. Then for any $\epsilon > 0$, $\eta > 0$, 
\begin{multline}
\bbP \Big(\big|\big(\bfW^k_n(u) \setminus \bfW_{n,\delta}^k(u)\big) \cap \bfG_n(u)\big| > \epsilon\Big)
\\ \leq 
\bbP \Big(\big|\bfW^k_n(u) \setminus \bfW_{n,\delta}^k(u)\big| > \eta \sqrt{n}\Big) + 
\bbP \Big(\big|\big(\bfW^k_n(u) \setminus \bfW_{n,\delta}^k(u)\big)\cap \bfG_n(u)\big| > 
|\bfW^k_n(u)\setminus \bfW_{n,\delta}^k(u)| \frac{\epsilon}{\eta} n^{-1/2} \Big) \,.
\end{multline}
The second probability can be made arbitrarily small for all $n$ large enough, by choosing $\eta > 0$ small enough, thanks to the (conditional) Markov inequality and~\eqref{e:300.3.45}. For all such $\eta$, the first probability can then be made arbitrarily small for all $n$ large enough, by choosing $\delta > 0$ small enough. This shows that
\begin{equation}
\big|\big(\bfW^k_n \setminus \bfW_{n,\delta}^k) \cap \bfG_n(u)\big| \overset{\bbP} \longrightarrow 0
\end{equation}
as $n \to \infty$ followed by $\delta > 0$, which implies that we may remove the $\delta$ restriction in the first expectation of~\eqref{e:300.3.16} as well.

Finally, we observe that the proof works essentially verbatim for general $\xi_n$, with the only difference is the introduction of this random variable as a multiplicative factor in all expressions above. 
\end{proof}

\subsection{Running time fluctuations as the DGFF average}

In this subsection, we give the proof of Proposition~\ref{p:1103.6}. To this end, we shall need the following auxiliary lemma.
\begin{lem}
\label{l:1103.17}
There exists $C < \infty$ such that for all $n \geq 1$,
\begin{equation}
	\Var \bigg( \sum_{x \in \rmD_n} h^2_n(x)\bigg)  \leq C |\rmD_n|^2 \,.
\end{equation}
\end{lem}
\begin{proof}
By Wick's Formula or standard computation, for $x,y \in \rmD_n$,
\begin{equation}
\bbE h_n(x)^2 h_n(y)^2 = \bbE h^2_n(x) \bbE h^2_n(y) + 2\big(\bbE h_n(x) h_n(y)\big)^2 \,.
\end{equation}
It follows that the desired variance is equal to twice the sum over all $x,y \in \rmD_n$ of the square of $\bbE h_n(x) h_n(y)$. Recalling that the latter mean is at most $G_{n}(x,y) \leq n-\log |x-y| \vee 1 + C$ (see \cite[Lemma~3.1]{Tightness})
 and partitioning the sum according to the log of the distance between $x$ and $y$, we thus obtain the bound,
\begin{equation}
	2\sum_{k=0}^\infty \sum_{\substack{x,y \in \rmD_n\\ |x-y| \leq \rme^k}} \big(n-\log (|x-y|\vee 1) + C\big)^2
	\lesssim |\rmD_n| \sum_{k=0}^{n} \rme^{2k} (n-k+1)^2 \lesssim |\rmD_n| \rme^{2n} \lesssim |\rmD_n|^2 \,.
\end{equation}
\end{proof}

We are now ready to give the proof of Proposition~\ref{p:1103.6}.

\begin{proof}[Proof of Proposition~\ref{p:1103.6}]
Summing both sides of~\eqref{e:3.1n} over all $x \in \wh{\rmD}_n$, dividing by
$2\sqrt{t_n} |\rmD_n|$, rearranging the terms and finally noting that $\sum_{x \in \wh{\rmD}_n} L_{t_n}(x) = \bfT_{n,t_n}$, the left-hand side of~\eqref{e:1103.19} can be equated to
\begin{equation}
	\frac{1}{2\sqrt{t_n}} \frac{1}{|\rmD_n|} \Big(\sum_{x \in \rmD_n} h'^2_n(x) - \sum_{x \in \rmD_n} h^2_n(x)\Big) + \frac{\sqrt{t_n}}{2|\rmD_n|}
\end{equation}
The second moment of the first term tends to $0$, by virtue of Lemma~\ref{l:1103.17} and the fact that $t_n \to \infty$ with $n$.
The second term tends to $0$ by assumption. 
\end{proof}

\section{I.I.D. structure of the law of local time clusters}\label{sec:iid}

In this section we show that the set $\rmW_n(u)$ is cluster-i.i.d., that is, that the law of the local time on each cluster is asymptotically identical and that, between different clusters, local-time values are independent. The product of this section is Proposition~\ref{p:300.1a} and, as a direct consequence, also Proposition~\ref{p:300.5}.
The main idea is the following. First, we switch to considering downcrossings of concentric balls, in place of local time. This approach, which as pioneered in~\cite{dembo2004cover}, recovers the Spatial Markov property, which is otherwise absent when local time is considered. This will give the independence between the clusters, once we consider the number of downcrossings around them and the entry and exit points at each. 

Then, to get the identical distribution of the local time at each of the clusters, we consider not one set of downcrossings, but rather a full sequence of these, leading from some large scale $l$ to the scale of the cluster $k$. We then show that the (square root of the normalized) downcrossing trajectory along this sequence is repelled: that is, instead of decreasing linearly to $O(1)$, it stays above the linear interpolation following the square-root function. This repulsion then allows us to show that the law of the downcrossing information at scale $k$ (i.e., the number of downcrossings at scale $k$ together with their corresponding entry and exit points) is asymptotically insensitive to the downcrossing information at scale $l$.Thus, even if the law of the latter is not identical for different clusters, the law of the former will be. 

We begin by introducing the downcrossing terminology. We follow a similar notation to that in~\cite{Tightness}.

\subsection{Downcrossing terminology}\label{sec:downprelim}

We first recall the general definition of downcrossing from \cite{Tightness}. Throughout the sequel, the inner boundary of a set $A \subseteq \bbZ^2$ is defined as $\partial_i A := \partial A^\rmc$. 

To define the latter precisely, given any $x \in \rmD_n$ and $l > k \geq 0$ with $\ol{\rmB(x;l)} \subseteq \rmD_n$, we introduce the sequence of stopping times $(\sigma_m)_{m \geq 0}$ and $(\sigma'_m)_{m \geq 1}$ defined recursively by first setting $\sigma_0 := 0$ and then, for $m \geq 1$, letting 
\begin{equation}
	\sigma'_m := \inf \{ t  \geq \sigma_{m-1} :\: X_t \in \rmB(x;k) \}
 \ , \quad \sigma_m := \inf \{ t  \geq \sigma'_{m} :\: X_t \in \rmB(x;l)^\rmc \}\,,
\end{equation} where we recall that, for each $r \geq 0$, $\rmB(x;r):=\{ y \in \bbZ^2 : \|x-y\| \leq \lfloor \rme^r \rfloor\}$ as defined in Section~\ref{s:2}. We shall call any excursion of the walk $\bfX$ from $\rmB(x;k)$ to $\rmB(x;l)^\rmc$ an $(x;k,l)$-\textit{excursion} and any excursion from $\rmB(x;l)^\rmc$ to~$\rmB(x;k)$ an $(x;k,l)$-\textit{downcrossing}, often calling these simply excursions or downcrossings whenever the choice of $x,k$ and $l$ is clear. Having in mind the stopping~times above and recalling that our walk $\bfX$ starts from $\partial$,  $(x;k,l)$-excursions correspond to the paths $(\bfX_t : t \in [\sigma'_m,\sigma_{m}])$ and $(x;k,l)$-downcrossings to the paths $(\bfX_t : t \in [\sigma_{m-1},\sigma'_m])$ for each $m \geq 1$. 
The number of $(x;k,l)$-{\em downcrossings} made by the walk until the total running time is $t$ is then
\begin{equation}\label{eq:defntz}
\bfN_t(x;k,l) := \sup\, \{m \geq 1 :\: \sigma'_m \leq t \} \,.
\end{equation} The number of $(x;k,l)$-downcrossings until the walk accumulates local time $t$ at the boundary~$\partial$ is then given by $N_t(x;k,l) := \bfN_{\bfL^{-1}_t(\partial)}(x;k,l)$. Note that $N_t(x;k,l)$ coincides with the number of $(x;k,l)$-excursions until such time, since each downcrossing accounted for in $N_t(x;k,l)$ must be followed by an $(x;k,l)$-excursion prior to the next downcrossing and/or return to the boundary. We also introduce a {\em normalized} version of $N_t(x;k,l)$ given by
\begin{equation}\label{eq:defntz2}
	\wh{N}_t(x;k,l) := (l-k) N_t(x;k,l) \,.
\end{equation} 
For $m \geq 1$, the \textit{entry} and \textit{exit} points of the $m$-th excursion will be defined as $\bfX_{\sigma'_m}$ and $\bfX_{\sigma_m}$, respectively. 
We also let 
\begin{equation}\label{eq:defsigma}
	\cF(x;k,l) := \sigma \big( X_t :\: t \in \cup_{m\geq 1} \, [\sigma_{m-1}, \sigma'_m] \big)
\end{equation}
 be the $\sigma$-algebra generated by the random walk $\bfX$ when ``observed only during downcrossings'', and notice that both $N_t(x;k,l)$ and $\wh{N}_t(x;k,l)$ for any $t \geq 0$ as well as the entry and exit points of any $(x;k,l)$-excursion are all measurable with respect to $\cF(x;k,l)$. 

In our work, we will often consider a whole sequence of concentric balls around a given~point, and be interested in downcrossings between any two concentric balls this sequence. To handle this setting, we introduce some additional notation which is more convenient. 

Given any $k\geq 1$ and $T \in \N$, let us fix any $n \geq k$ and $x \in \rmD_n$ such that $\overline{\rmB(x; k+Tk^\gamma)} \subseteq \rmD_n$. 
For each $i=1,\dots,T$, let us define
\begin{equation}
	\rmB^-[x;i]:=\rmB(x;k+(i-1)k^\gamma) \qquad\text{ and }\qquad \rmB^+[x;i]:=\rmB(x;k+ik^\gamma-4\rme^{-k^\gamma}).
\end{equation} 
and call any $(x;k+(i-1)k^\gamma,k+ik^\gamma-4\rme^{-k^\gamma})$-excursion of the walk $\bfX$ an $[x;i]$-\textit{excursion} and any $(x;k+(i-1)k^\gamma,k+ik^\gamma-4\rme^{-k^\gamma})$-downcrossing an $[x;i]$-\textit{downcrossing}. We point out that, if $k$ is large enough (depending only on $\gamma$), then, for all $i=1,\dots,T-1$, we have
\begin{equation}\label{eq:incballs}
	\ol{\rmB^+[x;i]} \subseteq \rmB^-[x;i+1] \qquad \text{ and }\qquad \rme^{k+(i-1)k^\gamma} \leq \rmd(\ol{\rmB^+[x;i]}, \partial_i \rmB^-[x;i+1]) \leq 5\rme^{k+(i-1)k^\gamma},
\end{equation} so that the annuli $\ol{\ol{\rmB^+[x;i]}\setminus \rmB^-[x;i]}$ for $i=1,\dots,T$ are all at a positive distance from each~other (and concentric) but the two  boundaries $\partial \rmB^+[x;i]$ and $\partial_i \rmB^-[x;i+1]$ become indistinguishable at log-scale $k+ik^\gamma$. This  feature will be convenient when studying downcrossing trajectories, i.e. the number of $[x;i]$-downcrossings as a function of $i$, so that, in the following, we will assume $k$ to be always large enough so that \eqref{eq:incballs} holds. We also define the $\sigma$-algebra 
\begin{equation}
\cF[x;i]:=\cF(x;k+(i-1)k^\gamma,k+ik^\gamma -4\rme^{-k^\gamma}),
\end{equation} where for $l \geq k$ the $\sigma$-algebra $\cF(x;k,l)$ is as in \eqref{eq:defsigma}. \textcolor{black}{Given a random variable $U$ and $A \subseteq\cF[x;i]$, in the sequel we will  write $\bbE(U| A\,; \cF[x;i])$ to denote the restriction of the conditional expectation $\bbE(U|\cF[x;i])$ to the event $A$.}
In addition, for any $t > 0$ let us set
\begin{equation}
	N_t[x;i]:=N_t(x;k+(i-1)k^\gamma,k+ik^\gamma-4\rme^{-k^\gamma})\text{ and }\wh{N}_t[x;i]:=k^\gamma N_t(x;k+(i-1)k^\gamma,k+ik^\gamma-4\rme^{-k^\gamma}) 	
\end{equation} as the number (resp. normalized number) of $[x;i]$-downcrossings until $\partial$-time $t$ and, given $j \in \N$, let $Y_j^{\rm in}[x;i]$ and $Y_j^{\rm out}[x;i]$ denote respectively the entry and exit points of the $j$-th $[x;i]$-excursion and write $\ol{Y}_j[x;i]:=(Y^{\rm in}_j[x;i], Y^{\rm out}_j[x;i])-x$ for the centered $j$-th pair of such entry/exit points, $\ol{Y}_{(j)}(x;k):=(\ol{Y}_1[x;i],\dots,\ol{Y}_j[x;i])$ for the vector of the first $j$ centered pairs of entry/exit points and, finally,
\begin{equation}\label{eq:defz}
	Z_t[x;i]:=\Big( \sqrt{\wh{N}_t[x;i]} \, , \ol{Y}_{(N_t[x;i])}[x;i]\Big)	
\end{equation} for the vector listing the (square root of the normalized) number of $[x;i]$-excursions until $\partial$-time~$t$ together with their corresponding entry and exit points (with the convention that $\ol{Y}_{(0)}[x;i]=\emptyset$ and $Z_t[x;i]=(0,\emptyset):=0$, used whenever $\wh{N}_t[x;i]=0$). For future reference, we observe that the random variable $ \sqrt{\wh{N}_t[x;i]}$ takes values on the set 
\begin{equation}
\mathcal{N}_k:=\{ \sqrt{ k^\gamma m} : m \in \N_0\},
\end{equation} while $Z_t[x;i]$ takes values on the set 
\begin{equation}\label{eq:defZcal}
\mathcal{Z}_k[i]:=\{ (u,\ol{y}) : u \in \cN_k\,,\, \ol{y} \in \textrm{E}_{\wt{u}_k}[i] \},
\end{equation}
where, given $u \in \cN_k$, we write $\wt{u}_k:=\frac{s^2}{k^\gamma}$ (so that $\wt{u}_k$ is the unique integer $m$ such that  $u= \sqrt{ k^\gamma m}$) and, for each $m \in \N$, we define
\begin{equation}
	E_m[i]:=\{ \overline{y}= ((y_1,y_1'),\dots,(y_{m},y'_{m})) : (y_r,y_r') \in \partial_i \rmB^-[0;i] \times \partial \rmB^+[0;i] \text{ for all $r=1,\dots,m$}\}
\end{equation} and set $\textrm{E}_0[i]:=\emptyset$ by convention.  In addition, given $\eta \in (0,\tfrac{1}{2})$, we introduce the corresponding repulsion interval 
\begin{equation}\label{eq:defirep}
		\mathfrak{I}_k^\eta(i):=[(k+ik^\gamma)^{1/2-\eta},(k+ik^\gamma)^{1/2+\eta}]\,,
	\end{equation} and consider the set $\cR^\eta_k(i)$ of all $u \in \cN_k$ which lie on the interval $\sqrt{2}(k+(i-1)k^\gamma) + \mathfrak{I}_k^\eta(i)$, i.e.	
	\begin{equation}
		\mathcal{R}_k^\eta(i):=\Big\{ u \in \cN_k : u - \sqrt{2}(k+(i-1)k^\gamma) \in \mathfrak{I}_k^\eta(i)\Big\}.	
	\end{equation} Given $u \in \cR_k^\eta(i)$, we define the \textit{recentering} of $u$ as 
	\begin{equation}\label{eq:defudc}
	\wh{u}:=u - \sqrt{2}(k+(i-1)k^\gamma).
	\end{equation}
\textcolor{black}{Notice that $\wh{u}$ depends on $i$ but we are omitting this dependence from the notation for simplicity. Nevertheless, this will not cause any confusion in the sequel, as we will always clarify beforehand that $u \in \cR^\eta_k(i)$ for some specific value of $i$, which is the precise value to be used when recentering.} Next, we define the corresponding extension of $\cR^\eta_k(i)$ to $\cZ_k[i]$ given by
	\begin{equation}
\cZ^{\eta}_k[i]:=\{ (u,\ol{z}) \in \cZ_k[i] : u \in \cR^{\eta}_k(i)\}.	
\end{equation} Finally, we define define the event 
	\begin{equation}
	\mathrm{NR}^{\eta,i}_{t}(x):= \Big\{ \sqrt{\wh{N}_t[x;i]}  \notin \mathcal{R}^\eta_k(i)\Big\}  
\end{equation} 
and extend this definition to subsets $K \subseteq [1,T]$ in place of $i$ as usual by union over all $i \in K \cap \N_0$.

\subsection{General outline of the proof of Proposition~\ref{p:300.1a}}

We now give a general outline of the proof of Proposition~\ref{p:300.1a},  expanding on the ideas discussed at the beginning of the section.

The first step of the proof of Proposition~\ref{p:300.1a} is to show that, to establish the convergence of the empirical distribution $\bfLambda^k_{n,u}$, it suffices to consider only clusters whose downcrossing trajectory is properly repelled. This is due to the fact that, as shown in~\cite{Tightness}, the downcrossing trajectory corresponding to a low local time vertex in $\rmD_n$ is typically repelled. To make an analogous claim, albeit with a slightly different formulation,  we set  
\begin{equation}\label{eq:deflt}
l=l(k):=k^2 \qquad , \qquad \cT_k:= \left\lfloor \frac{l - k}{k^\gamma}\right\rfloor
\end{equation} and, for $n \geq k$, $u \geq 0$ and $0 < \eta' < \eta < \frac{1}{2}$, define
\begin{equation}
\label{e:defR0}
\begin{split}
\bfR_n^{k}(u) := \Big\{ z \in \bfW_n^{k}(u) :\: &  \sqrt{\wh{N}_{t^A_n}[x(z);i]} \in \cR_k^{\eta'}(i) \text{ for all  }i \in \{1,\cT_k\}, \\ &
\sqrt{\wh{N}_{t^A_n}[x(z);i]} \in \cR_k^\eta(i) \text{ for all  }i \in [2,\cT_k+5]\setminus \{ \cT_k\}\Big\},
\end{split}
\end{equation} where $x(z) \in \bbX_k$ is that from \eqref{e:3.30.1} and we chose to suppress the dependence on the parameters $\eta',\eta$ from the notation for simplicity.

\begin{rem} In the definition of $\bfR^k_n(u)$, and also in all similar definitions to come in Section~\ref{sec:iid}, we~implicitly assume that $z$ in addition satisfies $\ol{\rmB^+[x(z);\cT_k+5]} \subseteq \rmD_n$, so that $\wh{N}_{t^A_n}[x(z);\cT_k+t]$ is well-defined. We choose not to state this condition explicitly to avoid overloading the definition. 
\end{rem}

The first part of the proof consists then of establishing the following analogue of \cite[Lemma~7.2]{Tightness}, which states that the number of clusters without a repulsed downcrossing trajectory is negligible.
\begin{prop}
\label{prop:rest} If $\gamma$ is chosen small enough, then, given any $0 < \eta' < \eta < \frac{1}{2}$ small enough \mbox{(depending only on $\gamma$),} for any fixed $u \geq 0$ we have  
\begin{equation}
\frac{\big|\bfW^k_n(u) \setminus \bfR^{k}_n(u)\big|}{\sqrt{n}} \overset{\bbP}\longrightarrow 0 \,,
\end{equation}
in the limit as $n \to \infty$, followed by $k \to \infty$.
\end{prop}

Proposition~\ref{prop:rest} will be an immediate consequence of the (slightly) stronger Proposition~\ref{prop:rest3} below, which will be proved in Section~\ref{sec:prest}.

Next, we show that, under a repelled trajectory, the law of the downcrossing ``output'' data at the terminal scale $k$ is insensitive to the downcrossing input data at the initial scale $l \gg k$ which is conditioned upon. This is a consequence of the decoupling between the input and output data given by the following ballot-type estimate for the downcrossing \mbox{trajectory of the random walk.} In the statement of Theorem~\ref{prop:coupling0} below, given a nonempty set $A \subseteq \bbZ^2$, $\Pi_A$ and $\amalg_A$ respectively denote the Poisson kernel and the harmonic measure (from infinity) associated with $A \subseteq \bbZ^2$, see Section~\ref{sec:prelimballot} for a precise definition.

\begin{thm}\label{prop:coupling0} If $\gamma$ chosen small enough, then, given any $0 < \eta' < \eta < \tfrac{1}{2}$ sufficiently small (depending only on $\gamma$) there exists a constant $c=c(\eta,\eta') > 0$ such that, if $k$ is sufficiently large, for any $2k \leq T \leq \rme^{ck^\gamma}$, $n \geq k$ and $x,\wt{x} \in \rmD_n$ satisfying that $x \in \rmB^-[\wt{x};T-1]$ and $\overline{\rmB^+[\wt{x};T]} \subseteq \rmD_n$, the asymptotics	\begin{equation}\label{eq:lemacoup}
		\bbP \Big( \big\{ Z_t[x;1] = (u,\ol{y})\big\} \setminus \mathrm{NR}^{\eta,[2,T-1]}_{n,t}(x) \,\Big|\,Z_t[\wt{x};T] = (v,\ol{z}) \Big)=r_{k,T}(\wh{u},\wh{v})h_k(\ol{y})(1+o_k(1))
	\end{equation} hold uniformly over all $t > 0$, $(u,\ol{y}) \in \cZ_k^{\eta'}[1]$ and $(v,\ol{z}) \in \cZ_k^{\eta'}[T]$, where the functions $r_{k,T}$, $h_k$ are respectively given by
	\begin{equation}\label{eq:funcfg}
		r_{k,T}(\wh{u},\wh{v}):=	 \sqrt{\frac{2}{\pi}}(Tk^{\frac{3}{2}})^{-1}\rme^{-2(T-1)k^\gamma}\left(\wh{u}\rme^{2\sqrt{2}\wh{u}}\right)\left(\wh{v}\rme^{-2\sqrt{2}\wh{v}-\frac{\wh{v}^2}{Tk^\gamma}}\right)
	\end{equation}and
	\begin{equation}\label{eq:defhk1}
		h_k(\ol{y}):=\begin{cases}1 & \text{ if $\ol{y}=\emptyset$}\\ \\\displaystyle{\prod_{i=1}^m \amalg_{\rmB^-[0;1]}(y_i) \Pi_{\rmB^+[0;1]}(0,y_i')} &\text{ if $\ol{y}=((y_1,y_1'),\dots,(y_{m},y'_{m}))$ for some $m \in \N$}\,.
		\end{cases}	
	\end{equation}
\end{thm}
\noindent
This theorem, which constitutes one of the main efforts in the manuscript, is proved in Section~\ref{s:barrier}.

Using the independence induced by considering downcrossings and the law-invariance given by the preceding theorem, the next step of the proof is to show the convergence of the empirical measure of the local-time fields corresponding to repelled clusters. To make a precise statement,  given $n \geq k$, $u \geq 0$ and $0 < \eta' < \eta < \frac{1}{2}$,  we introduce 
the empirical measure of local time clusters in $\bfR^{k,}_n(u)$ via
\begin{equation}
\label{e:300.312bb}
\mathcal{O}^{k}_{n,u} :=	\frac{1}{\big|\bfR_n^{k}(u)\big|}
	\sum_{z \in \bfR_n^{k}(u)} \delta_{\rme^{-n}z} \otimes 
	\delta_{L_{t_n^A}(\rmQ(x(z);k))} \,, 
\end{equation}
with space marginal $\wh{\Theta}^k_{n,u}$ given by 
\begin{equation}
\wh{\Theta}^k_{n,u}:=\frac{1}{|\bfR^k_n(u)|}\sum_{z \in \bfR^k_n(u)} \delta_{\rme^{-n}z}.	
\end{equation}
These are analogues of the measures $\bfLambda_{n,u}^{k}$ and $\wh{\bfXi}^k_{n,u}$ from~\eqref{e:300.312n}. The next step of the proof is then to show the following analogue of Proposition~\ref{p:300.1a} for $\mathcal{O}^{k}_{n,u}$.
\begin{prop}
\label{p:300.1} If $\gamma$ is chosen small enough, then, given any $0 < \eta' < \eta < \frac{1}{2}$ small enough (depending only on $\gamma$), for each $u,k \geq 0$ there exists a probability measure $\rmO^{k}_u$ on $\bbR^{\rmQ(k)}$ such~that 
\begin{equation}\label{eq:convos}
\frac{\big|\bfR_n^{k}(u)\big|}{\sqrt{n}}
\bigg|\int f_k\,\rmd \mathcal{O}^{k}_{n,u} - \int f_k\, \rmd (\wh{\Theta}^k_{n,u} \otimes \rmO^{k}_u)\bigg| \overset{\bbP} \longrightarrow \, 0 \,,
\end{equation}
as $n \to \infty$ followed by $k \to \infty$, uniformly over all measurable $f_k: \rmD \times \bbR^{\rmQ(k)} \to \bbR$ with $\|f_k\|_\infty \leq 1$.
\end{prop}

The proof of this result (for which Theorem~\ref{prop:coupling0} is essential) is deferred to the next subsection. With the aid of Proposition~\ref{prop:rest} and Proposition~\ref{p:300.1}, we can now give the proofs of Proposition~\ref{p:300.1a} and Proposition~\ref{p:300.5}. 

\begin{proof}[Proof of Proposition~\ref{p:300.1a}]
By Proposition~\ref{prop:rest}, we have that the difference between 
$|\bfW_n^k(u)|$ and $|\bfR_n^{k}(u)|$ divided by $\sqrt{n}$ tends to zero in probability in the limit as $n \to \infty$ followed by $k \to \infty$. Using this together with Proposition~\ref{p:300.1} gives
\begin{equation}
\label{e:300.3.53}
\frac{\big|\bfW_n^k(u)\big|}{\sqrt{n}}
\bigg|\int f_k\,\rmd \bfLambda^{k}_{n,u} - \int f_k \,\rmd (\wh{\Theta}^k_{n,u} \otimes \rmO^{k}_u)\bigg|
\overset{\bbP} \longrightarrow \, 0 \,,
\end{equation}
as $n \to \infty$ followed by $k \to \infty$, 
uniformly over all $f_k$ as in the statement of Proposition~\ref{p:300.1}. Above $\rmO^{k}_u$ is as in Proposition~\ref{p:300.1} and $\bfLambda_{n,u}^k$ is as in~\eqref{e:300.312n}. 

Now, given $r > 0$ and a measurable function $f$ on $\rmD \times \bbR^{\rmQ(r)}$ with $\|f\|_\infty \leq 1$, for each $j \in \bbN$ take $k_j \geq \bbN$ so that $\lfloor \rme^{k_j}\rfloor = \lfloor \rme^r \rfloor^j$ and define $\varphi_j^{f}$ on $\rmD \times \bbR^{\rmQ(k_j)}$ via
\begin{equation}
\varphi_j^{f}(x,\omega) := \sum_{\substack{y \in \bbX_r:\:\\\rmQ(y;r) \subseteq \rmQ(k_j)}}
1_{\cA_{k_j,r}(y)}(\omega) f \big(x,\omega\big(\rmQ(y;r)\big)\big) \,,
\end{equation}
where 
\begin{equation}
\cA_{k_j,r}(y) :=	\Big\{\omega \in \bbR^{\rmQ(k_j)} :\: 
	\min_{\rmQ(y;r)} \omega \leq u \,,\,\,
	\min_{\rmQ(k_j) \setminus \rmQ(y;r)} > u \Big\} \,.
\end{equation}
Then, since each box $\rmQ(y;r)$ with $y \in \bbX_r$ is contained in exactly one box $\rmQ(x;k_j)$ with $x \in \bbX_{k_j}$, for all $n$ large enough (depending only on $k_j$) we have 
\begin{equation}
	\left| \big|\bfW_n^r(u)\big| \int f \,\rmd \bfLambda_{n,u}^r -
	\big|\bfW_n^{k_j}(u)\big| \int \varphi_j^{f}\,\rmd \bfLambda_{n,u}^{k_j} \right| \leq |\bfW^{\wt{\rmD}_n}_n(u)|, 
\end{equation} where $\bfW^{\wt{\rmD}_n}_n(u)$ is defined as in Lemma~\ref{l:1103.7} for $\wt{\rmD}_n$ given by  
\begin{equation}
\wt{\rmD}_n:= \bigcup_{z \in \bbX_{n-r_n}} \rmQ(z;n-r_n) \setminus \rmQ(z;n-r_n-\rme^{-r_n}).
\end{equation}
Plugging this in~\eqref{e:300.3.53} and using Lemma~\ref{l:1103.7} together with the lower tightness of $|\bfW_n(u)|/\sqrt{n}$, since $\| \varphi_j^f\|_\infty \leq 1$ and $|\wt{\rmD}_n \cap \rmD_n| \leq o_n(1)|\rmD_n|$ for some $o_n(1) \to 1$, we get
\begin{equation}\label{eq:conv2variable}
\Bigg|
\frac{\big|\bfW^r_n(u)\big|}{\big|\bfW_n(u)\big|}
\int f\,\rmd \bfLambda^{r}_{n,u} - 
\int \varphi_j^{f}\, \rmd (\wh{\Theta}^{k_j}_{n,u} \otimes \rmO^{k_j}_u)\Bigg|\overset{\bbP} \longrightarrow \, 0 \,,
\end{equation}
in the limit as $n \to \infty$ followed by $j \to \infty$, uniformly in all such $f$. 

Now, if we consider functions of the form $f(x,\omega):=\wt{f}(\omega)$ for some measurable $\wt{f}$ on $\bbR^{\rmQ(k_j)}$ with $\|\wt{f}\|_\infty \leq 1$, then the former gives
\begin{equation}\label{eq:diff1v}
\Bigg|
\frac{\big|\bfW^r_n(u)\big|}{\big|\bfW_n(u)\big|}
\int f\,\rmd \bfLambda^{r}_{n,u} - 
\int \varphi_j^{\wt{f}}\, \rmd  \rmO^{k_j}_u\Bigg|\overset{\bbP} \longrightarrow \, 0 \,,
\end{equation}
in the limit as $n \to \infty$ followed by $j \to \infty$, uniformly in all such $\wt{f}$, where $\varphi_j^{\wt{f}}$ is defined analogously to $\varphi^f_j$, but with $\wt{f}$ in place of $f$. But, in this case, since the two terms in the difference in \eqref{eq:diff1v} depend on different limiting parameters, it follows by a standard argument that
\begin{equation}
\label{e:300.3.58}
\frac{\big|\bfW^r_n(u)\big|}{\big|\bfW_n(u)\big|}
\int f\,\rmd \bfLambda^{r}_{n,u} 
\ \underset{n \to\infty}{\overset{\bbP}\longrightarrow}\ 
\lim_{j \to \infty}
\int \varphi_j^{\wt{f}} \rmd \rmO^{k_j}_u  \,,
\end{equation}
uniformly in all $\wt{f}$ with $\| \wt{f} \|_\infty \leq 1$.

In particular, plugging in $\wt{f} \equiv 1$ in this last display yields the convergence of
$|\bfW^r_n(u)\big|/\big|\bfW_n(u)\big|$ as $n \to \infty$ to the 
limit $\rmO_u(r):=\lim_{j\to\infty}
\rmO_u^{k_j} \big(\cA_{k_j,r})$, where 
\begin{equation}\label{eq:conv1variable}
\cA_{k_j,r} := \bigcup_{\substack{y \in \bbX_r:\:\\\rmQ(y;r) \subseteq \rmQ(k_j)}} \cA_{k_j,r}(y)
\end{equation}
is the event that exactly one box $\rmQ(y;r)$ in $\rmQ(k_j)$ with $y \in \bbX_r$ contains a vertex with local time at most $u$. Notice that this limit cannot be zero, by the tightness below and above of
$|\bfW^r_n(u)\big|/\sqrt{n}$ and $|\bfW_n(u)\big|/\sqrt{n}$, respectively. Using this in~\eqref{e:300.3.58} we thus get
\begin{equation}\label{eq:limit}
\int f\,\rmd \bfLambda^{r}_{n,u} 
\ \underset{n \to\infty}{\overset{\bbP}\longrightarrow}\ 
\lim_{j \to \infty}
\int \varphi_j^{\wt{f}} \rmd \rmO^{k_j}_u(\cdot|\cA_{k_j,r})  \,,
\end{equation}
uniformly in all $\wt{f}$ as above.

Lastly, the limit on the right-hand side of \eqref{eq:limit} is clearly a bounded linear functional on the space of bounded functions on $\bbR^{\rmQ(r)}$ and, as such, is a finitely additive measure on $\bbR^{\rmQ(r)}$. It follows from the uniformity in the convergence that this measure must be also countably~additive. Denoting this measure by $\rmP_u^r$, in light  of \eqref{eq:conv2variable} and the fact that $|\bfW^r_n(u)\big|/\big|\bfW_n(u)\big|$ converges to $\rmO_u(r) > 0$ as $n \to \infty$, in order to complete the proof it suffices to show that
\begin{equation}
\left|\int \varphi^f_j\,\rmd \big(\wh{\Theta}^{k_j}_{n,u} \otimes \rmO^{k_j}_u(\cdot| \cA_{k_j,r}) \big)- \int f\,\rmd (\wh{\Theta}^{r}_{n,u} \otimes \rmP^r_u)\right|	\overset{\bbP}{\longrightarrow} 0
\end{equation} as $n \to \infty$ followed by $j \to \infty$ uniformly over all measurable functions $f$ (now of the two variables $(x,\omega)$) on $\rmD \times \bbR^{\rmQ(r)}$ such that $\|f\|_\infty \leq 1$. 

To this end, observe that, by considering functions of the form $f(x,\omega):=\wh{f}(x)$ for some measurable $\wh{f}:\rmD \to \bbR$, \eqref{eq:conv2variable} implies that
\begin{equation}
\bigg|\int \wh{f} \rmd \wh{\bfXi}^{r}_{n.u} - \int \wh{f} \rmd \wh{\Theta}^{k}_{n,u}\bigg| \overset{\bbP}{\longrightarrow} 0	
\end{equation} as $n \to \infty$ followed by $k \to \infty$, uniformly over all such $\wh{f}$.

\end{proof}

As a consequence of Proposition~\ref{p:300.1a}, we immediately obtain Proposition~\ref{p:300.5}.

\begin{proof}[Proof of Proposition~\ref{p:300.5}] 
For fixed $k$, by Proposition~\ref{p:300.1a} we have that
\begin{equation}\label{eq:firstconv}
\frac{|\bfW^k_n(0)|}{|\bfW^k_n(u)|}\int f(x) \,\wh{\bfXi}^k_{n,0}(\rmd x)  - C_{u,k} \int f(x)\,\wh{\bfXi}^k_{n,u}(\rmd x)  \ \underset{n \to\infty}{\overset{\bbP}\longrightarrow} 0,
\end{equation}
as $n \to \infty$ uniformly over all measurable functions $f: \rmD \to \bbR$ with $\|f\|_\infty \leq 1$, where
\begin{equation}
C_{u,k}:=\rmP_u^k\Big(\big\{ \omega \in \bbR^{\rmQ(k)} : \min_{\rmQ(k)} \omega = 0\big\}\Big).
\end{equation}
On the other hand, by the lower tightness of $|\bfW_n(u)|/\sqrt{n}$ and Theorem~\ref{p:300.3} (used for $u' \in \{0,u\}$ in place of $u$), 
\begin{equation}
\frac{|\bfW_n(0)|}{|\bfW_n(u)|}-\frac{|\bfW^k_n(0)|}{|\bfW^k_n(u)|} \overset{\bbP}{\longrightarrow} 0
\end{equation} in the limit as $n \to \infty$ followed by $k \to \infty$. Taking $f \equiv 1$ in \eqref{eq:firstconv}, it then follows that
\begin{equation}
	\frac{|\bfW_n(0)|}{|\bfW_n(u)|}-C_{u,k} \overset{\bbP}{\longrightarrow} 0
\end{equation} in the same limits. As the two terms in this difference depend on different parameters, it follows from a standard argument that the sequence $(C_{u,k})_{k \in \N}$ converges and, moreover, that, as $n \to \infty$,
\begin{equation}
	\frac{|\bfW_n(0)|}{|\bfW_n(u)|} \overset{\bbP}{\longrightarrow} C_u:=\lim_{k \to \infty} C_{u,k}.	
\end{equation} The fact that $C_u$ is less or equal than one is immediate. The fact that it is positive follows from the lower tightness of $|\bfW_n(0)|/\sqrt{n}$ and the upper tightness of $|\bfW_n(u)|/\sqrt{n}$.
\end{proof}

Thus, the remainder of Section~\ref{sec:iid} is devoted to the proof of Proposition~\ref{p:300.1}.

\subsection{Convergence of the empirical measure of repelled low-local time clusters}\label{sec:exist0}

\subsubsection{General outline of the proof}\label{sec:exist1}
In this subsection we prove Proposition~\ref{p:300.1}. The general strategy of the proof can be summarized as follows. We shall consider versions of the set $\bfR^k_n(u)$ in which, to determine whether a given $z \in \bfW_n(u)$ belongs to this version, one does not need to look at all $k$-boxes with centers in~$\bbX_k$, but only at boxes belonging to a specific subset of these which are at least $l$ log-distance apart, where $l$ is as in \eqref{eq:deflt}. Doing so will provide enough room to consider downcrossing trajectories up to scale $l \gg k$ around each of these separated boxes, in such a way that the trajectory around any specific box will never give any additional information about the trajectories of the others. By the spatial Markov property, upon conditioning on the  downcrossings information at scale $l$ around each box, we will obtain that the local time fields inside each of these separated $k$-boxes are independent. Furthermore, since we are working only with repelled low-local time clusters, by Theorem~\ref{prop:coupling0}, the law of these fields will turn out to be insensitive to the precise number of downcrossings at scale $l$ around the corresponding box, which will then imply that these fields are also identically distributed. The convergence of the associated empirical measure will then be a straightforward consequence of the law of large numbers. Finally, to conclude the proof of Proposition~\ref{p:300.1}, we will need to show that the existence of limit for the empirical measures associated with each of these versions of $\bfR^k_n(u)$ implies the same for the original set $\bfR^k_n(u)$.

 To make this strategy rigorous, we will need to perform several auxiliary constructions. First, for each $m_1 \geq m_2 \geq 1$, let us fix a partition $(\bbX_{m_1,m_2}(j) : j \in J_{m_1,m_2})$ of $\bbX_{m_1}$ into subgrids $\bbX_{k,l}(j)$ such that, for any $j \in J_{k,l}$ and  $x,y \in \bbX_{m_1,m_2}(j)$ with $x\neq y$, we have $\rmB(x;m_2)\cap \rmB(y;,m_2) = \emptyset$. Note that it is possible to choose this partition so that $|J_{m_1,m_2}|=\Big(\Big\lfloor 2\frac{\lfloor \rme^{m_2}\rfloor}{\lfloor \rme^{m_1}\rfloor}\Big\rfloor+1\Big)^2\leq C \rme^{2(m_2-m_1)}$ for some universal constant $C > 0$ and all $m_2 \geq m_1$ with $m_1$ large enough. The collections $(\rmQ(x;k) : x \in \bbX_{k,l}(j))$ with $j \in J_{k,l}$ are precisely the specific subsets of separated $k$-boxes alluded to in the preceding paragraph. Next, to introduce the version of $\bfR^k_n(u)$ looking only at $k$-boxes with centers in $\bbX_{k,l}(j)$, we first consider the subset of $z \in \bfR^k_n(u)$ whose associated low-local time cluster lies far away from the boundary of $\rmQ(z;n-r_n)$ and also deep within the unique $k$-box containing it. More precisely, for $\epsilon \in [0,1)$, $1 \leq k \leq  n -r_n $ and $u \geq 0$, define
\begin{equation}
\begin{split}
	\bfW^{k, \epsilon}_n(u) := \Big \{ z \in \bfW^k_n(u) :\:
	&\rmW_n(u) \cap \rmQ(x(z); k-\epsilon) \neq \emptyset ,\\
	& \rmW_n(u) \cap \rmQ(z;n - r_n ) \subseteq \rmQ(z; n-r_n - \epsilon)
	\Big\} \,,
\end{split}
\end{equation} where $x(z)$ above denotes the unique $x(z) \in \bbX_k$ from \eqref{e:3.30.1}, together with 
\begin{equation}
	\bfR^{k,\epsilon}_n(u):=\bfW^{k, \epsilon}_n(u) \cap \bfR^k_n(u).
\end{equation}
Observe that $\bfW_n^k(u)=\bfW_n^{k,0}(u)$ and, as such, $\bfR^{k}_n(u)=\bfR^{k,0}_n(u)$. The advantage of working with $\bfR^{k,\epsilon}_n(u)$ instead of $\bfR^k_n(u)$ is that the extra $\epsilon$-room will be more convenient for our computations. Furthermore, restricting ourselves to the set $\bfR^{k,\epsilon}_n(u)$ will turn out to be asymptotically harmless, as the next result shows.

\begin{prop}\label{prop:rest2} If $\gamma$ is chosen small enough, then, given any $0 < \eta' < \eta < \frac{1}{2}$ small enough \mbox{(depending only on $\gamma$),} for any fixed $u \geq 0$ we have
\begin{equation}
\frac{|\bfR^k_n(u) \setminus \bfR^{k,\epsilon}_n(u)|}{\sqrt{n}}\overset{\bbP}{\longrightarrow} 0	
\end{equation} in the limit as $n \to \infty$, followed by $k \to \infty$ and finally $\epsilon \to 0$.
\end{prop}
Proposition~\ref{prop:rest2} will be a consequence of the (slightly) stronger Proposition~\ref{prop:rest3} below, which will be proved in Section~\ref{sec:prest}.

Finally, we introduce the versions of $\bfR^{k}_n(u)$ we mentioned in the beginning of the subsection. These will be, in fact, versions of $\bfR^{k,\epsilon}_n(u)$ in which we only look at $k$-boxes with centers in $\bbX_{k,l}(j)$ for some specific $j \in J_{k,l}$. To this end, fix $k \geq 1$ and, recalling the definition of $l$ and $\cT_k$ in~\eqref{eq:deflt}, for $n \geq l$, $j \in J_{k,l}$ and all parameters as before, define the set
\begin{equation}
\label{e:2.20a}
\begin{split}
	\bfW^{k,0}_n(u; j) := \Big \{ z \in \bfW_n(u)  :\:  \exists!\; x_j(z)\in \bbX_{k,l}(j) \text{ s.t. } & \rmQ(x_j(z)) \subseteq \rmQ(z;n-r_n) \\ & \rmW_n(u) \cap \rmQ(x_j(z);k)\neq \emptyset \Big\} 
\end{split}
\end{equation}  
together with, for $\epsilon \in (0,1)$,
\begin{equation}
\begin{split}
	\bfW^{k, \epsilon}_n(u; j) := \Big \{ z \in \bfW^{k,0}_n(u;j)  :\:  & \rmW_n(u) \cap \rmQ(x_j(z);k-\epsilon) \neq \emptyset \\ &
\rmW_n(u) \cap \rmQ(z;n-r_n) \subseteq \rmB(x_j(z);l) \\ & \rmW_n(u) \cap \rmQ(z;n-r_n) \subseteq \rmQ(z;n-r_n-\epsilon) \Big\} 
\end{split}
\end{equation}  
and 
\begin{equation}\begin{split}
\bfR_n^{k,\epsilon}(u; j) := \Big\{ z \in \bfW_n^{k,\epsilon}(u; j)  :\: 
& \sqrt{\wh{N}_t[x_j(z);i]} \in \cR_k^{\eta'}(i) \text{ for all }i \in \{1,\cT_k\} \\
& \sqrt{\wh{N}_t[x_j(z);i]} \in \cR_k^\eta(i) \text{ for all }i \in [2,\cT_k+5] \setminus \{\cT_k\}\Big\}.
\end{split}
\end{equation}

Note that $\bfW^{k,\epsilon}_n(u) \subseteq \bigcup_{j \in J_{k,l}} \bfW^{k,\epsilon}_n(u;j)$ by definition (and hence $\bfR^{k,\epsilon}_n(u) \subseteq \bigcup_{j \in J_{k,l}} \bfR^{k,\epsilon}_n(u;j)$) if $n$ is sufficiently large (depending only on $k$ and $\epsilon$), but that the latter set could be strictly~larger if a given cluster intersects more than one $k$-box but their centers belong to $\bbX_{k,l}(j)$ for different values of $j \in J_{k,l}$. 
 
The corresponding analog of $\mathcal{O}_{n,u}^{k}$ in \eqref{e:300.312bb} for $\bfR^{k,\epsilon}_n(u;j)$ is the empirical measure $\cO_{n,u}^{k,\epsilon}[j]$, defined as in~\eqref{e:300.312bb}, but only with $\bfR_n^{k,\epsilon}(u;j)$ used in place of $\bfR_n^{k}(u)$, i.e.
\begin{equation}
\cO^{k,\epsilon}_{n,u}[j] :=	\frac{1}{\big|\bfR_n^{k,\epsilon}(u;j)\big|}
	\sum_{z \in \bfR_n^{k,\epsilon}(u;j)}
	\delta_{\rme^{-n}z} \otimes \delta_{\cL_{t_n^A}(\rmQ(x_j(z);k))} \,.
\end{equation} These empirical measures have space marginals
\begin{equation}
\wh{\Theta}^{k,\epsilon}_{n,u}[j] :=	\frac{1}{\big|\bfR_n^{k,\epsilon}(u;j)\big|}
	\sum_{z \in \bfR_n^{k,\epsilon}(u;j)}
	\delta_{\rme^{-n}z} \,.
\end{equation}

Having defined these empirical measures, our first task will be to prove the following analog of Proposition~\ref{p:300.1} for each $\cO_{n,u}^{k,\epsilon}[j]$.

\begin{prop}
\label{p:300.31}
If $\gamma$ is small enough, then, given any $0 < \eta' < \eta< \tfrac{1}{2}$ small enough (depending only on $\gamma$), there exists for each $\epsilon \in (0,1)$, $u \geq 0$ and $k \geq 1$ a probability measure $\rmO^{k,\epsilon}_u$ on $\R^{\rmQ(k)}$ (depending on $\eta',\eta$) such that, for any fixed $\delta > 0$, we have that,  as $n \to \infty$ followed by $k \to \infty$, 
\begin{equation}
S_{n,u}^{k,\epsilon}(f_k;\delta):=\sup_{j \in J_{k,l}}\left[ \mathbf{1}_{\{ |J_{k,l}||\bfR^{k,\epsilon}_n(u;j)| > \delta \sqrt{n}\}}\left| \int f_k \, \rmd  \cO^{k,\epsilon}_{n,u}[j] - \int f_k \,\rmd \big( \wh{\Theta}^{k,\epsilon}_{n,u}[j] \otimes \rmO^{k,\epsilon}_u\big) \right| \right]\overset{\bbP}{\longrightarrow} 0
\end{equation} uniformly over all choices of measurable functions $f_k: \rmD \otimes \bbR^{\rmQ(k)} \to \R$ with $\|f_k\|_\infty\leq1$.
\end{prop}

To obtain Proposition~\ref{p:300.1} from its counterpart for each $\cO_{n,u}^{k,\epsilon}[j]$, it will be enough to show~that the difference between $\bfR^{k,\epsilon}_n(u)$ and $\bigcup_{j \in J_{k,l}} \bfR^{k,\epsilon}_n(u;j)$ is negligible, as given by the following result.

\begin{prop}
\label{p:300.32} If $\gamma$ is small enough, then, given any $0 < \eta' < \eta< \tfrac{1}{2}$ small enough (depending only on $\gamma$), for any fixed $\epsilon > 0$ and $u \geq 0$ we have  
\begin{equation}
\frac{\sum_{j \in J_{k,l}} \big|\bfR_n^{k,\epsilon}(u;j)\big|}{|\bfR^{k,\epsilon}_n(u)|} \overset{\bbP}{\longrightarrow} 1
\end{equation}
in the limit as $n \to \infty$ followed by $k \to \infty$.
\end{prop}

The above two results in combination with Proposition~\ref{prop:rest2} yield Proposition~\ref{p:300.1}.  

\begin{proof}[Proof of Proposition~\ref{p:300.1}] Given $\epsilon \in (0,1)$, consider the empirical measure
\begin{equation}
\cO^{k, \varepsilon}_{n,u}:=	\frac{1}{\big|\bfR_n^{k,\epsilon}(u)\big|}
	\sum_{z \in \bfR_n^{k,\epsilon}(u)} \delta_{\rme^{-n}z} \otimes 
	\delta_{\cL_{t_n^A}(\rmQ(x(z);k))}, 
\end{equation} with space marginal
\begin{equation}
\wh{\Theta}^{k, \varepsilon}_{n,u}:=	\frac{1}{\big|\bfR_n^{k,\epsilon}(u)\big|}
	\sum_{z \in \bfR_n^{k,\epsilon}(u)} \delta_{\rme^{-n}z}. 	
\end{equation}
Observe that, to obtain Proposition~\ref{p:300.1}, it will suffice to show that, for $\rmO^{k,\epsilon}_u$ as in Proposition~\ref{p:300.31},
	\begin{equation}\label{eq:convos2}
\frac{\big|\bfR_n^{k,\epsilon}(u)\big|}{\sqrt{n}}
\bigg|\int f_k\,\rmd \mathcal{O}^{k,\epsilon}_{n,u} - \int f_k \rmd (\wh{\Theta}^{k,\epsilon}_u \otimes\rmO^{k,\epsilon}_u)\bigg| \overset{\bbP} \longrightarrow \, 0 \,,
\end{equation} as $n \to \infty$ followed by $k \to \infty$, uniformly in all choices of measurable functions $f_k: \rmD \times \bbR^{\rmQ(k)} \to \bbR$ with $\|f_k\|_\infty \leq 1$. Indeed, if this is the case,  we may obtain the existence of a probability measure $\rmO^{k}_u$ satisfying \eqref{eq:convos} by \textcolor{black}{essentially repeating the argument in the proof of Proposition~\ref{p:300.1a}}, with the only difference that here one must use Proposition~\ref{prop:rest2} instead of Proposition~\ref{prop:rest} and replace the limit as $k \to \infty$ therein by the limit as $k \to \infty$ followed by $\epsilon \to 0$. We  omit the details (which are only technical) and focus instead on the proof of \eqref{eq:convos2}.

To this end, we notice that, by the triangle inequality,
\begin{equation}\label{eq:decompp4.7}
	\left| \int f_k \,\rmd \cO^{k,\epsilon}_{n,u} - \int f_k\, \rmd(\wh{\Theta}^{k, \varepsilon}_{n,u} \otimes \rmO^{k,\epsilon}_u)\right| \leq (a)+(b)+(c),
\end{equation} where
\begin{equation}
(a):=\left| \int f_k 	\,\rmd \cO^{k,\epsilon}_{n,u} - \frac{1}{|\bfR^{k,\epsilon}_n(u)|} \sum_{j \in J_{k,l}} \big|\bfR_n^{k,\epsilon}(u;j)\big| \int f_k \,\rmd \cO_{n,u}^{k,\epsilon}[j]\right|,
\end{equation}
\begin{equation}
(b):=\left| 	\frac{1}{|\bfR^{k,\epsilon}_n(u)|} \sum_{j \in J_{k,l}} \big|\bfR_n^{k,\epsilon}(u;j)\big| \left( \int f_k \,\rmd \cO_{n,u}^{k,\epsilon}[j] - \int f_k \rmd (\wh{\Theta}^{k, \varepsilon}_{n,u}\otimes \rmO^{k,\epsilon}_u)\right)\right|
\end{equation} and
\begin{equation}
(c):=	\left| \left(	\frac{\sum_{j \in J_{k,l}} \big|\bfR_n^{k,\epsilon}(u;j)\big|}{|\bfR^{k,\epsilon}_n(u)|}  -1\right)\int f_k \rmd (\wh{\Theta}^{k, \varepsilon}_{n,u}\otimes \rmO^{k,\epsilon}_u)\right|.
\end{equation}
Now, since $\|f_k\|_\infty \leq 1$, by a straightforward computation using that $\bfR^{k,\epsilon}_n(u) \subseteq \bigcup_{j \in J_{k,l}} \bfR^{k,\epsilon}_n(u;j)$ for all $n$ large enough, we have 
\begin{equation} \label{eq:boundmax}
\max\{(a),(c)\} \leq 	\left| \frac{\sum_{j \in J_{k,l}} \big|\bfR_n^{k,\epsilon}(u;j)\big|}{|\bfR^{k,\epsilon}_n(u)|}  -1\right|.
\end{equation} Therefore, by Proposition~\ref{p:300.32}, ($a$) and ($c$) both tend to zero in probability in the limit as $n \to \infty$ followed by $k \to \infty$. 
On the other hand, splitting the sum in ($b$) into two, depending on whether $j$ is such that $|J_{k,l}||\bfR^{k,\epsilon}_n(u;j)| > \delta \sqrt{n}$ or not, yields that
\begin{equation}\label{eq:ineq2}
(b) \leq 2\|f_k\|_\infty \delta \frac{\sqrt{n}}{|\bfR^{k,\epsilon}_n(u)|} + S^{k,\epsilon}_{n,u}(f_k;\delta) 
\frac{\sum_{j \in J_{k,l}} \big|\bfR_n^{k,\epsilon}(u;j)\big|}{|\bfR^{k,\epsilon}_n(u)|}.	
\end{equation} Thus, by the lower tightness of $|\bfR^{k,\epsilon}_n(u)|/\sqrt{n}$ (given by Propositions~\ref{prop:rest} and \ref{prop:rest2} in combination with Theorems~\ref{p:300.3} and \ref{t:2.1o}) and Propositions~\ref{p:300.31}--\ref{p:300.32}, by taking the limit as $n \to \infty$, followed by $k \to \infty$ and finally $\delta \to 0^+$ in \eqref{eq:ineq2}, we see that ($b$) tends to zero in probability~in~the limit as $n \to \infty$ followed by $k \to \infty$ (uniformly in all choices of $f_k$ with $\| f_k\|_\infty \leq 1$) which, in combination with \eqref{eq:boundmax} and upon recalling \eqref{eq:decompp4.7}, immediately implies \eqref{eq:convos2} as desired.
\end{proof}

The remainder of Section~\ref{sec:iid} is now devoted to the proofs of Proposition~\ref{p:300.31} and Proposition~\ref{p:300.32} (Proposition~\ref{prop:rest} and Proposition~\ref{prop:rest2} are proved in Section~\ref{sec:prest}). Before we begin with the proofs, we present some preliminary results to be used in these.

\subsubsection{Preliminaries}\label{sec:prelimiid}

We state here some auxiliary results we shall need to prove Proposition~\ref{p:300.31} and Proposition~\ref{p:300.32}. Their proofs are relegated to Section~\ref{sec:auxproof}. 

The first auxiliary result gives tail estimates for the number of downcrossings at scale $k+ik^\gamma$, conditional on the number of downcrossings at some larger scale $k+jk^\gamma$ with $i <j$.

\begin{prop}\label{prop:GR-1} There exists a constant $c>0$ such that, if $k$ is large enough, for any $i,j \in \N$ with $i<j$, $n \geq k$ and $x,\wt{x} \in \rmD_n$ such that $x \in \rmB^-[\wt{x};j-1]$ and $\overline{\rmB^+[\wt{x};j]}\subseteq \rmD_n$, we have that 
\begin{equation}\label{eq:GR-1}
\bbP\Big( \sqrt{\wh{N}_t[x;i]} \leq u \,\Big|\, \sqrt{\wh{N}_t[\wt{x};j]} \geq v\,;\cF[\wt{x};j]\Big) \leq 2 \exp\bigg\{ -\frac{(v-u)^2}{(j-i+1)k^\gamma}\bigg\}
\end{equation} uniformly over all $t > 0$ and $u,v \in \cN_k$ satisfying that $u\leq v$,  $0 \leq \wt{u}_k \leq \rme^{ck^\gamma}$ and $1 \leq \wt{v}_k \leq \rme^{ck^\gamma}$, as well as  
\begin{equation}
\bbP\Big( \sqrt{\wh{N}_t[x;i]} \geq u \,\Big|\, \sqrt{\wh{N}_t[\wt{x};j]} \leq v\,;\cF[\wt{x};j]\Big) \leq 2 \exp\bigg\{ -\frac{(u-v)^2}{(j-i+1)k^\gamma}\bigg\}	
\end{equation} uniformly over all $t > 0$ and $u,v \in \cN_k$ satisfying that $v \leq u$,  $0 \leq \wt{u}_k \leq \rme^{ck^\gamma}$ and $1 \leq \wt{v}_k \leq \rme^{ck^\gamma}$.
\end{prop} 		

The second auxiliary result is an estimate on the probability of having a low local time vertex in a ball of log-radius $k+(j-2)k^\gamma$, given proper downcrossing repulsion at scale $k+jk^\gamma$.

\begin{prop}\label{prop:LT-1} Given $\eta \in (0,\frac{1}{2}-\gamma)$, for each $\theta \geq 0$ there exists a constant $C=C(\theta,\gamma) > 0$ such~that, if $k$ is sufficiently large, given any integer $j > 1$, $n \geq k$ and $\wt{x} \in \rmD_n$ with $\overline{\rmB^+[\wt{x};j]}\subseteq \rmD_n$, we have that 
	\begin{equation}
	\label{eq:lltb}
 \rme^{-2\sqrt{2}\wh{v}-\frac{\wh{v}^2}{k+jk^\gamma}-Ck^\gamma} \leq \bbP \bigg( \min_{x \in \rmB^-[\wt{x};j-1]}L_t(x) \leq \theta \,\Big|\, \sqrt{\wh{N}_t[\wt{x};j]} \geq v \,; 
\cF[\wt{x};j]\bigg) \leq  \rme^{- 2\sqrt{2}\wh{v}+Ck^\gamma}
\end{equation} uniformly over all $t >0$ and $v \in \cR^\eta_k(j)$, where above we abbreviate $\wh{v}:=v-\sqrt{2}(k+(j-1)k^\gamma)$.
\end{prop}

Our third auxiliary result is an upper bound on the probability that two far apart vertices~in a certain ball of log-radius $k+(j-2)k^\gamma$ have low local time, given proper downcrossing repulsion at scale $k+jk^\gamma$. 
\begin{prop}
\label{prop:LT-2} Given $c,\theta>0$, there exists $C=C(c,\theta,\gamma) > 0$ such that, if $k$ is large enough, given any integer $j>1$, $n \geq k$ and $\wt{x} \in \rmD_n$ such that $\ol{\rmB^+[\wt{x};j]}\subseteq \rmD_n$, for all $x,x' \in \rmB^-[\wt{x};j-1]$ with $\log \| x-x'\| \geq k+(j-2-c)k^\gamma$ we have
\begin{equation}\label{eq:4.6bc}
	\bbP \Big( \max\{L_t(x), L_t(x')\} \leq \theta \,\Big|\, \sqrt{\wh{N}_t[\wt{x};j]} \geq v \,; 
\cF[\wt{x};j]\Big) \\ \leq  \rme^{ - 4(k+jk^\gamma) -4\sqrt{2}\wh{v}  + Ck^\gamma}.
\end{equation}
uniformly over all $t >0$ and $v \geq \sqrt{2}(k+(j-1)k^\gamma)$.
\end{prop}

We shall also require a generalization of the so-called Resampling Lemma found in~\cite{Tightness}, namely Lemma~4.4 of said reference. We present this generalization next.

\begin{lem}[Resampling Lemma]\label{lem:resampling} Given $n \geq k \geq 1$, $u \geq 0$ and $i \in \N_{\geq 2}$, consider a nonempty Borel subset $M^{i}_{k,n} \subseteq \cN_k$ and a family $(f^{i}_{n,u,y,v} :  y \in \bbX_{k+(i-2)k^\gamma},v \in \cN_k)$ of measurable functions $f^{i}_{n,u,y,v}: \bbR^{\ol{\rmB^+[y;i-1]}} \to \bbR$. Then, if for each  $z \in \bfW_n(u)$ we define
\begin{equation}
\bbY^i_n(u;z):=\Big\{	y \in \bbX_{k+(i-2)k^\gamma} :\: \rmW_n(u) \cap \rmQ(z;n-r_n) \subseteq \rmB^-[y;i-1] \,,\, 
 \sqrt{\wh{N}_{t}[y;i]} \in M^i_{k,n}\Big\},
\end{equation}
the quantity $\bbZ^i_n(u)$ given by  
\begin{equation}\label{eq:sumz}
\bbZ_n^i(u):= \sum_{z \in \bfW_n(u) \atop \bbY^i_n(u;z) \neq \emptyset}  \max_{y \in \bbY^i_n(u;z)} f^i_{n,u,y,\sqrt{\wh{N}_t[y;i]}} \big(L_t\big(\ol{\rmB^+[y;i-1]}\big)\big).	
\end{equation} satisfies, for any $\kappa,t > 0$ and all $k$ large enough (depending only on $\gamma$),
\begin{equation}\label{eq:resamplebound1}
\bbP( |\bbZ^i_n(u)| > \kappa |\bfW_n(u)|) \leq C \kappa^{-1}\rme^{4k^\gamma}b_{k,n}^{(i)}(u),	
\end{equation}where $C > 0$ is some universal constant and  
\begin{equation}\label{eq:defbn.0}
b_{k,n}^{(i)}(u):= \sup_{y,\rho} \E \left( \Big|f^i_{n,u,y,v}\big(L_t\big(\ol{\rmB^+[y;i-1]}\big)\big)\Big|\,\Big|\,  Z_{t^A_n}[y;i]=\rho\,,\rmW_n(u) \cap \rmB^-[y;i-1] \neq \emptyset \right),
\end{equation} with the supremum being over the set of all $y \in \bbX_{k+(i-2)k^\gamma}$ with $\overline{\rmB^+[y;i]} \subseteq \rmD_n$ (so that $N_{t}[y;i]$ is well-defined) and all $\rho:=(v,\overline{z}) \in \cZ_k(i)$ with $v \in M^i_{k,n}$ (with the convention that $b_{n,k}^{(i)}(u):=0$ if this set is empty, in which case $\bbZ_n^i(u) = 0$ since the sum in \eqref{eq:sumz} is also over an empty set). 
In particular, given a nonempty subset $\cC_{k,n} \subseteq \bbN_{\geq 2}$ of indices and $\eta \in (0,\frac{1}{2}-\gamma)$, if 
\begin{equation}
 \limsup_{k \to \infty} \limsup_{n \to \infty} \left[\sup_{i \in \cC_{k,n}} b_{k,n}^{(i)}(u) \rme^{2(k+ik^\gamma)^{\frac{1}{2}-\eta}}\right] \leq 1,	
\end{equation} then 
\begin{equation}\label{eq:resamplebound2}
	\lim_{k \to \infty} \limsup_{n \to \infty} \bbP\left( \exists i \in \cC_{k,n} : |\bbZ^{i}_{n}(u)| > \rme^{-(k+ik^\gamma)^{\frac{1}{2}-\eta}}\right) = 0.
\end{equation}
\end{lem}

Finally, in the proofs of Proposition~\ref{prop:rest3} and Proposition~\ref{p:300.32} below, we will need to compare the number of $[x;i]$-downcrossings for some $x \in \bbX_k$ with those of a nearby vertex $y \in \bbX_{k+(i-2)k^\gamma}$. The precise result we will need is the following lemma. 

\begin{lem}\label{lem:incball.1} If $k$ is large enough (depending only on $\gamma$), then, given $n\geq 0$, $i \in \bbN_{\geq 2}$ and~$x \in \bbX_k$ such that $\ol{\rmB^-[x;i+1] }\subseteq \rmD_n$, there exists $y \in \bbX_{k+(i-2)k^\gamma}$ satisfying $\ol{\rmB(y;k+ik^\gamma -3\rme^{-k^\gamma})} \subseteq \rmD_n$ such that $\rmB^-[x;\max\{i-2,1\}] \subseteq \rmB^-[y;i-1]$, and, for all $t > 0$,
\begin{equation}
\wh{N}^{\rm{in}}_t[y;i] \geq \wh{N}_{t}[x;i]	 \geq \wh{N}^{\rm{out}}_t[y;i],
\end{equation}  
 where 
\begin{equation}
\wh{N}^{\rm{in}}_t[y;i]:=	k^\gamma N_{t}(y; k+(i-1)k^\gamma +\rme^{-k^\gamma},k+ik^\gamma -5\rme^{-k^\gamma})
\end{equation}
and 
\begin{equation}
\wh{N}^{\rm{out}}_t[y;i]:= k^\gamma N_{t}(y; k+(i-1)k^\gamma -\rme^{-k^\gamma},k+ik^\gamma -3\rme^{-k^\gamma}).
\end{equation}
\end{lem}

\begin{rem}\label{rem:ninout} We mention that the analogues of Propositions~\ref{prop:GR-1}--\ref{prop:LT-1}--\ref{prop:LT-2} and Lemma~\ref{lem:resampling} also hold for both $\wh{N}^{\rm{in}}_t[x;i]$ and $\wh{N}^{\rm{out}}_t[x;i]$ in place of $\wh{N}_t[x;i]$. Indeed, this is due to the fact that, in all three cases, we are dealing with downcrossing counts of the form $\wh{N}_t(x;a,b)=k^\gamma N_t(x;a,b)$ for some $a,b$ satisfying $b-a=k^\gamma +O(\rme^{-k^\gamma})$, which is the only input necessary for the proofs of all these results.	
\end{rem}

\subsubsection{Convergence of the empirical measure over subgrids: proof of Proposition~\ref{p:300.31}}

We continue with the proof of Proposition~\ref{p:300.31}. As discussed in the proof outline in Section~\ref{sec:exist1},
the strategy of proof will consist of showing first that, given a $k$-box $\rmQ(x;k)$, then,  conditionally on the event that the downcrossing trajectory around $x$ is repelled from scale $k$ up to $l$, the law of the local time field on $\rmQ(x;k)$ is insensitive to the precise downcrossing information at scale~$l$. Then, since for every $j \in J_{k,l}$ the $k$-boxes centered in $\bbX_{k,l}(j)$ are all at least $l$ log-distance apart, by conditioning on the downcrossing information at scale $l$ around each $x_j(z)$ with $z \in \bfR^{k,\epsilon}_n(u)$, the spatial Markov property will imply that the fields $L_{t^A_n}(Q(x(z);k))$ are also independent, from where Proposition~\ref{p:300.31} will follow by (essentially) an application of the law of large numbers.

To make the argument precise, given $u \geq 0$, $\epsilon \in (0,1)$, $0< \eta' < \eta < \frac{1}{2}$, $n \geq k \geq 1$ and $x \in \rmD_n$ with $\overline{\rmB^+[x;\cT_k]} \subseteq \rmD_n$, with $\cT_k$ as in~\eqref{eq:deflt}, recalling the notation from Section~\ref{sec:downprelim} let us define the events
\begin{equation}
A^{k,\epsilon}_{t,u}(x):=\left\{ 	\min_{y \in \rmQ(x;k-\epsilon)}L_t(y) \leq u \right\},	
\end{equation}
\begin{equation}
\cE^{k,\epsilon,\eta',\eta}_{t,u}(x):= A^{k,\epsilon}_{t,u}(x) \setminus \left( \mathrm{NR}_{t}^{\eta',\{1\}}(x) \cup \mathrm{NR}_{t}^{\eta,[2,\cT_k-1]}(x) \right)
\end{equation} together with, for $z \in \cZ_k^{\eta'}[\cT_k]$, the probability distribution $\rmM^{k,\epsilon}_u(z,\cdot)$ on $\bbR^{\rmQ(k)}$ given by
\begin{equation}\label{eq:defm}
\rmM_{u}^{k,\epsilon}(z,\cdot):= \bbP\Big( L_t(\rmQ(x;k)) \in \cdot \,\Big|\, \cE^{k,\epsilon,\eta',\eta}_{t,u}(x) \cap \{Z_t[x;\cT_k]=z\}\Big),
\end{equation} where we omit the dependence on $\eta,\eta'$ from the notation from simplicity. Notice that $
\rmM_{u}^{k,\epsilon}(z,\cdot)$ does not depend on the choice of $x$ nor $t$. The main ingredient for the proof of Proposition~\ref{p:300.31} is the following result, which shows that  $\rmM_{u}^{k,\epsilon}(z,\cdot)$ is insensitive to the precise value of $z \in \cZ_k^{\eta'}[\cT_k]$. 

\begin{lem}\label{l:4.10} If $\gamma$ is small enough, then, given any $0 < \eta' < \eta < \tfrac{1}{2}$ small enough (depending only on $\gamma$), for each $u \geq 0$, $\varepsilon \in (0,1)$ and $k \geq 1$ there exists a probability measure $\rmO^{k,\epsilon}_u$ on $\bbR^{\rmQ(k)}$  such that 
\begin{equation}
\lim_{k \to \infty} \left[ \sup_{z \in \cZ^{\eta'}_k[\cT_k] }\Bigg\| \frac{\rmd \rmM^{k,\epsilon}_{u}(z,\cdot)}{\rmd \rmO_{u}^{k,\epsilon}} - 1 \Bigg\|_\infty\right]=0.
\end{equation} 
\end{lem}

At this point in the argument is where the sharp ballot estimates from Theorem~\ref{prop:coupling0} become essential, as they constitute the key element in the proof of Lemma~\ref{l:4.10}.

\begin{proof}[Proof of Lemma~\ref{l:4.10}] Given a Borel subset $B \subseteq \bbR^{\rmQ(k)}$, we may decompose
\begin{equation}\label{eq:div1}
\rmM^{k,\epsilon}_{u}(z,B)= \frac{\sum_{z' \in \cZ_k^{\eta'}[1]} \bbP\Big( \{ L_t(\rmQ(x;k)) \in B \,,\,Z_t[x;1]=z'\} \cap \cE^{k,\epsilon,\eta',\eta}_{t,u}(x) \,\Big|\,Z_t[x;\cT_k]=z\Big)}{\bbP\Big(\cE^{k,\epsilon,\eta',\eta}_{t,u}(x)\,\Big|\,Z_t[x;\cT_k]=z\Big)}	
\end{equation}
where, since $\ol{\rmQ(x;k)} \subseteq \rmB^-[x;1]$, each summand in the numerator can be further written as
\begin{equation}
	\bbP\Big(\{L_t(\rmQ(x;k)) \in B\} \cap A^{k,\epsilon}_{t,u}(x)\,\Big|\, Z_t[x;1]=z'\Big) \bbP\Big( \{ Z_t[x;1]=z' \} \setminus \mathrm{NR}_{t}^{\eta,[2,\cT_k-1]}(x)\,\Big|\,Z_t[x;\cT_k]=z\Big).	
\end{equation}
Now, by Theorem~\ref{prop:coupling0}, if $\gamma$ is chosen small enough, then, by taking $\eta'<\eta$ appropriately small, we have that there exist measurable functions $f_k,g_{k,\cT_k}:\R \to \R$ such that, for $k$ large enough,
\begin{equation}\label{eq:div2}
	\bbP\Big( \{ Z_t[x;1]=z' \} \setminus \mathrm{NR}_{t}^{\eta,[2,\cT_k-1]}(x)\,\Big|\,Z_t[x;\cT_k]=z\Big)=f_k(z')g_{k,\cT_k}(z)(1+o_k(1))
\end{equation} holds uniformly over all $z' \in \cZ_k^{\eta'}[1]$, $z \in \cZ_k^{\eta'}[\cT_k]$. In particular, \eqref{eq:div2} implies that the numerator on the right-hand side of~\eqref{eq:div1} can be written as
\begin{equation}\label{eq:div3}
	\left(\sum_{z' \in \cZ_k^{\eta'}[1]} \bbP\Big(\{ L_t(\rmQ(x;k)) \in B\} \cap A^{k,\epsilon}_{t,u}(x)\,\Big|\, Z_t[x;1]=z'\Big) f_k(z')\right)g_{k,\cT_k}(z)(1+o_k(1)).
\end{equation}
 On the other hand, upon setting $B=\bbR^{\rmQ(k)}$ in \eqref{eq:div3} above, we obtain that also the denominator in the right-hand side of~\eqref{eq:div1} can be written in a similar fashion, i.e., as
 \begin{equation}\label{eq:div4}
	\left(\sum_{z' \in \cZ_k^{\eta'}[1]} \bbP\Big(A^{k,\epsilon}_{t,u}(x)\,\Big|\, Z_t[x;1]=z'\Big) f_k(z')\right)g_{k,\cT_k}(z)(1+o_k(1)).
\end{equation} Hence, if we define the probability measure $\rmO_u^{k,\epsilon}$ via the formula
\begin{equation}
\rmO_u^{k,\epsilon}(\cdot):= \sum_{z' \in \cZ_k^{\eta'}[1]} \bbP\Big( L_t(\rmQ(x;k)) \in \cdot \,\Big|\, A^{k,\epsilon}_{t,u}(x) \cap \{Z_t[x;1]=z'\}\Big) p_k(z'), 
\end{equation} where $p_k=(p_k(z') : z' \in \cZ^{\eta'}_k[1])$ is the probability vector given by
\begin{equation}
p_k(z'):=\frac{\bbP\Big(A^{k,\epsilon}_{t,u}(x)\,\Big|\, Z_t[x;1]=z'\Big)f_k(z')}{\sum_{w \in \cZ_k^{\eta'}[1]} \bbP\Big(A^{k,\epsilon}_{t,u}(x)\,\Big|\, Z_t[x;1]=w\Big)f_k(w)},
\end{equation} then, upon combining \eqref{eq:div3}--\eqref{eq:div4}, we conclude that
\begin{equation}
\rmM_{u}^{k,\epsilon}(z,B)= \rmO_{u}^{k,\epsilon}(B)(1+o_k(1))
\end{equation} holds uniformly over all Borel subsets $B \subseteq \bbR^{\rmQ(k)}$ and $z \in \cZ_k^{\eta'}[\cT_k]$, so the result now follows.
\end{proof}

We can now give the proof of Proposition~\ref{p:300.31}. 

\begin{proof}[Proof of Proposition~\ref{p:300.31}] By the union bound, it will suffice to show that, for any $\delta,\rho > 0$, 
\begin{equation}\label{eq:ob1}
\lim_{k \to \infty} \limsup_{n \to \infty} \sum_{j \in J_{k,l}} \bbP\left(	\left| \int f_k \, \rmd\cO^{k,\epsilon}_{n,u}[j] - \int f_k \,\rmd (\wh{\Theta}^{k,\epsilon}_{n,u}[j]\otimes \rmO^{k,\epsilon}_u)\right| > \rho\,,\, |\bfR^{k,\epsilon}_n(u;j)| > \frac{\delta}{|J_{k,l}|} \sqrt{n}\right) =0
\end{equation} uniformly over all sequences $(f_k)_{k \in \N}$ of measurable functions $f_k: \rmD \times\bbR^{\rmQ(k)} \to \R$ with $\|f_k\| \leq 1$.
To this end, we notice that, if we define the probability measure $\cM^{k,\epsilon}_{n,u}[j]$ by
\begin{equation}
	\cM^{k,\epsilon}_{n,u}[j]:=\frac{1}{|\bfR^{k,\epsilon}_n(u;j)|} \sum_{z \in \bfR^{k,\epsilon}_n(u;j)} \delta_{\rme^{-n}z} \otimes \rmM^{k,\epsilon}_u(Z_{t^A_n}[x_j(z);\cT_k],\cdot),
\end{equation} then we may decompose
\begin{equation}
	\int f_k \, \rmd\cO^{k,\epsilon}_{n,u}[j]- \int f_k \,\rmd (\wh{\Theta}^{k,\epsilon}_{n,u}[j]\otimes \rmO^{k,\epsilon}_u)= (a)+(b),
\end{equation} where  
\begin{equation}\label{eq:defexp1}
\begin{split}
	(a)&= \int f_k \, \rmd\cO^{k,\epsilon}_{n,u}[j] - \int f_k \,\rmd \cM^{k,\epsilon}_{n,u}[j]\\
	&=\frac{1}{|\bfR^{k,\epsilon}_n(u;j)|} \sum_{z \in \bfR^{k,\epsilon}_n(u;j)} \Big(f_k(\rme^{-n}z,L_{t^A_n}(\rmQ(x_j(z);k))-\int f_k(\rme^{-n}z,\cdot) \, \rmd \rmM^{k,\epsilon}_u(Z_{t^A_n}[x_j(z);\cT_k],\cdot)\Big)	
\end{split}
\end{equation}
and 
\begin{equation}\label{eq:defexp2}
	(b)=\frac{1}{|\bfR^{k,\epsilon}_n(u;j)|} \sum_{z \in \bfR^{k,\epsilon}_n(u;j)}\Big( \int f_k(\rme^{-n}z,\cdot) \, \rmd \rmM^{k,\epsilon}_u(Z_{t^A_n}[x_j(z);\cT_k],\cdot)-\int f_k \, \rmd (\wh{\Theta}^{k,\epsilon}_{n,u}[j]\otimes \rmO^{k,\epsilon}_u)\Big). 
\end{equation} 
Since $Z_{t^A_n}[x_j(z);\cT_k] \in \cZ_k^{\eta'}[\cT_k]$ for all $z \in \bfR_n^{k,\epsilon}(u;j)$, if $k$ is sufficiently large, for each $z \in \bfR_n^{k,\epsilon}(u;j)$ we have 
\begin{equation}
\left|\int f_k (\rme^{-n}z,\cdot) \, \rmd \rmM^{k,\epsilon}_u(Z_{t^A_n}[x(z);\cT_k],\cdot)-\int f_k (\rme^{-n}z,\cdot) \, \rmd \rmO^{k,\epsilon}_u\right| \leq \sup_{\ol{z} \in \cZ^{\eta'}_k[\cT_k] }\Bigg\| \frac{\rmd \rmM^{k,\epsilon}_{u}(\ol{z},\cdot)}{\rmd \rmO_{u}^{k,\epsilon}} - 1 \Bigg\|_\infty,
\end{equation} where, in order to obtain this inequality, we have used that $\|f_k\|_\infty \leq 1$ and that $\rmM^{k,\epsilon}_{u}(\overline{z},\cdot) \ll \rmO_{u}^{k,\epsilon}$ for every $\ol{z} \in \cZ_k^{\eta'}[\cT_k]$ if $k$ is sufficiently large by Lemma~\ref{l:4.10}. It then follows from Lemma~\ref{l:4.10} that the absolute value of the term ($b$) in \eqref{eq:defexp2} can be made smaller than $\frac{\rho}{2}$ uniformly in~$j \in J_{k,l}$ and the choice of sequence $(f_k)_{k \in \bbN}$ by choosing $k$ sufficiently large, so that, to conclude \eqref{eq:ob1}, it will suffice to show that
\begin{equation}\label{eq:ob2}
\lim_{k \to \infty} \limsup_{n \to \infty} \sum_{j \in J_{k,l}} \bbP\left(	\left| \int f_k \, \rmd\cO^{k,\epsilon}_{n,u}[j] - \int f_k \,\rmd \cM^{k,\epsilon}_{n,u}[j]\right| > \frac{\rho}{2}\,,\, |\bfR^{k,\epsilon}_n(u;j)| > \frac{\delta}{|J_{k,l}|} \sqrt{n}\right) =0
\end{equation} uniformly over all sequences $(f_k)_{k \in \N}$ of measurable functions $f_k: \rmD \times \bbR^{\rmQ(k)} \to \R$ with $\|f_k\| \leq 1$.
To this end, for $n$ large enough so that $\rme^{7l-(n-r_n)} < 1-\rme^{-\epsilon}$ and $j \in J_{k,l}$, define
\begin{equation}
\begin{split}
\bbX_{k,l}^{(n)}[j]:= \{ x \in \bbX_{k,l}(j) : \exists z \in \bbX_{ n-r_n } \text{ s.t. }& \rmQ(z; n - r_n ) \cap \rmD_n^\circ \neq \emptyset,\\ &\ol{\rmB^+[x;\cT_k+5]} \subseteq \rmQ(z;n - r_n ) \}
\end{split}
\end{equation} as well as the random subset $\mathfrak{X}^{k,\epsilon}_n(u;j) \subseteq \bbX_{k,l}^{(n)}[j]$ by the formula
\begin{equation}
\mathfrak{X}^{k,\epsilon}_n(u;j):=\{ x_j(z) : z \in \bfR^{k,\epsilon}_n(u;j)\}.
\end{equation} Note that, given $C \subseteq \bbX^{(n)}_{k,l}[j]$, since the balls $(\ol{\rmB^+[x;\cT_k]})_{x \in \bbX_{k,l}^{(n)}[j]}$ are disjoint by definition of $\bbX_{k,l}$, using the fact that by choice of $n$ we have $\ol{\rmB^+[x_j(z);\cT_k+5]} \subseteq \rmQ(z;n-r_n)$ for each $z \in \bfR^{k,\epsilon}_n(u;j)$ and that the boxes $(\rmQ(z;n-r_n))_{z \in \bbX_{n-r_n}}$ are disjoint by definition, %by definition of $\bbX_{k,l}^{(n)}[j]$, using the fact that 
%\begin{equation}
%\Bigg(\bigcup_{x_j(z) \in \mathfrak{X}^{k,\epsilon}_n(u;j)} (\rmQ(z;n-r_n) \setminus \rmB(x_j(z);l)\Bigg) \cap \Bigg(\bigcup_{y_j(z) \in \mathfrak{X}^{k,\epsilon}_n(u;j)} \ol{\rmB^+[y_j(z);T]}\Bigg)=\emptyset
%\end{equation}
%for all $k$ large enough by construction, 
we can write
\begin{equation}
	\{\mathfrak{X}^{k,\epsilon}_n(u;j) = C\} = \Lambda_C \cap \bigcap_{x \in C} \cE^{k,\epsilon,\eta',\eta}_{t^A_n,u}(x),
\end{equation} where $\Lambda_C$ is an event in $\cap_{x \in C} \cF[x;\cT_k]$ and, conditional on the $\sigma$-algebra $ \cap_{x \in C} \cF[x;\cT_k]$, the~events $(\cE^{k,\epsilon,\eta',\eta}_{t^A_n,u}(x))_{x \in C}$ are independent.  Thus, by conditioning on $\cap_{x \in C} \cF[x;\cT_k]$, a simple computation using that all $[x;\cT_k]$-excursions with $x$ ranging over $C$ are independent under this conditioning shows that, for any given choice of subset $C \subseteq \bbX^{(n)}_{k,l}[j]$ with $\bbP(\mathfrak{X}^{k,\epsilon}_n(u;j) = C)>0$,  any collection of Borel subsets $(B_x)_{x \in C} \subseteq \bbR^{\rmQ(k)}$ and any sequence $(z_x)_{x \in C} \subseteq \cZ^{\eta'}_k[T]$, we have that
\begin{equation}\label{eq:mdecomp}
\bbP( \cap_{x \in C} \{ L_t ( \rmQ(x;k)) \in B_x \}  \,|\, \cap_{x \in C} \{Z_t[x;T]=z_x\} \cap \{\mathfrak{X}^{k,\epsilon}_n(u;j) = C\}) =\prod_{x \in C} M^{k,\epsilon}_u(z_x,B_x). 	
\end{equation} 
Put in other words, conditionally on $\mathfrak{X}_n^{k,\epsilon}(u;j)$ and $(Z_{t^A_n}[x;\cT_k] : x \in \mathfrak{X}_n^{k,\epsilon}(u;j))$, the random fields $\big(L_{t^A_n}(\rmQ(x;k)) : x \in \mathfrak{X}_{n}^{k,\epsilon}(u;j)\big)$ are independent, and each having distribution $M^{k,\epsilon}_u(Z_{t^A_n}[x;\cT_k],\cdot)$. In particular, since $\bfR^{k,\epsilon}_n(u;j)$ is a measurable function of $\mathfrak{X}^{k,\epsilon}_n(u;j)$, conditionally on $\mathfrak{X}_n^{k,\epsilon}(u;j)$ and $(Z_{t^A_n}[x;\cT_k] : x \in \mathfrak{X}_n^{k,\epsilon}(u;j))$, all the summands in the second line of~\eqref{eq:defexp1} are independent, uniformly bounded (by $2$) and have zero mean, which, by Azuma-Hoeffding's inequality, implies that
\begin{equation}
\begin{split}
\bbP\Bigg(	\left| \int f_k \, \rmd\cO^{k,\epsilon}_{n,u}[j] - \int f_k \,\rmd \cM^{k,\epsilon}_{n,u}[j]\right| > \frac{\rho}{2}\,&,\, |\bfR^{k,\epsilon}_n(u;j)| > \frac{\delta}{|J_{k,l}|} \sqrt{n} \,\Bigg|\,\mathcal{G}_n^{k,\epsilon}(u;j)\Bigg) \\ & \leq \mathbf{1}_{\{ |\bfR^{k,\epsilon}_n(u;j)| > \frac{\delta}{|J_{k,l}|} \sqrt{n}\}} \rme^{- \frac{\rho^2}{34} |\bfR^{k,\epsilon}_n(u;j)|} \leq \rme^{-\frac{\delta\rho^2}{34|J_{k,l}|} \sqrt{n}}
\end{split}
\end{equation} uniformly over all $j \in J_{k,l}$, where $\mathcal{G}_n^{k,\epsilon}(u;j)$ above denotes the $\sigma$-algebra generated by $\mathfrak{X}_n^{k,\epsilon}(u;j)$ and $(Z_{t^A_n}[x;\cT_k] : x \in \mathfrak{X}_n^{k,\epsilon}(u;j))$. Upon taking expectations on the above inequality, it follows that the sum in \eqref{eq:ob2} is bounded from above by $|J_{k,l}|\rme^{-\frac{\delta\rho^2}{34|J_{k,l}|} \sqrt{n}}$, from where \eqref{eq:ob2} (and, hence, the entire result) immediately follows.
\end{proof}

%Observe that the set $\bfR^{k,\epsilon}_n(u;j)$ can be completely reconstructed by looking at the trajectory of the random walk only when it is performing an $[x;T]$-excursion for some $x \in \bbX_{k,l}^{(n)}[j]$. 
%Moreover, since the closed balls $(\ol{\rmB^+[x;T]})_{x \in \bbX_{k,l}^{(n)}[j]}$ are disjoint by construction of $\bbX_{k,l}$, conditionally on the $\sigma$-algebra $\cap_{x \in \bbX_{k,l}^{(n)}[j]} \cF[x;T]$, all $[x;T]$-excursions of the walk with $x$ ranging through all of $\bbX_{k,l}^{(n)}[j]$ are independent. In particular, if we define
%\begin{equation}
%\mathfrak{X}^{k,\epsilon}_n(u;j):=\{ x(z) : z \in \bfR^{k,\epsilon}_n(u;j)\},	
%\end{equation} then, by conditioning on $\cap_{x \in \bbX_{k,l}^{(n)}[j]} \cF[x;T]$, a simple computation shows that, given $C \subseteq \bbX^{(n)}_{k,l}[j]$, any collection of Borel subsets $(B_x)_{x \in C} \subseteq \bbR^{\rmQ(k)}$ and a sequence $(z_x)_{x \in C} \subseteq \cZ_k[T]$, we have that
%\begin{equation}\label{eq:mdecomp}
%\bbP( \cap_{x \in C} \{ L_t ( \rmQ(x;k)) \in B_x \}  \,|\, \cap_{x \in C} \{Z_t[x;T]=z_x\} \cap \{\mathfrak{X}^{k,\epsilon}_n(u;j) = C\}) =\prod_{x \in C} M^{k,\epsilon,\eta',\eta}_u(z_x,B_x). 	
%\end{equation} 

\subsubsection{Overcounting over subgrids is negligible: proof of Proposition~\ref{p:300.32}}

We conclude Section~\ref{sec:exist0} by giving the proof of our last auxiliary result, namely Proposition~\ref{p:300.32}. This result states that the difference between $|\bfR^{k,\epsilon}_n(u)|$ and $\sum_{j \in J_{k,l}} |\bfR^{k,\epsilon}
_n(u;j)|$ is asymptotically negligible. To prove this, we will show that, if $|\bfR^{k,\epsilon}_n(u)| $ and $\sum_{j \in J_{k,l}} |\bfR^{k,\epsilon}
_n(u;j)|$ differ, it is because there exist clusters centered at some $z \in \bbX_{n-r_n}$ belonging to more than one $\bfR^{k,\epsilon}_n(u;j)$. However, for this to occur, these clusters would need to have an unusually large diameter. Therefore, since we are considering only repelled clusters, their number becomes negligible in the limit and thus so does the difference between $|\bfR^{k,\epsilon}_n(u)|$ and $\sum_{j \in J_{k,l}} |\bfR^{k,\epsilon}
_n(u;j)|$. We give the rigorous argument below, which is in the spirit of \cite[Lemma~5.6]{Tightness}, albeit more delicate. 

\begin{proof}[Proof of Proposition~\ref{p:300.32}]

Notice that, since $\bfR^{k,\epsilon}_n(u) \subseteqq \bigcup_{j \in J_{k,l}} \bfR^{k,\epsilon}
_n(u;j)$ for all $n$ sufficiently~large by definition, we have 
\begin{equation}
\frac{\sum_{j \in J_{k,l}} \big|\bfR_n^{k,\epsilon}(u;j)\big|}{|\bfR^{k,\epsilon}_n(u)|} \geq 1
\end{equation} almost surely, so that, by the lower tightness of $|\bfR^{k,\epsilon}_n(u)|/\sqrt{n}$ (implied by Propositions~\ref{prop:rest}--\ref{prop:rest2} in combination with Theorems~\ref{p:300.3}--\ref{t:2.1o}), it will suffice to prove that, for each $\delta > 0$,
\begin{equation}\label{eq:l-ob1}
\lim_{k \to \infty} \limsup_{n \to \infty}\bbP\left( \sum_{j \in J_{k,l}} \big|\bfR_n^{k,\epsilon}(u;j)\big| - |\bfR^{k,\epsilon}_n(u)| > \delta \sqrt{n}\right) = 0.	
\end{equation}
To this end, notice that, if we define the set
\begin{equation}
\wh{\bfR}^{k,\epsilon}_n(u):=\bigg\{ z \in 	\bigcup_{j \in J_{k,l}} \bfR^{k,\epsilon}_{n}(u;j) : |I_{k,l}(z)| > 1\bigg\}, 
\end{equation} where $I_{k,l}(z):=\{j \in J_{k,l} : z \in \bfR^{k,\epsilon}_n(u;j)\}$, then we have the upper bound
\begin{equation}\label{eq:primbound}
\sum_{j \in J_{k,l}} \big|\bfR_n^{k,\epsilon}(u;j)\big| - |\bfR^{k,\epsilon}_n(u)| \leq \sum_{z \in \wh{\bfR}^{k,\epsilon}_n(u)} |I_{k,l}(z)| %\leq  \sum_{z \in \wh{\bfR}^{k,\epsilon}_n(u) \cap \bbW_n^{[0,r_n]}} |I_{k,l}(z)| + |\rmW_n^{[r_n,n-r_n]}(u)|,
\end{equation} %where, to obtain the second term in the right-hand side, we used the fact that, for $z \in \wh{\bfR}^{k,\epsilon}_n(u)$,
%\begin{equation}
%|I_{k,l}(z)|\leq \sum_{j \in J_{k,l}} 1_{\{z \in \bfR_n^{k,\epsilon}(u;j)\}} |\rmW_n(u) \cap \rmQ(x_j(z);k) | \leq |\rmW_n(u) \cap \rmB(z; \lfloor n -r_n \rfloor)|\footnote{Here I am using $x_j(z)$ instead of $x(z)$, might be convenient to do so everywhere.}
%\end{equation} since the $k$-boxes $\rmQ(x_j(z);k)$ are all disjoint and $\rmQ(z; \lfloor n -r_n \rfloor) \subseteq \rmB(z; \lfloor n -r_n \rfloor)$. 
On the other hand, if for each $z \in \wh{\bfR}^{k,\epsilon}_n(u)$ we define the \textit{annuli-scale} of its corresponding cluster as
\begin{equation}
\rho(z):=\min\left\{ m \in \bbN_{\geq 3} : \rmW_n(u) \cap \rmQ(z;n-r_n) \subseteq \rmB^-[x_j(z);m-2] \text{ for some }j \in I_{k,l}(z)\right\},
\end{equation}
then, since each of these clusters has annuli-scale at most $\cT_k+4$ by construction of each $\bfR^{k,\epsilon}_n(u;j)$, if for $m \in \N_{\geq 3}$ we abbreviate 
\begin{equation}
\wh{\bfR}^{k,\epsilon,m}_{n}(u):=\{ z \in \wh{\bfR}^{k,\epsilon}_n(u) : \rho (z)=m\}
\end{equation} and define
\begin{equation}
	\mathbf{D}_{n}^{k,\epsilon,m}(u):= \sum_{z \in \wh{\bfR}^{k,\epsilon,m}_{n}(u)} |I_{k,l}(z)|,
\end{equation}
 then\begin{equation}
\sum_{z \in \wh{\bfR}^{k,\epsilon}_n(u)} |I_{k,l}(z)| = \sum_{m=3}^{\cT_k+4} \mathbf{D}_{n}^{k,\epsilon,m}(u).
\end{equation} Thus, on the event $\{ \mathbf{D}^{k,\epsilon,m}_{n}(u) \leq \rme^{-(k+mk^\gamma)^{\frac{1}{2}-2\eta}}\sqrt{n} \text{ for all }m=3,\dots,\cT_k+4\}$, the left-hand side of \eqref{eq:primbound} is bounded from above by
\begin{equation}
\sum_{m=3}^{\cT_k+4}\rme^{-(k+mk^\gamma)^{\frac{1}{2}-2\eta}}\sqrt{n}  \leq C\sqrt{n}\rme^{-k^{\frac{1}{2}-2\eta}}
\end{equation} for some absolute constant $C>0$. Therefore, Proposition~\ref{p:300.32} will follow once we show that
\begin{equation}\label{eq:overc1}
\lim_{k \to \infty}\limsup_{n \to \infty}\bbP\Big( \exists m \in \{4,\dots,\cT_k+2\} : |\mathbf{D}^{k,\epsilon,m}_{n}(u)| > \rme^{-(k+mk^\gamma)^{\frac{1}{2}-2\eta}}\sqrt{n}\Big)=0
\end{equation} and
\begin{equation}\label{eq:overc2}
\lim_{k \to \infty}\limsup_{n \to \infty}\bbP\Big( |\mathbf{D}^{k,\epsilon,3}_{n}(u)| > \rme^{-(k+3k^\gamma)^{\frac{1}{2}-2\eta}}\sqrt{n}\Big)=0.
\end{equation}

We begin by showing \eqref{eq:overc1}. Let us notice that, if $k$ is large enough (depending only on~$\gamma$), then, given $i \geq 4$, for each $z \in \wh{\bfR}^{k,\epsilon,i}_{n}(u)$ there exist at least two points $x,x' \in \rmW_n(u) \cap \rmQ(z; n -r_n )$ such that $\log \|x-x'\| \geq k+ (i-5)k^\gamma$ (so that the annuli-scale of the corresponding cluster is at least $i$) and, by Lemma~\ref{lem:incball.1} (for $i-1$ in place of $i$), also some $y \in \bbX_{k+(i-2)k^\gamma}$ such that 
\begin{equation}
\rmW_n(u) \cap \rmQ(z;n -r_n ) \subseteq \rmB^-[y;i-1]	
\end{equation}
and
\begin{equation}
\wh{N}^{\rm{in}}_{t^A_n}[y;i] \geq \wh{N}_{t^A_n}[x_{j_0}(z);i]
\end{equation} for some $j_0 \in J_{k,l}$.
It follows from this choice of $x,x',y$ that, if we fix any $j \in I_{k,l}(z)$, then every $w_j \in \rmQ(x_j(z);k) \cap \rmW_n(u)$ satisfies $w_j \in \rmB^-[y;i-1]$ and $\log \|w_j - x\| \vee  \log \|w_j - x'\| \geq k+(i-6)k^\gamma$. Since the boxes $\rmQ(x_j(z);k)$ are disjoint for different values of $j \in J_{k,l}$, we conclude that 
\begin{equation}
|I_{k,l}(z)| \leq |\{ (w,w') : w,w' \in \rmW_n(u) \cap \rmB^-[y;i-1] \text{ and } \log \|w-w'\| \geq k+(i-6)k^\gamma\}|=:\rmH^i_{n,u}(y).
\end{equation} 
Hence, if we set $M^i_{k,\eta}:=[\sqrt{2}(k+(i-1)k^\gamma)+(k+ik^\gamma)^{\frac{1}{2}-\eta},\infty)$ and define
\begin{equation}
\begin{split}
\bbY^i_n(u;z):=\Big\{	y \in \bbX_{k+(i-2)k^\gamma} :\: & \rmW_n(u) \cap \rmQ(z;n-r_n) \subseteq \rmB^-[y;i-1] \\ 
& \sqrt{\wh{N}^{\rm{in}}_{t^A_n}[y;i]} \in M^i_{k,\eta}\Big\},
\end{split}
\end{equation}
together with $\bbZ_n^i(u):=\{ z \in \bfW_n(u) : \bbY^i_n(u;z) \neq \emptyset\}$, then 
\begin{equation}
\mathbf{D}^{k,\epsilon,i}_n(u) \leq \sum_{z \in \bbZ^i_n(u)} \max_{y \in \bbY^i_n(u;z)} \rmH^i_{n,u}(y),	
\end{equation} so that \eqref{eq:overc1} will now follow from Lemma~\ref{lem:resampling} (and Remark~\ref{rem:ninout}) if, for all $k$ large enough, we can show that
\begin{equation}\label{eq:boundbn.1}
b^{(i)}_n(u) \leq \rme^{-2(k+ik^\gamma)^{\frac{1}{2}-\eta}},
\end{equation} uniformly over all $i=4,\dots,\cT_k+4$, where 
\begin{equation}\label{eq:defbn}
b_n^{(i)}(u):= \sup_{y,\rho} \E \left( \rmH^i_{n,u}(y) \,\Big|\,  Z^{\mathrm{in}}_{t^A_n}[y;i]=\rho\,,\rmW_n(u) \cap \rmB^-[y;i-1] \neq \emptyset \right),
\end{equation} with the supremum being over the set of all $y \in \bbX_{k+(i-2)k^\gamma}$ with $\overline{\rmB(y;k+ik^\gamma -5\rme^{-k^\gamma})} \subseteq \rmD_n$ (so $N^{\mathrm{in}}_{t^A_n}[y;i]$ is well-defined) and $\rho:=(v,\overline{z}) \in \cZ^{\mathrm{in}}_k(i)$ with $v \in M^i_{k,\eta}$. Above $Z^{\mathrm{in}}_{t^A_n}[y;i]$ and $\cZ_k^{\mathrm{in}}(i)$ are defined as in \eqref{eq:defz} and \eqref{eq:defZcal} respectively, but with respect to $N_{t^A_n}^{\mathrm{in}}[y;i]$ instead of $N_{t^A_n}[y;i]$.

To check \eqref{eq:boundbn.1}, notice that, if $n$ is large enough (depending only on $\eta_0, \gamma$ and $\rmD$) so that the supremum in \eqref{eq:defbn} is over a nonempty set, then, for any $i=4,\dots,\cT_k+4$ and  $y,\rho:=(v,\ol{z})$ as in \eqref{eq:defbn}, whenever $\eta < \min \{\frac{1}{2}-\gamma,\frac{1}{6}\}$ and $k$ is large enough, we have by Proposition~\ref{prop:LT-2}~that 
\begin{equation}
\begin{split}
\E \left( \rmH^i_{n,u}(y) \,\Big|\,  Z^{\mathrm{in}}_{t^A_n}[y;i]=\rho\right) &\leq \sum_{w,w' \in \rmB^-[y;i-1] \atop \log \|w-w'\| \geq k+(i-6)k^\gamma} \bbP( \max \{ L_t(w) , L_t(w')\} \leq u \,|\, Z^{\mathrm{in}}_{t^A_n}[y;i]=\rho)\\
& \leq |\rmB^-[y;i]|^2 \rme^{-4(k+ik^\gamma) -4\sqrt{2}(k+ik^\gamma)^{\frac{1}{2}-\eta} +Ck^\gamma} \leq \rme^{-5(k+(i-1)k^\gamma)^{\frac{1}{2}-\eta}},
\end{split}
\end{equation} and by Proposition~\ref{prop:LT-1} that
\begin{equation}
\bbP\Big( \rmW_n(u) \cap \rmB^-[y;i-1] \neq \emptyset \,\Big|\,  Z_{t^A_n}^{\mathrm{in}}[y;i]=\rho\Big) \geq \rme^{-2\sqrt{2}(k+ik^\gamma)^{\frac{1}{2}-\eta}-(k+ik^\gamma)^{2\eta}-Ck^\gamma} \geq \rme^{-3(k+ik^\gamma)^{\frac{1}{2}-\eta}} , 
\end{equation} where $C$ above is some positive constant depending only on $u$ and $\gamma$. Combining both inequalities immediately yields \eqref{eq:boundbn.1}, provided that $k$ is taken large enough (depending only on $\eta$ and $\gamma$). This concludes the proof of \eqref{eq:overc1}.

The proof of \eqref{eq:overc2} is similar. Indeed, as before, for each $z \in \wh{\bfR}^{k,\epsilon,3}_{n}(u)$ there exists $y \in \bbX_{k+k^\gamma}$ such that 
\begin{equation}
\rmW_n(u) \cap \rmQ(z;n -r_n ) \subseteq \rmB^-[y;2]	
\end{equation}
and some $j_0 \in J_{k,l}$ for which
\begin{equation}
\wh{N}^{\rm{in}}_{t^A_n}[y;3] \geq \wh{N}_{t^A_n}[x_{j_0}(z);3],
\end{equation} as well as $j_1,j_2 \in J_{k,l}$ with $j_1\neq j_2$ such that $z \in \bfR_n^{k,\epsilon}(u;j)$. In particular, there exist two points $\ol{x}_1,\ol{x}_2 \in \rmW_n(u) \cap \rmB^-[y;2]$ such that $\ol{x}_i \in \rmQ(x_{j_i}(z);k-\epsilon)$ for $i=1,2$. Since $\| x_{j_1}(z) - x_{j_2}(z)\| \geq \lfloor \rme^k\rfloor$ by construction, it follows that $\log \|\ol{x}_1-\ol{x}_2\| \geq k - C_\epsilon$ for all $k$ large enough (depending only on $\epsilon$), where $C_\epsilon > 0$ is some constant depending only on~$\epsilon$. Thus, if we fix any $j \in I_{k,l}(z)$, then every $w_j \in \rmQ(x_j(z);k-\epsilon) \cap \rmW_n(u)$ satisfies $w_j \in \rmB^-[y;3]$ and $\log \|w_j - \ol{x}_1\| \vee  \log \|w_j - \ol{x}_2\| \geq k-C_\epsilon-1$. From there the proof continues as that of \eqref{eq:overc1}, we omit the details.
\end{proof}

\section{Sharp Ballot Estimates for Repelled Downcrossings}\label{s:barrier}

In this section, we prove Theorem~\ref{prop:coupling0}, the key element which allowed us to obtain Proposition~\ref{p:300.1a}. We begin by giving a brief outline of the proof. This outline provides a roadmap, highlighting the main auxiliary results required to establish Theorem~\ref{prop:coupling0}, without entering into their proofs. These auxiliary results will then be proved in separate subsections.
Before we begin, we suggest the reader to recall the definitions from Section~\ref{sec:downprelim}.

\subsection{General outline of the proof}

In order to prove Theorem~\ref{prop:coupling0}, our first step will be to show that the downcrossings probability on the left-hand side of~\eqref{eq:lemacoup} is asymptotically equivalent as $k \to \infty$ to that corresponding to a one-dimensional continuous time simple random walk (CTSRW). We will then obtain Theorem~\ref{prop:coupling0} by deriving the analogous sharp ballot estimates for the CTSRW.

To this end, let $\wt{X}^{[0,T]}$ be a CTSRW on the linear graph $\{0,\dots,T\}$ starting from $T$ with all transition rates between neighboring vertices equal to~$1$ and, given any $i=1,\dots,T$ and $n \in \N$, let $T_{n}(i)$ denote the number of transitions (or \textit{downcrossings}) from $i$ to $i-1$ of $\wt{X}^{[0,T]}$ in its first $n$ excursions away from $T$ and, for $s \in \cN_k$, let us write  $\wh{T}_{(s)}(i):=k^\gamma T_{\wt{s}_k}(i)$ with $\wt{s}_k:=\frac{s^2}{k^\gamma}$. Then, if given $s \in \cN_k$, $\eta \in (0,\tfrac{1}{2})$ and $i=1,\dots,T$ we define the event 
\begin{equation}
	\cT^{\eta,i}_{(s)}:= \Big\{ \sqrt{\wh{T}_{(s)}(i)}  \notin \mathcal{R}_k^\eta (i)\Big\}, 
\end{equation} and extend this definition to subsets $K \subseteq [1,T]$ in place of $i$ as usual by union over all $i \in K \cap \N_0$, the first step in the proof is to obtain the following asymptotic estimate, which is in the spirit of \cite[Lemma~2.4]{abe2021second} but also incorporates the entry and exit points at the initial and final scales.

\begin{prop}\label{lem:transfer} Given $\eta \in (0,\tfrac{1}{2}-\gamma)$, there exists $c=c(\eta) > 0$ such that, if $k$ is large enough, for any $2 \leq T \leq \rme^{ck^\gamma}$, $n \geq k$ and $x,\wt{x} \in \rmD_n$ with $x \in \rmB^-[\wt{x};T-1]$ and $\overline{\rmB^+[\wt{x};T]} \subseteq \rmD_n$, one has  
	\begin{equation}\label{eq:lemacoup2}
		\frac{\bbP \Big( \big\{ Z_t[x;1] = (u,\ol{y})\big\} \setminus \mathrm{NR}^{\eta,[2,T-1]}_{n,t}(x) \,\Big|\,Z_t[\wt{x};T] = (v,\ol{z}) \Big)}{\bbP \Big( \big\{ \sqrt{\wh{T}_{(v)}(1)} = u\big\} \setminus \mathcal{T}^{\eta,[2,T-1]}_{(v)}\Big)}=h_{k}(\ol{y})(1+O(\rme^{-ck^\gamma}))
	\end{equation} uniformly over all $t > 0$, $(u,\overline{y}) \in \cZ_k[1]$ with $0 \leq \wt{u}_k \leq \rme^{ck^\gamma}$ and $(v,\overline{z}) \in \cZ_k[T]$ with $1 \leq \wt{v}_k \leq \rme^{ck^\gamma}$, where $h_k$ is as in~\eqref{eq:defhk1}.
\end{prop}

In view of Proposition~\ref{lem:transfer}, to obtain Theorem~\ref{prop:coupling0} it will suffice to prove the analogous result for the one-dimensional walk $\wt{X}^{[0,T]}$. This is the content of the following proposition.

\begin{prop}\label{prop:coupling1} If $\gamma$ is chosen sufficiently small then, for any $0 < \eta' < \eta < \tfrac{1}{2}$ small enough (depending only on $\gamma$), if $k$ is sufficiently large then, for any $2k \leq T \leq \rme^{\frac{1}{2}k^\gamma}$, we have that
	\begin{equation}\label{eq:lemacoupp0}
		\bbP \Big( \big\{ \sqrt{\wh{T}_{(v)}(1)} = u\big\} \setminus \mathcal{T}^{\eta,(T)}_{(v)} \Big)=r_{k,T}(\wh{u},\wh{v})(1+o_k(1))
	\end{equation} holds uniformly over all $u \in \mathcal{R}_k^{\eta'}(1)$ and $v \in \mathcal{R}_k^{\eta'}(T)$, where $r_{k,T}$ is as in~\eqref{eq:funcfg}.
\end{prop}
 
From these two propositions, Theorem~\ref{prop:coupling0} follows at once.

\begin{proof}[Proof of Theorem~\ref{prop:coupling0}]
Immediate from Proposition~\ref{lem:transfer} and Proposition~\ref{prop:coupling1}.	
\end{proof}

Thus, we now turn to discuss how to prove Proposition~\ref{lem:transfer} and Proposition~\ref{prop:coupling1}.	

Proposition~\ref{lem:transfer} is a consequence of a careful analysis of the trajectories of the random walk $\bfX$, using its strong Markov property together with standard gambler's ruin estimates combined with sharp asymptotics for the Poisson kernel and harmonic measure of balls. We give the proof of this result in Section~\ref{sec:prooftransfer} below.

In order to prove Proposition~\ref{prop:coupling1}, we will first establish the analogous statement but instead for the local time field associated with $\wt{X}^{[0,T]}$. To this end, we first need some further definitions. Given $t > 0$, we define the local time field $\wt{L}_t:=(\wt{L}_t(i) : i=0,\dots,T)$, where $\wt{L}_t(i)$ here denotes the local time spent at site $i$ by $\wt{X}^{[0,T]}$ until $T$-time $t$, i.e., until it spends $t$ local time at site $T$. In addition,  we set $\wh{L}_t:=k^\gamma \wt{L}_t$ and, for each $\eta \in (0,\tfrac{1}{2})$ and $i=0,\dots,T$, define the set $\mathcal{Q}_k^\eta(i)$ of all $u \geq 0$ which lie on the interval $\sqrt{2}(k+ik^\gamma) + \mathfrak{I}_k^\eta(i)$, i.e. 
\begin{equation}
		\mathcal{Q}_k^\eta(i):=\Big\{ u \in \R_{\geq 0} : u - \sqrt{2}(k+ik^\gamma) \in \mathfrak{I}_k^\eta(i)\Big\}\,,	
\end{equation} where $\mathfrak{I}_k^\eta(i)$ is as in \eqref{eq:defirep}. Given $i=1,\dots,T$ and $u \in \cQ^\eta_k(i)$, we define the \textit{recentering} of $u$ as 
\begin{equation}\label{eq:defwlt}
\wh{u}:= u - \sqrt{2}(k+ik^\gamma).
\end{equation} \textcolor{black}{Notice that $\wh{u}$ depends on $i$ but we are omitting this dependence from the notation for simplicity. However, this will not cause any confusion in the sequel, as we will always clarify beforehand that $u \in \cQ^\eta_k(i)$ for some specific value of $i$, which is the precise value to be used when recentering. On a similar note, note the difference between the recentering in \eqref{eq:defwlt} for $u \in \cQ^\eta_k(i)$ and the~one in \eqref{eq:defudc} for $u \in \cR^\eta_k(i)$.  Again, no confusion will arise from this, as we will always clarify whether we are working with $u \in \cQ^\eta_k(i)$ or $u \in \cR^\eta_k(i)$ beforehand.}
Finally, if, for $i=0,\dots,T$ and $t > 0$, we define the event
\begin{equation}
	\cL^{\eta,i}_t:= \Big\{ \sqrt{\wh{L}_t(i)}  \notin \mathcal{Q}_k^\eta(i)\Big\}, 
\end{equation} and extend the definition to subsets $K \subseteq [0,T]$ in place of $i$ as usual by union over all $i \in K \cap \N_0$, we have the following result.

\begin{prop}\label{prop:coupling2} If $\gamma > 0$ is taken small enough then, for any $\eta \in (0,\tfrac{1}{2})$ sufficiently small (depending only on $\gamma$) we have that, for all $k$ sufficiently large and $T \in [k,\rme^k]$,
	\begin{equation}\label{eq:lemacoupp}
		\bbP \bigg( \bigg\{ \sqrt{\wh{L}_{\wt{v}_k}(0)} \in \rmd u\bigg\} \setminus \mathcal{L}^{\eta,[1,T-1]}_{\wt{v}_k} \bigg)=\frac{4}{\sqrt{\pi}} \cdot \frac{1}{Tk^{1/2+\gamma}}\rme^{-2Tk^\gamma}\left(\wh{u}\rme^{2\sqrt{2}\wh{u}}\right)\left(\wh{v}\rme^{-2\sqrt{2}\wh{v}-\frac{\wh{v}^2}{Tk^\gamma}}\right)(1+o_k(1))\,
	\end{equation} holds uniformly over $u \in \mathcal{Q}_k^\eta(0)$ and $v \in \mathcal{Q}_k^\eta(T)$, where $\wh{u}:=u-\sqrt{2}k$ and $\wh{v}:=v -\sqrt{2}(k+Tk^\gamma)$.
\end{prop}

Having the ballot estimates for the local time field of $\wt{X}^{[0,T]}$ as given by Proposition~\ref{prop:coupling2}, our next step will be to show that an analogue of Proposition~\ref{prop:coupling2} still holds if we replace $\wt{L}_t(0)$ by $T_{\wt{v}_k}(1)$ with $\wt{v}_k$ the number of excursions of $\wt{X}^{[0,T]}$ until $T$-time $t$. In other words, our next step is to show a version of the ballot estimate from Proposition~\ref{prop:coupling1} in which the repulsion condition is stated in terms of  local time instead of downcrossings. 

To state this version properly, let us consider the local time field associated with the CTSRW $\wt{X}^{[0,T]}$ but measured with respect to number of excursions away from $T$. That is, for each $m \in \N$ and $i=0,\dots,T$, we define the local time $\wt{L}^*_{m}(i):=\wt{L}_{\tau_{m}}(i)$ where $\tau_m$ denotes the total local time spent at site $T$ until the $m$-th return time to $T$ (i.e., after $m$ excursions away from $T$). Then, if given $s \in \cN_k$ we write $\wh{L}_{(s)}:=k^\gamma \wt{L}^*_{\wt{s}_k}$ and, for $\eta \in (0,\tfrac{1}{2})$ and $i=0,\dots,T$, we define the event 
\begin{equation}
	\cL^{\eta,i}_{(s)}:= \Big\{ \sqrt{\wh{L}_{(s)}(i)}  \notin \mathcal{Q}_k^\eta(i)\Big\} 
\end{equation} and extend this definition to subsets $K \subseteq [0,T]$ in place of $i$ as usual by union over all $i \in K \cap \N_0$, the modification of Proposition~\ref{prop:coupling1} we are after is the following. 

\begin{lem}\label{lem:asympforl}
	If $\gamma > 0$ is small enough then, for any $0 < \eta' < \eta < \tfrac{1}{2}$ sufficiently small (depending only on $\gamma$)  we have that, for all $k$ sufficiently large and $T \in [2k,\rme^{\tfrac{1}{2}k^\gamma}]$,
	\begin{equation}\label{eq:asymp3}
		\bbP \Big( \Big\{ \sqrt{\wh{T}_{(v)}(1)} = u\Big\} \setminus \cL^{\eta,[1,T-1]}_{(v)} \Big)=r_{k,T}(\wh{u},\wh{v})(1+o_k(1))
	\end{equation} holds uniformly over all $u \in \cR^{\eta'}_k(1)$ and $v \in \cR^{\eta'}_k(T)$, where $r_{k,T}$ is as in \eqref{eq:funcfg}.
\end{lem}

The last ingredient for the proof of Proposition~\ref{prop:coupling1} is to show that in~\eqref{eq:asymp3} we can replace $\cL_{(v)}^{\eta,[1,T-1]}$ by $\cT_{(v)}^{\eta,[1,T-1]}$, i.e., replace local time repulsion by repulsion in number of downcrossings. To do this, we will show that, if the local time field is repelled, then the downcrossings trajectory cannot be to far apart and thus must also be repelled, and viceversa. More precisely, if for $s \in \cN_k$, $\eta \in (0,\frac{1}{2})$ and $i=1,\dots,T$ we define the event
\begin{equation}
	\cJ^{\eta,i}_{(s)}:=\Big\{ \Big|\sqrt{\wh{L}_{(s)}(i)} -\sqrt{\wh{T}_{(s)}(i)}\Big| > (k+ik^\gamma)^{\frac{1}{2}-\eta}\Big\}
\end{equation} and extend this definition to subsets $K \subseteq [1,T]$ in place of $i$ as usual by union over all $i \in K \cap \N_0$, then the last step of the proof is to show the following result.

\begin{lem}\label{lem:asympdisc} If $\gamma > 0$ is chosen small enough then, for any $\eta > 0$ small enough (depending only~on~$\gamma$) we have that, for all $k$ sufficiently large and $T \in [2k,\rme^{\frac{1}{2}k^\gamma}]$, 
	\begin{equation}\label{eq:step3}
		\bbP \Big( \Big\{ \sqrt{ \wh{T}_{(v)}(1)} = u \Big\} \cap \cJ^{\eta,[1,T-1]}_{(v)} \setminus \Big( \cL_{(v)}^{\eta,[1,T-1]} \cap \cT_{(v)}^{\eta,[2,T-1]}\Big)
		\Big) = r_{k,T}(\wh{u},\wh{v})o_k(1)
	\end{equation} holds uniformly over all $u \in \cR^{\eta}_k(1)$ and $v \in \cR^\eta_k(T)$, where $r_{k,T}$ is as in \eqref{eq:funcfg}.
\end{lem}

Having Propositions~\ref{lem:asympforl}--\ref{lem:asympdisc} at our disposal, the proof of Proposition~\ref{prop:coupling1} is straightforward.

\begin{proof}[Proof of Proposition~\ref{prop:coupling1}] Let us fix $0< \eta' < \eta < \frac{1}{2}$ and take $\eta_1,\eta_2$ such that $\eta' < \eta_2 < \eta < \eta_1$. Then, by the union bound, if $k$ is sufficiently large then the left-hand side of \eqref{eq:lemacoupp0} is bounded from above by
	\begin{equation}
		\bbP \Big( \big\{ \sqrt{\wh{T}_{(v)}(1)} = u\big\} \setminus \mathcal{L}^{\eta_1,[1,T-1]}_{(v)} \Big) + \bbP \Big( \big\{ \sqrt{\wh{T}_{(v)}(1)} = u\big\} \cap \cJ_{(v)}^{\eta_1,[1,T-1]} \setminus \mathcal{T}^{\eta_1,[2,T-1]}_{(v)} \Big)
	\end{equation} and from below by
	\begin{equation}
		\bbP \Big( \big\{ \sqrt{\wh{T}_{(v)}(1)} = u\big\} \setminus \mathcal{L}^{\eta_2,[1,T-1]}_{(v)} \Big) - \bbP \Big( \big\{ \sqrt{\wh{T}_{(v)}(1)} = u\big\} \cap \cJ_{(v)}^{\eta,[1,T-1]} \setminus \mathcal{L}^{\eta,[1,T-1]}_{(v)} \Big)\,.
	\end{equation} The result now follows at once from Lemma~\ref{lem:asympforl} and Lemma~\ref{lem:asympdisc}.
\end{proof}

In order to complete the proof of Theorem~\ref{prop:coupling0}, the plan for the rest of this section is to give the proofs of Propositions~\ref{lem:transfer} and \ref{prop:coupling2}, together with those of Lemmas~\ref{lem:asympforl} and \ref{lem:asympdisc}- We begin by presenting some preliminary results to be used in the proofs. 

\subsection{Preliminaries}\label{sec:prelimballot}

We collect some auxiliary results which are necessary for the proofs in the subsections to follow. The proof of all these auxiliary results can be found in Section~\ref{sec:prelimballot}. Henceforth we denote by $\tau_\rmA$ and $\ol{\tau}_\rmA$ the first hitting (real) time of $\rmA \subseteq \rmD_n$ and the first return (real) time to that set by the random walk $\bfX$, namely,
\begin{equation}
	\label{e:7.1i}
	\tau_\rmA = \inf \{u \geq 0 :\: \bfX_u \in \rmA \} \quad ;\qquad
	\ol{\tau}_\rmA = \inf \{u > \tau_{\rmA^\rmc}:\: \bfX_u \in \rmA \} \,.
\end{equation}
We extend these definitions also to the simple symmetric random walk on $\bbZ^2$ (not just on~$\wh{\rmD}_n$).
The first auxiliary result gives standard gambler's ruin estimates for simple random walk on~$\bbZ^2$.

\begin{lem} \label{lem:ll1} Given $x \in \bbZ^2$ and $0 < r < R$, the simple symmetric random walk on $\bbZ^2$ satisfies the following estimates:% < n$ with $\rmB(x;R) \subseteq \rmD_n$, the following estimates hold
	\begin{enumerate}
		\item [a)] For any $y \in \rmB(x;R)\setminus \rmB(x;r)$, 
		\begin{equation}
			\bbP_y( \tau_{\rmB(x;r)} < \tau_{\partial \rmB(x;R)})  = \frac{R- \log \|x-y\| +O(\rme^{-r})}{R-r}.  
		\end{equation}
		\item [b)] For any $y \in \rmB(x;R) \setminus \{x\}$, 
		\begin{equation}\label{eq:LLest2}
			\bbP_y( \tau_x < \tau_{\partial \rmB(x;R)}) = \left(\frac{R-\log\|x-y\|+O(\|x-y\|^{-1})}{R}\right)(1+O(R^{-1})).  
		\end{equation} 
	\end{enumerate}
\end{lem}

We shall also need precise asymptotics for the Poisson kernel and harmonic measure of balls. Let us first recall what these two objects are. Given any finite subset $A \subseteq \bbZ^2$ with $0 \in A$ and a SSRW $(Z_n)_{n \in \bbN_0}$ on $\bbZ^2$, we write $\Pi_A : \ol{A} \times \partial A \to [0,1]$ for the corresponding \textit{Poisson kernel} given by
\begin{equation}
\Pi_A(x,z):=\bbP_x( Z_{\tau_{\partial A}} = z )
\end{equation} and $\amalg_A: \partial_i A \to [0,1]$ to denote the corresponding \textit{harmonic measure} (from infinity) given by 
\begin{equation}\label{eq:defamalg}
\amalg_{A}(z):=\lim_{|x| \to \infty} \bbP_x(Z_{\tau_A} = z).	
\end{equation} It can be shown that the limit in~\eqref{eq:defamalg} indeed exists and satisfies the equality 
\begin{equation}\label{eq:amalgdef}
\amalg_{A}(w)=\E_w(g_A(Z_{\ol{\tau}_A \wedge \tau_{\partial \rmB(0;m)}})),
\end{equation} for any $m \geq 1$, where $g_{A}(z):=a(z)-\E_z(a(Z_{\tau_{A}}))$, see e.g. \cite[Section 6.6]{LL} for further details. The asymptotic estimates we will need for these quantities are contained in the following lemma.

\begin{lem}\label{lem:kernelfrominf} There exists $\ol{c} > 0$ such that, if $r,r'$ are large enough, given any $n \geq 0$, $x,\wt{x} \in \rmD_n$ and $k,l \geq 1$ such that $l \geq k+2$, $\rmB(x;k+r) \subseteq \rmB(\wt{x};l)$ and $\ol{\rmB(\wt{x};l+r')}\subseteq \rmD_n$, it holds that:
	\begin{equation}\label{eq:miss1}
		\bbP_{y} ( Z_{\tau_{\partial \rmB(\wt{x};l+r')}}=\wt{x}+z)= \Pi_{\rmB(0;l+r')}(0,z)(1+O(\rme^{-\ol{c}r'}))
	\end{equation} uniformly over all $z \in \partial \rmB(0;l+r')$ and $y \in \rmB(\wt{x};l)$; 
\begin{equation}\label{eq:miss2}
		\bbP_{y} ( Z_{\tau_{\partial \rmB(\wt{x};l+r')}}=\wt{x}+z\,|\, \tau_{\partial \rmB(\wt{x};l+r')} <  \ol{\tau}_{\rmB(x;k)})= \Pi_{\rmB(0;l+r')}(0,z)(1+O(\rme^{-\ol{c}r')}))
	\end{equation} uniformly over all $z \in \partial \rmB(0;l+r')$ and $y \in \rmB(\wt{x};l)\setminus (\rmB(x;k) \setminus \partial_i \rmB(x;k))$; and, finally, 
	\begin{equation}\label{eq:miss3}
		\bbP_{y}\big( Z_{\tau_{\rmB(x;k)}} = x+w \,\big| \tau_{\rmB(x;k)}< \tau_{\partial \rmB(\wt{x};l+r')}\big)=\amalg_{\rmB(0;k)}(w)(1+O(\rme^{-\ol{c}(r \wedge k)}))
	\end{equation}
	uniformly over all $w \in \partial_i \rmB(0;k)$ and $y \in \rmB(\wt{x};l+r') \setminus \rmB(x;k+r)$.
	\end{lem}
	
In addition, in the proofs we will need to identify the distributions the local time field $\wt{L}_{(v)}(j)$ and number of downcrossings $T_n(i)$ associated with the one-dimensional CTSRW. To this end, we introduce the concept of compound distributions. 

For us, a \textit{compound distribution} will be the law of any random sum of the form $S:=\sum_{i=1}^N V_i$, where $(V_i)_{i \in \bbN}$ are i.i.d. random variables and $N$ is an independent $\bbN_0$-valued random variable. In the sequel, we will say that $S$ has:
\begin{itemize}
\item [i.] \textit{Binomial-Geometric} distribution of parameters $(n,p;q)$, and denote it $S \sim \Bigeo(n,p;q)$, if $N \sim \textrm{Binormail}(n,p)$ and $V_i \sim \textrm{Geometric}(q)$; 	
\item [ii.] \textit{Binomial-Exponential} distribution of parameters $(n,p;\lambda)$, denoted as $S \sim \Biexp(n,p;\lambda)$, if $N \sim \textrm{Binormail}(n,p)$ and $V_i \sim \textrm{Geometric}(q)$; 	
\item [iii.] \textit{Poisson-Geometric} distribution of parameters $(\mu,q)$, and denote it by $S \sim \Poigeo(\mu;q)$, if $N \sim \textrm{Binormail}(n,p)$ and $V_i \sim \textrm{Geometric}(q)$; 
\end{itemize}

The following auxiliary result gives some basic facts about the laws of $\wt{L}_{(v)}(j)$ and $T_n(j)$.

\begin{lem}\label{lem:idlaw} Given $T \in \N$, the random walk $\wt{X}^{[0,T]}$ satisfies:
\begin{itemize}
\item [\rm{i.}] For any $i=0,\dots,T-1$, $j=i,\dots,T-1$, $v \in \cN_k$ and $m \in \bbN$, conditional on $T_{\wt{v}_k}(T-i)=m$, 

\begin{equation}\label{eq:propltc}
T_{\wt{v}_k}(T-j) \sim \Bigeo\left(m, \frac{1}{j-i+1},\frac{1}{j-i+1}\right)	.
\end{equation}
 	\item [\rm{ii.}] For any $i=1,\dots,T$ and $v \in \cN_k$, 
	\begin{equation}\label{eq:proplta}
	\wt{L}_{(v)}(T-i) \sim \Biexp\left(\wt{v}_k,\frac{1}{i};\frac{1}{i}\right).
\end{equation}
	\item [\rm{iii.}] For any $i=0,\dots,T-1$, $j=i,\dots,T-1$ and $\theta >0$, conditional on $\sqrt{\wt{L}_{(v)}(T-i)}=\theta$, 
\begin{equation}\label{eq:propltb}
T_{\wt{v}_k}(T-j) \sim \Poigeo\left(\frac{\wt{\theta}_k}{j-i+1};\frac{1}{j-i+1}\right).
\end{equation}
\end{itemize} 
\end{lem}

In light of Lemma~\ref{lem:idlaw}, the following tail estimates for compound distributions will be of use. 

\begin{lem}\label{lem:compound} Let $n \in \bbN$, $p,q \in (0,1]$ and $\lambda,\mu > 0$ be fixed parameters. Then:
\begin{itemize}
\item [\rm{i.}] If $S \sim \Bigeo(n,p;q)$, then $\begin{cases}\bbP( S \leq \theta ) \leq \rme^{-(\sqrt{np}-\sqrt{\theta q})^2}  \quad \text{ if $\theta \leq \frac{np}{q}$,} \\ \\ \bbP(  S  \geq \theta ) \leq \rme^{-(\sqrt{\theta q}-\sqrt{np})^2} \quad \text{ if $\theta \geq \frac{np}{q}$}.\end{cases}$ 
\item [\rm{ii.}] If $S \sim \Biexp(n,p;\lambda)$, then $\begin{cases}
\bbP( S \leq \theta ) \leq \rme^{-(\sqrt{np}-\sqrt{\theta \lambda})^2} \quad \text{ if $\theta \leq \frac{np}{\lambda}$,}\\ \\
\bbP(  S  \geq \theta ) \leq \rme^{-(\sqrt{\theta \lambda}-\sqrt{np})^2} \quad \text{ if $\theta \geq \frac{np}{\lambda}$}.\end{cases}$
\item[\rm{iii.}] If $S \sim \Poigeo(\mu;q)$, then $\begin{cases}
\bbP(S \leq \theta) \leq \rme^{-(\sqrt{\mu}-\sqrt{\theta p})^2} \quad \text{ if $\theta \leq \frac{\mu}{p}$,} \\ \\ 
\bbP(S \geq \theta) \leq \rme^{-(\sqrt{\mu}-\sqrt{\theta p})^2} \quad \text{ if $\theta \geq \frac{\mu}{p}$}.\end{cases}$
\end{itemize}
\end{lem}

\subsection{Transfer to a 1D CTSRW: proof of Proposition~\ref{lem:transfer}}\label{sec:prooftransfer}

To prove Proposition~\ref{lem:transfer}, we first obtain a similar statement but without the repulsion condition. Throughout this subsection, $\wt{X}^{[i-1,j]}$ will denote a CTSRW on the linear graph $\{i-1,\dots,j\}$ starting from $j$ with all transition rates between neighbors equal to~$1$ and, given $i=1,\dots,j$ and $n \in \N$, we will write $T^{[i-1,j]}_{n}(i)$ for the number of downcrossings from $i$ to $i-1$ of $\wt{X}^{[i-1,j]}$ in its first $n$ excursions away from $j$ and, given $s \in \cN_k$, also abbreviate $\wh{T}^{[i-1,j]}_{(s)}(i):=k^\gamma T^{[i-1,j]}_{\wt{s}_k}(i)$. The auxiliary asymptotics without repulsion we shall need are the following.

\begin{lem}[Transfer Lemma]\label{lem:rep1} There exists a constant $c > 0$ such that, if $k$ is large enough, for any $i,j \in \N$ with $i<j$, $n \geq k$ and $x,\wt{x} \in \rmD_n$ such that $x \in \rmB^-[\wt{x};j-1]$ and $\overline{\rmB^+[\wt{x};j]}\subseteq \rmD_n$, we have that	
\begin{equation}\label{eq:GR-3}
\bbP\Big( Z_t[x;i] = (u,\ol{y}) \,\Big|\, Z_t[\wt{x};j] = (v,\ol{z})\,;\cF[\wt{x};j]\Big) = \bbP \Big( \sqrt{\wh{T}^{[i-1,j]}_{(v)}(i)} = u\Big)h_{k,i}(\ol{y}) (1+O(\rme^{-ck^\gamma})), 
\end{equation}
uniformly over all $t> 0$ and $(u,\ol{y}) \in \cZ_k[i]$, $(v,\ol{z}) \in \cZ_k[j]$ with $0 \leq \wt{u}_k \leq \rme^{ck^\gamma}$ and $1 \leq \wt{v}_k \leq \rme^{ck^\gamma}$. Above, the function $h_{k,i}$ is given by 
	\begin{equation}
		h_{k,i}(\ol{y}):=\begin{cases}1 & \text{ if $\ol{y}=\emptyset$}\\ \\\displaystyle{\prod_{r=1}^m \amalg_{\rmB^-[0;i]}(y_r) \Pi_{\rmB^+[0;i]}(0,y_r')} &\text{ if $\ol{y}=((y_1,y_1'),\dots,(y_{m},y'_{m})) \in \textrm{E}_{m}[i]$ for $m \in \N$}.
		\end{cases}	
	\end{equation}
\end{lem}
 
\begin{proof} Let $V_r$ be the number of $[x;i]$-downcrossings during the $r$-th $[\wt{x};j]$-excursion of the walk on $\wh{\rmD}_n$. Then, on the one~hand, on the event $\Big\{\sqrt{\wh{N}_t[\wt{x};j]}=v\Big\}$,
	\begin{equation}\label{eq:ntrep}
		N_t[x;i]=\sum_{r=1}^{\wt{v}_k} V_r.	
	\end{equation} On the other hand, the random variables $(V_r)_{r \in \N}$ are independent conditionally on $\cF[\wt{x};j]$ since, by the strong Markov property, these excursions are independent given their entry/exit points. Moreover, if $\bbP_{z,z'}$ denotes the law of an $[\wt{x};j]$-excursion conditioned to start at $z$ and end at~$z'$ and $V$ denotes the number of $[x;i]$-downcrossings during this excursion, for any $m \in \bbN$ we have
\begin{equation}
	\bbP(V_r = m \,|\, \cF[\wt{x};j]) = \bbP_{Y_r^{\text{in}}[\wt{x};j],Y_r^{\text{out}}[\wt{x};j]}(V = m), 
\end{equation} where $Y_r^{\text{in}}[\wt{x};j]$ and $Y_r^{\text{out}}[\wt{x};j]$ are respectively the entry and exit points of the $r$-th $[\wt{x};j]$-excursion.  With all this information in mind, let us first prove \eqref{eq:GR-3}. We proceed by induction on $\wt{v}_k \in \N$. 

Assume first that $\wt{v}_k=1$. In this case, notice that proving \eqref{eq:GR-3} boils down to showing that, for some suitable constant $c_1 > 0$,
\begin{equation}\label{eq:new-v=1}
\bbP_{z.z'}(N^*_{1}[x;i]=\wt{u}_k\,,\, \ol{Y}_{(\wt{u}_k)}[x;i]=\ol{y})=\bbP(T_1^{[i-1,j]}(i)=\wt{u}_k)h_k(\overline{y})(1+O(\rme^{-c_1k^\gamma}))
\end{equation} holds uniformly over all $\overline{z}=(z,z') \in \partial_i \rmB^-[\wt{x};j] \times \partial \rmB^+[\wt{x};j]$ and $(u,\overline{y}) \in \cZ_k[i]$ with $0 \leq \wt{u}_k \leq \rme^{ck^\gamma}$, where, for $m \in \N$, $N^*_{m}[x;i]$ is the number of $[x;i]$-excursions during the first $m$ $[\wt{x};j]$-excursions. 
Furthermore,  by standard gambler's ruin estimates for a random walk on $\bbZ$, a straightforward computation shows that
\begin{equation}
	\bbP(T_1^{[i-1,j]}(i)=\wt{u}_k)=\begin{cases} \displaystyle{\frac{j-1}{j-i+1}} & \text{ if $\wt{u}_k=0$ }\\ \\ \displaystyle\left(\frac{1}{j-i+1}\right)^2\left(\frac{j-i}{j-i+1}\right)^{\wt{u}_k-1} & \text{ if $\wt{u}_k \geq 1$.}\end{cases}
\end{equation} With this in mind, let us show \eqref{eq:new-v=1}. To this end, note first that for $\wt{u}_k=0$ the left-hand side of \eqref{eq:new-v=1} becomes
	\begin{equation}\label{eq:ndecomp}
		\bbP_{z,z'}(N^*_{1}[x;i] = 0) = \frac{\bbP_{z}( Z_{\tau_{\partial \rmB^+[\wt{x};j]}}=z' \,| \tau_{\partial \rmB^+[\wt{x};j]} < \tau_{\rmB^-[x;i]}) \bbP_{z}(\tau_{\partial \rmB^+[\wt{x};j]} < \tau_{\rmB^-[x;i]})}{\bbP_{z} (Z_{\tau_{\partial \rmB^+[\wt{x};j]}}=z')},	
	\end{equation} where we recall that $(Z_n)_{n \in \bbN_0}$ above denotes a SSRW on $\bbZ^2$. But then, by Lemma~\ref{lem:kernelfrominf} we see~that 
	\begin{equation}\label{eq:new-hc}
		\frac{\bbP_{z}( Z_{\tau_{\partial \rmB^+[\wt{x};j]}}=z' \,| \tau_{\partial \rmB^+[\wt{x};j]} < \tau_{\rmB^-[x;i]}) }{\bbP_{z} (Z_{\tau_{\partial \rmB^+[\wt{x};j]}}=z')}=1+O(\rme^{-\overline{c}k^\gamma})
	\end{equation} for some $\overline{c} > 0$, while, on the other hand,  we have
	\begin{equation}\label{eq:new-cross1}
		\bbP_{z}(\tau_{\partial \rmB^+[\wt{x};j]} < \tau_{\rmB^-[x;i]}) = \frac{j-i}{j-i+1}(1+O(\rme^{-k^\gamma}))\,.
	\end{equation} 
Indeed, notice that, since $x \in \rmB^-[\wt{x};j-1]$, if we define the radii $\rho^\pm_j := k+jk^\gamma -4\rme^{-k^\gamma} \pm  \rme^{-\tfrac{3}{2}k^\gamma}$, then, for all $k$ large enough (depending only on $\gamma$),  
\begin{equation}
\rmB^-[\wt{x};j] \subseteq \rmB(x;\rho^-_l) \subseteq \rmB^+[\wt{x};j] \subseteq \rmB(x;\rho^+_j)	,
\end{equation} which implies the inequalities
\begin{equation}
\bbP_z (\tau_{\partial \rmB(x;\rho_j^+)} < \tau_{\rmB^-[x;i]}) \leq \bbP_{z}(\tau_{\partial \rmB^+[\wt{x};j]} < \tau_{\rmB^-[x;i]}) \leq \bbP_z (\tau_{\partial \rmB(x;\rho_j^-)} < \tau_{\rmB^-[x;i]}), 
\end{equation} from where \eqref{eq:new-cross1} now follows by Lemma~\ref{lem:ll1}, since $\rho_j^\pm - (k+(i-1)k^\gamma)=(j-i+1)k^\gamma + O(\rme^{-k^\gamma})$ and $\log |x-z| = k + (j-1)k^\gamma + O(\rme^{-k^\gamma})$ 
uniformly over all $z \in \partial_i \rmB^-[\wt{x};j]$. Inserting \eqref{eq:new-hc}--\eqref{eq:new-cross1} into \eqref{eq:ndecomp} immediately gives \eqref{eq:new-v=1} for $\wt{u}_k=0$.
	
	To show \eqref{eq:new-v=1} for $\wt{u}_k \geq 1$ (recall that $\wt{u}_k$ is a nonnegative integer since $u \in \cN_k$), notice that it will suffice to prove there exists $c_1>0$ such that
	\begin{equation}\label{eq:new-c1}
		\bbP_{z,z'}(N^*_{1}[x;i]=\wt{u}_k\,,\, \ol{Y}_{(\wt{u}_k)}[x;i]=\ol{y}) =\left(\frac{1}{j-i+1}\right)^2 \left(\frac{j-i}{j-i+1}\right)^{\wt{u}_k-1}	h_{k,i}(\ol{y})(1+\wt{u}_kO(\rme^{-2c_1k^\gamma }))
	\end{equation} uniformly over $1 \leq \wt{u}_k \leq \rme^{c_1k^\gamma}$, $\ol{y} \in (\partial_i \rmB^-[0;i] \times \partial \rmB^+[0;i])^{\wt{u}_k}$ and $(z,z') \in \partial_i \rmB^-[\wt{x};j] \times \partial \rmB^+[\wt{x};j]$. 
	We proceed by induction on $\wt{u}_k$.
	
	If $\wt{u}_k=1$ and we abbreviate $x+\ol{y}=x+(y,y')=:(w,w')$, by the strong Markov property we may write the left-hand side of \eqref{eq:new-c1} as
	\begin{equation}\label{eq:new-decoup1}
		\bbP_{z}(Z_{\tau_{\rmB^-[x;i]}}=w\,, \tau_{\rmB^-[x;i]} < \tau_{\partial \rmB^+[\wt{x};j]})\bbP_{w}(Z_{\tau_{\partial \rmB^+[x;i]}}=w')\bbP_{w',z'}( N^*_{1}[x;i]=0) \frac{ \bbP_{w'}( Z_{\tau_{\partial \rmB^+[\wt{x};j]}}=z')}{ \bbP_z( Z_{\tau_{\partial \rmB^+[\wt{x};j]}}=z')}\,,
	\end{equation} where $N^*_{1}[x;i]$ denotes the number of $[x;i]$-downcrossings of the $[\wt{x};j]$-excursion with law $\bbP_{w',z'}$. Furthermore, by Lemma~\ref{lem:kernelfrominf} there exists $\ol{c} > 0$ such that, uniformly over all $\overline{y}$, $(z,z')$ as above,
	\begin{equation}\label{eq:new-decoup2}
		\bbP_{z}(Z_{\tau_{\rmB^-[x;i]}}=w\,, \tau_{\rmB^-[x;i]} < \tau_{\partial \rmB^+[\wt{x};j]}) = \amalg_{\rmB^-[0;i]}(y)\bbP_{z}(\tau_{\rmB^-[x;i]} < \tau_{\partial \rmB^+[\wt{x};j]}) (1+O(\rme^{-\ol{c}k^\gamma }))\,,
	\end{equation}
	\begin{equation}\label{eq:new-decoup3}
		\bbP_{w}(Z_{\tau_{\partial \rmB^+[x;i]}}=w')= \Pi_{\rmB^+[0;i]}(0,y')(1+O(\rme^{-\ol{c}k^\gamma}))\,,
	\end{equation} as well as 
	\begin{equation}\label{eq:new-decoup4}
		\frac{ \bbP_{w'}( Z_{\tau_{\partial \rmB^+[\wt{x};j]}}=z')}{ \bbP_z( Z_{\tau_{\partial \rmB^+[\wt{x};j]}}=z')}=1+O(\rme^{-\ol{c}k^\gamma})\,.
	\end{equation}
	On the other hand, by the same type of argument used when $\wt{u}_k=0$ (with the minor difference that the starting point of the excursion is now $w' \in \partial \rmB^+[x;i]$ instead of $z \in \partial_i \rmB^-[\wt{x};j]$, but~which does not affect the overall argument), we have 
	\begin{equation}\label{eq:new-decoup5}
		\bbP_{w',z'}( N^*_{1}[x;i]=0) = \frac{1}{j-i+1}(1+O(\rme^{-k^\gamma}))\,.
	\end{equation} Finally, inserting the estimates \eqref{eq:new-decoup2}--\eqref{eq:new-decoup3}--\eqref{eq:new-decoup4}--\eqref{eq:new-decoup5} together with \eqref{eq:new-cross1} into \eqref{eq:new-decoup1} gives \eqref{eq:new-c1} in the case $\wt{u}_k=1$ for $c_1:=\frac{1}{2}(\ol{c} \wedge 1)$.
		
	For $\wt{u}_k > 1$, by the strong Markov property again we can write the left-hand side of \eqref{eq:new-c1} as in \eqref{eq:new-decoup1}, with the only difference that, if we write $\ol{y}=(\ol{y}_1,\dots,\ol{y}_{\wt{u}_k}) \in (\partial_i \rmB_k^-[0;i] \times \partial \rmB^+[0;i])^{\wt{u}_k}$ and $x+\ol{y}_1:=(w,w')$, the term $\bbP_{w',z'}(N^*_{1}[x;i]=0)$ from \eqref{eq:new-decoup1} is now replaced with
	\begin{equation}\label{eq:new-decouple7}
		\bbP_{w',z'}(N^*_{1}[x;i]=\wt{u}_k-1, \ol{Y}_{(\wt{u}_k-1)}[x;i]=(\ol{y}_2,\dots,\ol{y}_{\wt{u}_k}))\,. 
	\end{equation} One may then factorize \eqref{eq:new-decouple7} by using the strong Markov property and inductive hypothesis, but this requires some care since the starting point here is $w' \in \partial \rmB^+[x;i]$ instead of $z \in \partial_i \rmB^-[\wt{x};j]$, as it was in the case $\wt{u}_k=1$. Nevertheless, one may deal with this by the same type of argument used for $\wt{u}_k=1$, which yields that \eqref{eq:new-decouple7} is in this case
	\begin{equation}
		\left(\frac{j-i}{j-i+1}\right)^{\wt{u}_k-1}\left(\frac{1}{j-i+1}\right) h_{k,i}(\ol{y}_2,\dots,\ol{y}_{\wt{u}_k}) \prod_{r=1}^{\wt{u}_k-1}(1+O_r(\rme^{-2c_1k^\gamma}))\,,
	\end{equation} where the $O_r(\rme^{-2c_1k^\gamma})$ error terms coming from each inductive step are all uniformly bounded~by $C\rme^{-2c_1k^\gamma}$ for some universal constant $C>0$. \eqref{eq:new-c1} now follows by noticing that for $\wt{u}_k \leq \rme^{c_1k^\gamma}$ (and since the error terms $O_r$ are uniformly bounded) we have
	\begin{equation}
		\prod_{r=1}^{\wt{u}_k-1}(1+O_r(\rme^{-2c_1k^\gamma}))=1+(\wt{u}_k-1)O(\rme^{-2c_1k^\gamma}).
	\end{equation} 
	
	Next, we continue with the proof of \eqref{eq:GR-3} in the case $\wt{v}_k > 1$. To this end, we notice that, by the independence of the different $[\wt{x};j]$-excursions, for each $\wt{v}_k > 1$ and $\ol{z}=(\ol{z}_1,\dots,\ol{z}_{\wt{v}_k}) \in \textrm{E}_{\wt{v}_k}[j]$, by splitting $u$ into the different number of $[x;i]$-downcrossings made in each $[\wt{x};j]$-excursion, we can decompose the numerator in the left-hand side of~\eqref{eq:lemacoup2} as
	\begin{equation}
		\sum_{\substack{(m_1,\dots,m_{\wt{v}_k}) \in \N_0^{\wt{v}_k} \\ \sum m_r = \wt{u}_k}} \prod_{r=1}^{\wt{v}_k} \bbP \Big( Z_t[x;i] = (\sqrt{k^\gamma m_i},(\ol{y}_{m_1+\dots+m_{r-1}+1},\dots,\ol{y}_{m_1+\dots+m_r})) \,\Big|\,Z_t[\wt{x};j] = (\sqrt{k^\gamma},\ol{z}_i)\Big)
	\end{equation} and, similarly, its denominator as
	\begin{equation}\label{eq:n-dec1}
		\sum_{\substack{(m_1,\dots,m_{\wt{v}_k}) \in \N_0^{\wt{v}_k} \\ \sum m_j = \wt{u}_k}} \prod_{i=1}^{\wt{v}_k} \bbP (T^{[i-1,j]}_1(i)=m_i)
	\end{equation} Therefore, since by the case $\wt{v}_k=1$ we have that
	\begin{equation}\label{eq:n-dec2}
		\frac{\bbP \Big( Z_t[x;i] = (u,\ol{y}) \,\Big|\,Z_t[\wt{x};j] = (\sqrt{k^\gamma},\ol{z})\Big)}{\bbP(T^{[i-1,j]}_1(i)=\wt{u}_k)}=h_{k,i}(\ol{y})(1+O(\rme^{-c_1k^\gamma}))\,	
	\end{equation} holds uniformly over all $t>0$, $(\sqrt{k^\gamma},\ol{z}) \in \cZ_k[j]$ and $(u,\ol{y}) \in \cZ_k[i]$ satisfying that $0 \leq \wt{u}_k \leq \rme^{ck^\gamma}$, the decompositions in \eqref{eq:n-dec1} and \eqref{eq:n-dec2} imply that \eqref{eq:GR-3} holds also uniformly over $1 < \wt{v}_k \leq \rme^{c_2k^\gamma}$ with $c_2:=\frac{c_1}{2}$ since, whenever $\wt{v}_k \leq \rme^{c_2k^\gamma}$, we have
	\begin{equation}
		\prod_{i=1}^{\wt{v}_k} (1+O_i(\rme^{-c_1k^\gamma}))=1+\wt{v}_kO(\rme^{-c_1k^\gamma}) =1 + O\big(\rme^{-c_2 k^\gamma}\big)\,.
	\end{equation} This concludes the proof.
	\end{proof} 
 
With the help of the Transfer Lemma, we can now give the proof of Proposition~\ref{lem:transfer}.

\begin{proof}[Proof of Proposition~\ref{lem:transfer}]  Notice that the case $T=2$ is an immediate consequence of Lemma~\ref{lem:rep1}. Thus, in order to conclude the proof, it remains to show the general case $T > 2$. To this end, by splitting into the different values of $Z_t[x;i]$ for $i=2,\dots,T-1$ and using the Markov property, we can write the numerator in the left-hand side of \eqref{eq:lemacoup2} as 
	\begin{equation}\label{eq:asymprw0}
		\sum_{(z_1,\dots,z_{T})}\prod_{i=2}^{T} \bbP \Big( Z_t[x_{i-1};i-1] = z_{i-1} \,\Big|\,Z_t[x_i;i] = z_i \Big)\,
	\end{equation} where, for convenience, for $i=1,\dots,T$ we write $z_i:=(u_i,\ol{y}_i)$, $x_i:=x$ if $i < T$, $x_i:=\wt{x}$ if $i=T$, and the sum is over all $T$-tuples $(z_1,\dots,z_{T})$ such that $z_1=(u,\ol{y})$, $z_T=(v,\ol{z})$ and $z_j \in \cZ_k^\eta(j)$ for each $j=2,\dots,T-1$. Now,  by Lemma~\ref{lem:rep1} we have that if $k$ is large enough, for each $i=2,\dots,T$ the $i$-th factor in the product in \eqref{eq:asymprw0} is
	\begin{equation}\label{eq:asymprw1}
		\bbP\Big(\sqrt{\wh{T}_{(v)}(i-1)}=u_{i-1}\,\,\Big|\,\sqrt{\wh{T}_{(v)}(i)}=u_{i}\Big)h_{k,i-1}(\ol{y}_{i-1})(1+O_i(\rme^{-c_2k^\gamma}))
	\end{equation} uniformly over $t>0$, $0 \leq \wt{u}_k \leq \rme^{c_2k^\gamma}$, $1 \leq \wt{v}_k \leq \rme^{c_2k^\gamma}$ and $(u_j,\ol{y}_j) \in \cZ_k^\eta[j]$ for each $j=2,\dots,T-1$, where the $O_i(\rme^{-c_2k^\gamma})$ error terms are all uniformly bounded in $i$ by $C\rme^{-c_2k^\gamma}$ for some $C>0$. Upon noticing that
	\begin{equation}
		\sum_{(y,y') \in \partial_i \rmB^-[0;i] \times \partial \rmB^+[0;i]} h_{k,i}((y,y'))=1
	\end{equation} since $\amalg_{\rmB^-[0;i]}$ and $\Pi_{\rmB^+[0;i]}$ are both probability measures, inserting the asymptotics from~\eqref{eq:asymprw1} in \eqref{eq:asymprw0} and summing over all $\ol{z}_j$ with $j=2,\dots,T-1$ yields that \eqref{eq:asymprw0} equals
	\begin{equation}\label{eq:asymprw2}
		h_k(\ol{y}) \sum_{(u_1,\dots,u_{T})}\left[\prod_{i=2}^{T} \bbP\Big(\sqrt{\wh{T}_{(v)}(i-1)}=u_{i-1}\,\,\Big|\,\sqrt{\wh{T}_{(v)}(i)}=u_{i}\Big)\right]\left[\prod_{i=2}^{T} (1+O_i(\rme^{-c_2k^\gamma}))\right]\,,
	\end{equation} where the sum is over all $T$-tuples $(u_1,\dots,u_{T})$ satisfying that $u_1:=u$, $u_T:=v$ and $u_j \in \mathcal{R}_k^\eta(j)$ for each $j=2,\dots,T-1$. Thus, if we take $c:=\frac{c_2}{2}$ then for all $T \leq \rme^{ck^\gamma}$ (and since the errors 	$O_i$ are uniformly bounded) we have
	\begin{equation}
		\left[\prod_{i=2}^{T} (1+O_i(\rme^{-c_2k^\gamma}))\right]=1+O(\rme^{-ck^\gamma})
	\end{equation} and hence \eqref{eq:asymprw2} becomes
	\begin{equation}
		\left[\sum_{(u_1,\dots,u_{T})}\left[\prod_{i=2}^{T} \bbP\Big(\sqrt{\wh{T}_{(v)}(i-1)}=u_{i-1}\,\,\Big|\,\sqrt{\wh{T}_{(v)}(i)}=u_{i}\Big)\right]\right]h_k(\ol{y})(1+O(\rme^{-ck^\gamma}))
	\end{equation} which, by the strong Markov property of $\wt{X}^{[0,T]}$, coincides with 
	\begin{equation}
		\bbP \Big( \big\{ \sqrt{\wh{T}_{(v)}(1)} = u\big\} \setminus \mathcal{T}^{\eta,[2,T-1]}_{(v)}\Big)h_{k}(\ol{y})(1+O(\rme^{-ck^\gamma}))
	\end{equation} and thus concludes the proof. 
\end{proof}

\subsection{Ballot estimates for local time of 1D CTSRW: proof of Proposition~\ref{prop:coupling2}}

To prove Proposition~\ref{prop:coupling2}, we will need the following result, which characterizes the distribution of the local time field $\wt{L}_t$ in terms of Brownian motion.

\begin{lem}\label{lem:law} Let $\wt{X}^{[0,T]}$ be a CTSRW on the linear graph $\{0,\dots,T\}$ with edge transition rates equal to~$1$ starting from $T$ and denote by $\wt{L}_t$ its local time field at $T$-time $t$. Then, for any $t>0$, we have  
	\begin{equation}\label{eq:law1}
		\left( \wt{L}_t(T-i) : i=0,\dots,T\right) \overset{d}{=} \left( Y_i : i=0,\dots,T\right), 
	\end{equation} where $Y=(Y_s)_{s \in [0,T]}$ is $\tfrac{1}{2}$ times a $0$-dimensional Bessel process starting from $2t$ (so that $Y_0=t$). In particular, the process $Y$ satisfies, for any Borel measurable function $\varphi : \R^{T+1} \to \R$, 
	\begin{equation}\label{eq:law2}
		\E\left( \varphi(Y_0,\dots,Y_T) ; Y_T \neq 0\right) = \E\left( \varphi(B_0^2,\dots,B_T^2) \sqrt{\frac{\sqrt{t}}{B_T} }\exp\left(-\frac{3}{16}\int_0^T B_s^{-2} \rmd s\right) ; \min_{s \in [0,T]} B_s > 0 \right)\,,
	\end{equation} where $(B_s)_{s \in [0,T]}$ is a Brownian motion starting from $\sqrt{t}$ with variance $\frac{1}{2}$.	
\end{lem}

\begin{proof} \eqref{eq:law1} follows from \cite[Lemma~7.7]{belius2017subleading}, whereas the proof of \eqref{eq:law2} can be found in \cite{Abe_Tree}. 
\end{proof}

\begin{proof}[Proof of Proposition~\ref{prop:coupling2}] The proof is similar to that of \cite[Lemma~6.1]{cortines2021scaling} but we include it~here for completeness.
	By Lemma~\ref{lem:law}, the left-hand side of \eqref{eq:lemacoupp} as $\bbP\big(W_T \in \rmd u \,\big|\,  W_0 = v\big)$ times 
	\begin{equation}\label{eq:integralasd}
		\sqrt{\frac{v}{u}}\E\Bigg(\exp\Bigg(-\frac{3}{8}\cdot\frac{k^\gamma}{2}\int_0^T W_s^{-2} \rmd s\Bigg) \,;\,\bigcap_{i=1}^{T-1}\Big\{ W_i \in \mathcal{Q}_k^\eta(i)\Big\} \cap \Big\{ \min_{s \in [0,T]} W_s > 0 \Big\} \,\Bigg|\,W_0=v\,,\,W_T=u\Bigg)\,,
	\end{equation} where $(W_s)_{s \in [0,T]}$ is a Brownian motion with variance $\tfrac{k^\gamma}{2}$ starting from $v$ (which we write formally as conditioning on $W_0=v$). 
	
	It follows from standard estimates for the Brownian bridge (see \cite[Section 4.3., Eq.~(3.40)]{Karatzas}) that, for any $a,b \in \R$ and $x \leq a \wedge b$, 
	\begin{equation}
		\bbP\Big( \min_{s \in [0,T]} W_s \leq x \,\Big|\, W_0=a\,,\,W_T=b\Big)=\exp\left(-\left(\frac{k^\gamma}{2}\right)^{-1}\frac{2(b-x)(a-x)}{T}\right)\,.
	\end{equation} In particular, by tilting the Brownian bridge by the linear function $H_s:=-\frac{\sqrt{2}}{2}(k+(T-s)k^\gamma)$, we obtain 
	\begin{equation}\label{eq:boundmin}
		\bbP\Big( \min_{s \in [0,T]} \big[W_s + H_s \big] \leq 0\,\Big|\,W_0=v\,,\,W_T=u\Big)= \exp\left(- \frac{4(\wh{v}-H_0)(\wh{u} - H_T)}{Tk^\gamma}\right)= O(\rme^{-\frac{3}{2}k})\,,
	\end{equation} where $\wh{u}:=u-\sqrt{2}k \in \mathfrak{I}_k^\eta(0)$ and $\wh{v}:=v-\sqrt{2}(k+Tk^\gamma) \in \mathfrak{I}_k^\eta(T)$.
	On the other hand, upon tilting by $2H_s$ and then dividing by $\sqrt{k^\gamma}$, we see that
	\begin{equation}\label{eq:eqprob}
		\bbP\Bigg(\bigcap_{i=1}^{T-1}\Big\{ W_i \in \mathcal{Q}_k^\eta(i) \Big\} \,\Bigg|\, W_0=v\,,W_T=u\Bigg) = \bbP\Bigg(\bigcap_{i=1}^{T-1}\Big\{ B_i \in \mathfrak{I}_k^{\eta+\gamma/2}(i) \Big\} \,\Bigg|\, B_0=\frac{\wh{v}}{\sqrt{k^\gamma}}\,,B_T=\frac{\wh{u}}{\sqrt{k^\gamma}}\Bigg)\,,
	\end{equation} where  $(B_s)_{s \in [0,T]}$ is a Brownian motion with variance $\tfrac{1}{2}$ starting from $\frac{\wh{v}}{\sqrt{k^\gamma}}$.
	
	Now, if $\eta+\gamma/2$ is sufficiently small, then standard estimates for Brownian motion (see e.g.~\cite{CHL17Supp}) show that the probability on the right-hand side of \eqref{eq:eqprob} has the same asymptotics as $k \to \infty$ as the probability of the same Brownian bridge remaining positive on $[0,T]$ or, equivalently, as 
	\begin{equation}
		\bbP\Big( \min_{s \in [0,T]} W_s > 0 \,\Big|\, W_0=\wh{v}\,,\,W_T=\wh{u}\Big)=\frac{4\wh{v}\wh{u}}{Tk^\gamma}(1+o_k(1))=O(k^{-\gamma/3})	\,,
	\end{equation} where, for the last two equalities to hold, we have used that $T \geq k$ and $\eta$ is taken small enough (depending only on $\gamma$) so that $\frac{\wh{v}\wh{u}}{Tk^\gamma}=O(k^{-\gamma/3})$ uniformly in $\wh{u}\in \mathfrak{I}_k^\eta(0)$ and $\wh{v} \in \mathfrak{I}_k^\eta(T)$.
	Finally, since by taking $\eta$ sufficiently small (depending only on $\gamma$) we have, for all $k$ large enough, that
	\begin{equation}\label{eq:eqprob3}
		\frac{\wh{v}\wh{u}}{Tk^\gamma} \geq \frac{1}{T} \geq \rme^{-k}
	\end{equation} and, in addition, on the complement of the event in \eqref{eq:boundmin} the integral in \eqref{eq:integral} is bounded from above by $4k^{-(1+\gamma)}$, in light of \eqref{eq:eqprob}--\eqref{eq:eqprob3} we see that \eqref{eq:integralasd} becomes
	\begin{equation}\label{eq:asymp1}
		4\cdot \sqrt{\frac{v}{u}}\cdot \frac{\wh{v}\wh{u}}{Tk^\gamma}(1+o_k(1))=\frac{4\wh{v}\wh{u}}{\sqrt{Tk^{1+\gamma}}}(1+o_k(1))\,.
	\end{equation}
	
	At the same time, if $\eta$ is taken sufficiently small, since $T \geq k$ we have that
	\begin{equation}\label{eq:asymp2}
		\begin{split}
			\bbP\big(W_T \in \rmd u \,\big|\,  W_0 = v\big) &= \frac{1}{\sqrt{\pi Tk^\gamma}} \exp\left( - \frac{(u-v)^2}{Tk^\gamma} \right) \\ &= \frac{1}{\sqrt{\pi}}\frac{1}{\sqrt{Tk^\gamma}} \exp\bigg(-2\sqrt{2}(\wh{v}-\wh{u}) - \frac{\wh{v}^2}{Tk^\gamma} - 2Tk^\gamma\bigg)(1+o_k(1))\,.
		\end{split}
	\end{equation}
	Combining \eqref{eq:asymp1} and \eqref{eq:asymp2}, we conclude that the left-hand side of \eqref{eq:lemacoupp} equals
	\begin{equation}
		\frac{4}{\sqrt{\pi}}\cdot \frac{1}{Tk^{1/2+\gamma}}\rme^{-2Tk^\gamma}\left(\wh{u}\rme^{2\sqrt{2}\wh{u}}\right)\left(\wh{v}\rme^{-2\sqrt{2}\wh{v}-\frac{\wh{v}^2}{Tk^\gamma}}\right)(1+o_k(1))\, 
	\end{equation} from where the result now immediately follows.
\end{proof}

\subsection{A ballot estimate with local time field repulsion: proof of Lemma~\ref{lem:asympforl}}

In this subsection, we give the proof of Lemma~\ref{lem:asympforl}. 

\begin{proof}[Proof of Lemma~\ref{lem:asympforl}] By conditioning on the pair $(\wh{L}_{(v)}(1),\wh{L}_{(v)}(T-1))$ and using the (spatial) Markov property of~$\wt{X}^{[0,T]}$, we can write the left-hand side of~\eqref{eq:asymp3} as
	\begin{equation}\label{eq:integral}
		\int_{\cQ^{\eta}_k(T-1)} \int_{\cQ^{\eta}_k(1)}  r_1(u,x) r_2(x,y)r_3(y,v)\,\rmd x \rmd y\,
	\end{equation} where
	\begin{equation}
		r_1(u,x):= \bbP\bigg( \sqrt{\wh{T}_{(v)}(1)}= u \,\bigg|\, \sqrt{\wh{L}_{(v)}(1)} = x\bigg)\,,
	\end{equation}
	\begin{equation}
		r_2(x,y):=\bbP \bigg( \sqrt{\wh{L}_{(v)}(1)} \in \rmd x \,,\, \sqrt{\wh{L}_{(v)}(i)} \in \cQ^{\eta}_k(i) \,\forall i=2,\dots,T-2 \,\bigg|\, \sqrt{\wh{L}_{(v)}(T-1)} = y\bigg)
	\end{equation} and
	\begin{equation}
		r_3(y,v):= \bbP\bigg( \sqrt{\wh{L}_{(v)}(T-1)} \in \rmd y\bigg)\,.
	\end{equation} Our first task is now to show that the integral in \eqref{eq:integral} is concentrated around values of $x$ and~$y$ which are at most $k^{\frac{3}{2}\gamma}$ away from $u$ and $v$, respectively. To be more precise, we will show that, uniformly over all $u \in \cR^{\eta}_k(1)$ and $v \in \cR^{\eta}_k(T)$, 
	\begin{equation}\label{eq:integraldisc}
		\bbP\Big( \Big\{\sqrt{\wh{T}_{(v)}(1)} = u \,,\,  \Big|\sqrt{\wh{L}_{(v)}(1)} -u \Big| \vee \Big|\sqrt{\wh{L}_{(v)}(T-1)} -v \Big|> k^{\frac{3}{2}\gamma}\Big\} \setminus \cL^{\eta,[1,T-1]}_{(v)}\Big)= r_{k,T}(\wh{u},\wh{v}) o_k(1)\,,
	\end{equation} where $r_{k,T}$ is given by \eqref{eq:funcfg}.
	We begin by establishing that, with high probability on the event on the left-hand side of \eqref{eq:asymp3}, $\sqrt{\wh{L}_{(v)}(T-1)}$ is at distance at most $k^{\frac{3}{2}\gamma}$ from $v$.
	For this purpose, we first notice that, since
	$\wt{L}_{(v)}(T-1)\sim \Gamma(\wt{v}_k,1)$ by \eqref{eq:proplta}, by Lemma~\ref{lem:compound} we have, for any $z \geq 0$,
	\begin{equation}\label{eq:condeq0}
		\bbP\left( \Big|\sqrt{\wh{L}_{(v)}(T-1)} - v\Big| \geq z\right) \leq 2\rme^{-\frac{z^2}{k^\gamma}}\,.
	\end{equation} On the other hand, in light of \eqref{eq:propltb}, Lemma~\ref{lem:compound} again we have that, for any $z \geq u-v$,
	\begin{equation}\label{eq:condeq1}
		\bbP\Big( \sqrt{\wh{T}_{(v)}(1)} \leq u \,\Big|\, \sqrt{\wh{L}_{(v)}(T-1)}= v+z\Big) \leq \exp\Big\{-\frac{(v +z - u)^2}{(T-1)k^\gamma}\Big\}\,.
	\end{equation} In particular, if we denote as usual $\wh{v}:=v - \sqrt{2}(k+(T-1)k^\gamma)$ and $\wh{u}:=u-\sqrt{2}k$ then, whenever
	$v + z \in \cQ^{\eta}_k(T-1)$, for all $k$ large enough not only the condition $z \geq u-v $ is satisfied but also that $|\wh{v}+z| \leq 2 (Tk^\gamma)^{1/2+\eta}$, in which case, since $T \geq k$, for $\eta$ small enough we can further~bound
	\begin{align}\label{eq:condeq2}
		\exp\Big\{-\frac{(v +z - u)^2}{(T-1)k^\gamma}\Big\} &\leq 2 \exp\Big\{-2(T-1)k^\gamma -2\sqrt{2}(\wh{v}-\wh{u}+z) -\frac{(\wh{v}+z)^2}{(T-1)k^\gamma}\Big\} \nonumber \\ &\leq r_{k,T}(\wh{u},\wh{v})\rme^{k^\gamma+4|z|}\,,
	\end{align} where, to obtain the second inequality, we have used that $\wh{v}=o(Tk^\gamma)$, $\wh{u}\wh{v}\geq 1$ and $T \leq \rme^{\tfrac{1}{2}k^\gamma}$. Therefore, by conditioning on $\wt{L}_{(v)}(T-1)$ and then combining the estimates from \eqref{eq:condeq0}--\eqref{eq:condeq2}, for all $k$ large enough we obtain that, for any $m \in \bbZ \setminus \{0\}$, 
	\begin{equation}\label{eq:decompint0}
		\bbP\Big( \Big\{\sqrt{\wh{T}_{(v)}(1)} \leq u \,,\, \sqrt{\wh{L}_{(v)}(T-1)} -v \in  [m,m+1]\Big\} \setminus \cL^{\eta,T-1}_{(v)}\Big)\leq r_{k,T}(\wh{u},\wh{v})\rme^{2k^\gamma +4|m| - \frac{(|m|-1)^2}{k^\gamma}}\,.
	\end{equation} From here, upon summing \eqref{eq:decompint0} over all $m \in \bbZ$ with $|m| \geq k^{\frac{3}{2}\gamma}$, the union bound yields that
	\begin{equation}\label{eq:boundintegral1}
		\bbP\Big( \Big\{\sqrt{\wh{T}_{(v)}(1)} \leq u \,,\, \Big|\sqrt{\wh{L}_{(v)}(T-1)} -v \Big| > k^{\frac{3}{2}\gamma}\Big\} \setminus \cL^{\eta,T-1}_{(v)}\Big)\leq r_{k,T}(\wh{u},\wh{v})\rme^{-\frac{1}{2}k^{2\gamma}}
	\end{equation} holds, for all $k$ sufficiently large, uniformly over all $u \in \cR^{\eta}_k(1)$ and $v \in \cR^{\eta}_k(T)$.
	
	Similarly, with high probability on the event in~\eqref{eq:asymp3}, $\sqrt{\wh{L}_{(v)}(1)}$ is at a distance at most~$k^{\frac{3}{2}\gamma}$ away from $u$. Indeed, recalling \eqref{eq:propltb}, by Lemma~\ref{lem:compound} we have that, for any $z \geq 0$,
	\begin{equation}\label{eq:condeq3}
		\bbP\Big( \big|\sqrt{\wh{T}_{(v)}(1)} - \sqrt{k^\gamma\theta} \big| \geq z \,\Big|\,\sqrt{\wh{L}_{(v)}(1)} = \sqrt{k^\gamma \theta} \Big) \leq 2\rme^{-\frac{z^2}{k^\gamma}}\,.
	\end{equation} Moreover, since $\wt{L}_{(v)}(1) \sim \textrm{Bi-Exp}\big(\wt{v}_k,\frac{1}{T-1};\frac{1}{T-1}\big)$ by \eqref{eq:proplta}, by Lemma~\ref{lem:compound} we have , for any $z \in \bbR$ with $|u +z| \leq v$, 
	\begin{equation}\label{eq:condeq4}
		\bbP\Big( \sqrt{\wh{L}_{(v)}(1)} \leq u + z\Big) \leq \exp\Big\{- \frac{(v - (u + z))^2}{(T-1)k^\gamma}\Big\}\,.
	\end{equation} In particular, if for $\sqrt{k^\gamma \theta} \in \cQ^{\eta}_k(1)$ we define $z:=\sqrt{k^\gamma \theta} - u +1$ then $|\wh{u} + z| \leq 2(k+k^\gamma)^{1/2+\eta}$ for all $k$ large enough, so that then the condition $|u +z| \leq v$ holds and, in addition, since $T \geq k$ and $\wh{v} \leq (k+Tk^\gamma)^{1/2+\eta}$, for $\eta$ small enough we can further bound
	\begin{align}
		\exp\Big\{- \frac{(v - (u + z))^2}{(T-1)k^\gamma}\Big\} &\leq 2 \exp\Big\{- 2(T-1)k^\gamma -2\sqrt{2}(\wh{v}-\wh{u}-z) - \frac{\wh{v}^2}{(T-1)k^\gamma}\Big\} \nonumber \\ &  \leq r_{k,T}(\wh{u},\wh{v})\rme^{k^\gamma +2\sqrt{2}|z|}\label{eq:condeq5}\,,
	\end{align} where, once again, the second inequality follows from the fact that $T\leq \rme^{\frac{1}{2}k^\gamma}$ and $\wh{u}\wh{v}\geq 1$.
	Finally, by conditioning on $\wt{L}_{(v)}(1)$ and using \eqref{eq:condeq3}--\eqref{eq:condeq5}, a straightforward computation yields that, for any $m \in \bbZ \setminus \{0\}$, 
	\begin{equation}\label{eq:decompint}
		\bbP\Big( \Big\{\sqrt{\wh{T}_{(v)}(1)} = u \,,\, \sqrt{\wh{L}_{(v)}(1)} -u \in  [m,m+1]\Big\} \setminus \cL^{\eta,1}_{(v)}\Big)\leq r_{k,T}(\wh{u},\wh{v})\rme^{2k^\gamma +2\sqrt{2}|m| - \frac{(|m|-1)^2}{k^\gamma}}
	\end{equation} from where, by summing \eqref{eq:decompint} over all $m \in \bbZ$ with $|m| \geq k^{\frac{3}{2}\gamma}$, the union bound yields that
	\begin{equation}\label{eq:closeu}
		\bbP\Big( \Big\{ \sqrt{\wh{T}_{(v)}(1)} = u\,,\, \Big|\sqrt{ \wh{L}_{(v)}(1)} -u \Big| > k^{\frac{3}{2}\gamma}\Big\} \setminus \cL^{\eta,1}_{(v)}\Big)\leq r_{k,T}(\wh{u},\wh{v})\rme^{-\frac{1}{2}k^{2\gamma}}
	\end{equation} holds for all $k$ sufficiently large uniformly over $u \in \cR^\eta_k(1)$ and $v \in \cR^\eta_k(T)$. Together with~\eqref{eq:boundintegral1}, this implies \eqref{eq:integraldisc}.
	
	In light of \eqref{eq:integraldisc}, to conclude the proof it suffices to show that for all $k$ sufficiently large and $T \in [2k,\rme^{\frac{1}{2}k}]$, we have that
	\begin{equation}
		\int_{\cQ^{\eta}_k(T-1)} \int_{\cQ^{\eta}_k(1)}  r_1(u,x) r_2(x,y)r_3(y,v) 1_{\big\{|x-u| \vee |y-v| \leq k^{\frac{3}{2}\gamma}\big\}}\,\rmd x \rmd y = h_{k,T}(1+o_k(1))
	\end{equation} holds uniformly over all $u \in \cR^{\eta'}_k(1)$ and $v \in \cR^{\eta'}_k(T)$, where the functions $r_i$ are those from~ \eqref{eq:integral}.
	To this end, on the one hand we observe that, by \eqref{eq:propltb},
	\begin{equation}
		r_1(u,x)=\frac{\rme^{-\wt{x}_k}}{\wt{u}_k!}\left(\wt{x}_k\right)^{\wt{u}_k}\,.
	\end{equation} In particular, since $\wt{u}_k! = \sqrt{2\pi \wt{u}_k} (\wt{u}_k)^{\wt{u}_k} \rme^{-\wt{u}_k}(1+O(\wt{u}^{-1}_k))$ by Stirling's approximation and also $\wt{x}_k= \wt{u}_k(1+\frac{x-u}{u})^2$, a straightforward computation using the asymptotics $1+z=\rme^{z-\frac{1}{2}z^2}(1+O(z^3))$ as $z \to 0$ shows that, if $\gamma$ is chosen small enough, then one has 
	\begin{equation}\label{eq:asympeq1}
		r_1(u,x):= \frac{\sqrt{k^\gamma}}{\sqrt{2\pi}u} \rme^{-2\frac{(x-u)^2}{k^\gamma}}(1+o_k(1))
	\end{equation} uniformly over all $x$ with $|x-u|\leq k^{\frac{3}{2}\gamma}$. 
	On the other hand, since $\wt{L}_{(v)}(T-1) \sim \Gamma(\wt{v}_k,1)$ by~\eqref{eq:proplta}, we have that 
	\begin{equation}
		r_3(y,v)=\frac{2y}{k^\gamma}\frac{\rme^{-\wt{y}_k}}{(\wt{v}_k-1)!}{(\wt{y}_k)^{\wt{v}_k-1}}\,.
	\end{equation} By an analogous computation to the one yielding \eqref{eq:asympeq1}, we obtain that, if $\gamma$ is small enough, 
	\begin{equation}\label{eq:finalasymp3}
		r_3(y,v)=\frac{2}{\sqrt{2\pi k^\gamma}}\rme^{-2\frac{(y-v)^2}{k^\gamma}}(1+o_k(1)).
	\end{equation} Finally, by making a few minor modifications to the proof of Proposition~\ref{prop:coupling2}, one can show~that, if $T \in [2k,\rme^k]$, then
	\begin{equation}
		r_2(x,y)=\frac{4}{\sqrt{\pi}}\cdot \frac{1}{Tk^{1/2+\gamma}} \rme^{-2(T-2)k^\gamma}\left(\wh{x}\rme^{2\sqrt{2}\wh{x}}\right)\left(\wh{y}\rme^{-2\sqrt{2}\wh{y}-\frac{\wh{y}^2}{(T-2)k^\gamma}}\right)(1+o_k(1))\,
	\end{equation} holds uniformly over all $x \in \cQ^\eta_k(1)$ and $y \in \cQ^\eta_k(T-1)$, where as usual we write $\wh{x}:=x - \sqrt{2}(k+k^\gamma)$ and $\wh{y}:=y-\sqrt{2}(k+(T-1)k^\gamma)$. However, since $\wh{x}=(x-u)+\wh{u}-\sqrt{2}k^\gamma$ and $\wh{y}=(y-v)+\wh{v}$, if $\eta$ is small enough (depending only on $\gamma$), then whenever $|x-u| \vee |y-v| \leq k^{\frac{3}{2}\gamma}$ we have that
	\begin{equation}\label{eq:finalasymp2}
		r_2(x,y)=\frac{4}{\sqrt{\pi}}\cdot \frac{1}{Tk^{1/2+\gamma}} \rme^{-2Tk^\gamma}\left(\wh{u}\rme^{2\sqrt{2}((x-u)+\wh{u})}\right)\left(\wh{v}\rme^{-2\sqrt{2}((y-v)+\wh{v})-\frac{\wh{v}^2}{Tk^\gamma}}\right)(1+o_k(1))\,.
	\end{equation} Combining \eqref{eq:asympeq1}-\eqref{eq:finalasymp3}-\eqref{eq:finalasymp2} with \eqref{eq:integraldisc}, and since the condition $|x-u| \vee |y-v| \leq k^{\frac{3}{2}\gamma}$ implies that $x \in \cQ_k^{\eta}(1)$ and $y \in \cQ_k^{\eta}(T-1)$ uniformly over all $u \in \cR^{\eta'}_k(1)$ and $v \in \cR^{\eta'}_k(T)$ for all $k$ large enough whenever $\eta$ is chosen small enough depending only on $\gamma$ (here is where we use the fact that $\eta' < \eta$), a straightforward computation shows that the left-hand side of \eqref{eq:asymp3} is equal to
	\begin{equation}\label{eq:finalasymp4}
		r_{k,T}(\wh{u},\wh{v}) \left( \int_{u-k^{\frac{3}{2}}}^{u+k^{\frac{3}{2}}}\sqrt{\frac{2}{\pi k^\gamma}}\rme^{-\frac{(x-u-\frac{k^\gamma}{\sqrt{2}})^2}{\frac{k^\gamma}{2}}}\rmd x\right)\left( \int_{v-k^{\frac{3}{2}}}^{v+k^{\frac{3}{2}}}\sqrt{\frac{2}{\pi k^\gamma}}\rme^{-\frac{(y-v+\frac{k^\gamma}{\sqrt{2}})^2}{\frac{k^\gamma}{2}}}\rmd y\right)(1+o_k(1))\,.
	\end{equation} Since the two integrals in \eqref{eq:finalasymp4} each coincide with the probability that a standard Gaussian random variable lies in an interval of length $2k^\gamma(1+o_k(1))$ around the origin, we conclude that these two integrals tend to $1$ as $k \to \infty$ and thus \eqref{eq:asymp3} is finally equal to $r_{k,T}(\wh{u},\wh{v})(1+o_k(1))$ uniformly over all $u \in \cR^{\eta'}_k(1)$ and $v \in \cR^{\eta'}_k(T)$ as claimed. 
\end{proof}

\subsection{From local time to downcrossing repulsion: proof of Lemma~\ref{lem:asympdisc}}

The last step in the proof is to give the proof of Lemma~\ref{lem:asympdisc}, which we do next.

\begin{proof}[Proof of Lemma~\ref{lem:asympdisc}] Since $T \leq \rme^{\frac{1}{2}k^\gamma}$, by the union bound we see that, in order to show \eqref{eq:step3}, it will suffice to prove that, for each $i=1,\dots,T-1$, 
	\begin{equation}\label{eq:boundtp}
		\bbP \Big( \Big\{ \sqrt{ \wh{T}_{(v)}(1)} = u \Big\} \cap \cJ^{k,i}_{(v)} \setminus \cL_{(v)}^{\eta,[1,T-1]}\Big)\,\vee \,\bbP \Big( \Big\{ \sqrt{ \wh{T}_{(v)}(1)} = u \Big\} \cap \cJ^{k,i}_{(v)} \setminus  \cT_{(v)}^{\eta,[2,T-1]}
		\Big) \leq r_{k,T}(\wh{u},\wh{v})\rme^{-k^\gamma} 
	\end{equation} holds uniformly over all $u \in \cR^\eta_k(1)$ and $v \in \cR^\eta_k(T)$ whenever $\eta$ is small enough, $T \in [2k,\rme^{\frac{1}{2}k^\gamma}]$ and $k$ is large enough (but not depending on $i$). We will only show the first half of the inequality, i.e., 
	\begin{equation}\label{eq:boundtppri}
		\bbP \Big( \Big\{ \sqrt{ \wh{T}_{(v)}(1)} = u \Big\} \cap \cJ^{k,i}_{(v)} \setminus \cL_{(v)}^{\eta,[1,T-1]}\Big)\leq r_{k,T}(\wh{u},\wh{v})\rme^{-k^\gamma} 
	\end{equation} as the proof of the other half is very similar. We focus first on the case $1 < i < T-1$.
	
	In this case, our first step will be to show that we can restrict ourselves to the event in which $\sqrt{\wh{L}_{(v)}(i)} \leq v$, i.e., that, if $\eta$ is chosen small enough then, uniformly over  $i \in [2,T-2]$, $u \in \cR^\eta_k(1)$ and $v \in \cR^\eta_k(T)$, one has that
	\begin{equation}\label{eq:ex1}
		\bbP \Big( \sqrt{ \wh{T}_{(v)}(1)} = u \,,\,\sqrt{\wh{L}_{(v)}(i)} > v\Big) \leq r_{k,T}(\wh{u},\wh{v})\rme^{-k^\gamma}\,	
	\end{equation} for all $T \in [2k,\rme^{\frac{1}{2}k^\gamma}]$ provided that $k$ is taken large enough. 
	
	To this end, let us notice that, by \eqref{eq:propltb}, by Lemma~\ref{lem:compound} we have that, if $\eta$ is small enough (depending only on $\gamma$) then, for all $k$ large enough and $T \in [2k,\rme^{\frac{1}{2}k^\gamma}]$ we have~$\wh{v} \geq \wh{u}$,  so that
	\begin{equation}\label{eq:ex2}
		\sup_{\theta > v} \bbP\Big( \sqrt{\wh{T}_{(v)}(1)} \leq u	\,\Big|\, \sqrt{\wh{L}_{(v)}(i)}=\theta\Big) \leq \rme^{-\frac{(v-u)^2}{ik^\gamma}}\leq r_{k,T}(\wh{u},\wh{v})\rme^{-\frac{3}{2}k^\gamma}
	\end{equation} by a straightforward computation similar to that one yielding~\eqref{eq:condeq2} (here is the only moment where we use that $i < T-1$), and therefore, upon conditioning on $\wh{L}_{(v)}(i)$ and then using \eqref{eq:ex2}, \eqref{eq:ex1} now immediately follows.
	
	On the other hand, by conditioning first on $\wh{L}_{(v)}(i)$ and then on $\wh{T}_{(v)}(i)$, for $i > 1$ we can~bound the leftmost probability in \eqref{eq:boundtp} on the event $\{ \wh{L}_{(v)}(i) \leq v\}$ from above by 
	\begin{equation}\label{eq:split1}
		\int_{\cQ_k^\eta(i)}  \sum_{w \in \cN_k } 1_{\{|w-\theta| > q_{k,\eta}(i)\}} 1_{\{\theta \leq v\}}\, r^{(i)}_1(w)r^{(i)}_2(w,\theta)r^{(i)}_3(\theta) \,\rmd \theta\,,
	\end{equation} where we abbreviate $q_{k,\eta}(i):=(k+ik^\gamma)^{1/2-\eta}$ and define the functions $r_{j}^{(i)}$ as 
	\begin{equation}
		r^{(i)}_1(w):=\bbP \Big( \sqrt{ \wh{T}_{(v)}(1)} = u \,,\, \sqrt{\wh{L}_{(v)}(i-1)} \in \cQ^\eta_k(i-1) \Big| \sqrt{ \wh{T}_{(v)}(i)} = w \Big)\,,
	\end{equation}
	\begin{equation}\label{eq:defr2}
		r^{(i)}_2(w,\theta):=\bbP \Big( \sqrt{ \wh{T}_{(v)}(i)} = w \Big| \sqrt{ \wh{L}_{(v)}(i)}  = \theta \Big)\,,
	\end{equation} and
	\begin{equation}\label{eq:defr3}
		r^{(i)}_3(\theta):= \bbP \Big(\sqrt{ \wh{L}_{(v)}(i)}  \in \rmd \theta\Big)\,.
	\end{equation}
	By conditioning now on $\wh{L}_{(v)}(i-1)$, we can further bound 
	\begin{equation}\label{eq:boundr1f}
		r^{(i)}_1(w) \leq \sup_{\zeta \in \cQ^\eta_k(i-1)}\bbP \Big( \sqrt{ \wh{T}_{(v)}(1)} = u  \,\Big|\,  \sqrt{\wh{L}_{(v)}(i-1)} = \zeta \,,\, \sqrt{ \wh{T}_{(v)}(i)} = w \Big)\,.
	\end{equation} Since, conditional on $\sqrt{\wh{L}_{(v)}(i-1)}=\zeta$ and $\sqrt{\wh{T}_{(v)}(i)}=w$,  we have $T_{\wt{v}_k}(1) \sim \Poigeo(\frac{\wt{\zeta}_k}{i-1};\frac{1}{i-1})$, if $i \geq k^{\frac{1}{2}+\eta-\gamma}$ and $k$ is large enough then the condition $\zeta \in \cQ^\eta_k(i-1)$ implies that $\zeta \geq u$ so that, by Lemma~\ref{lem:compound}, a simple computation using the fact that $\wh{\zeta} \geq 0$ for $\cQ^\eta_k(i-1)$ yields that
	\begin{align}
		\sup_{\zeta \in \cQ^\eta_k(i-1)} \Big( \sqrt{ \wh{T}_{(v)}(1)} = u  \,\Big|\,  \sqrt{\wh{L}_{(v)}(i-1)} = \zeta\,,\, \sqrt{ \wh{T}_{(v)}(i)} = w \Big) &\leq \rme^{-2(i-1)k^\gamma +2\sqrt{2}\wh{u}} \nonumber \\ & \leq \rme^{-2(i-1)k^\gamma +4k^{\frac{1}{2}+\eta}} \label{eq:h1bound}
	\end{align} uniformly over $u \in \cR^\eta_k(1)$ for all $k$ large enough. But, since the bound in \eqref{eq:h1bound} is trivially true for all $i \leq k^{\frac{1}{2}+\eta-\gamma}$, we conclude that, if $k$ is large enough, then for each $i=2,\dots,T-1$ we have that 
	\begin{equation}
		r^{(i)}_1(w) \leq \rme^{-2(i-1)k^\gamma +4k^{\frac{1}{2}+\eta}+2\sqrt{2}\wh{u}}
	\end{equation} uniformly over all $w$ in the range considered in~\eqref{eq:split1} and all $u \in \cR^\eta_k(1)$. 
	
	On the other hand, in light of \eqref{eq:propltb}, by Lemma~\ref{lem:compound} we have that, uniformly over $\theta \in \cQ^\eta_k(i)$, 
	\begin{align}\label{eq:h2bound}
		\sum_{w \in \cN_k} 1_{\{|w-\theta| > q_{k,\eta}(i)\}}r_2^{(i)}(w,\theta) = \bbP \Big( \Big|\sqrt{\wh{T}_{(v)}(i)} - \theta\Big| > q_{k,\eta}(i) \,\Big|\, \sqrt{\wh{L}_{(v)}(i)} = \theta\Big) &\leq 2\rme^{-\frac{(q_{k,\eta}(i))^2}{k^\gamma}}\nonumber \\ &\leq 2\rme^{-(k+ik^\gamma)^{1-2\eta-\gamma}}\,.
	\end{align} Finally, if we define $\wh{\xi}:=\min\{ v ,\sqrt{2}(k+ik^\gamma) + (k+ik)^{\frac{1}{2}+\eta}\} - \sqrt{2}(k+(i-1)k^\gamma)$ then, by~\eqref{eq:proplta} and Lemma~\ref{lem:compound}, and since $\wh{v} \in \mathfrak{I}_k^\eta(T-1)$ and $\wh{\xi} \leq \sqrt{2}k^\gamma +(k+ik^\gamma)^{\frac{1}{2}+\eta}$, we have that
	\begin{align}
		\int_{\cQ^\eta_k (i)} 1_{\{ \theta \leq v\}} r_3^{(i)}(\theta) \,\rmd \theta &\leq \bbP \Big( \sqrt{\wh{L}_{(v)}(i)} \leq \min\{ v ,\sqrt{2}(k+ik^\gamma) + (k+ik^\gamma)^{\frac{1}{2}+\eta}\} \Big) \nonumber \\ &\leq \rme^{-2(T-i)k^\gamma -2\sqrt{2}(\wh{v}-\wh{\xi})-\frac{(\wh{v}-\wh{\xi})^2}{Tk^\gamma}} \leq \rme^{-2(T-i)k^\gamma -2\sqrt{2}\wh{v} -\frac{\wh{v}^2}{Tk^\gamma}+ 4(k+ik^\gamma)^{\frac{1}{2}+\eta}}
	\end{align} for all $k$ large enough, uniformly over all $v \in \cR^\eta_k(i)$. Therefore, combining the estimates obtained on $r_j^{(i)}$ for $j=1,2,3$, we conclude that, if $\eta$ is chosen small enough so that $1-2\eta-\gamma > \frac{1}{2}+\eta$, for all $T \in [2k,\rme^{\frac{1}{2}k^\gamma}]$ and $k$ large enough (but not depending on $i$) we have 
	\begin{equation}
		\bbP \Big( \Big\{ \sqrt{ \wh{T}_{(v)}(1)} = u \,,\,\sqrt{\wh{L}_{(v)}(i)} \leq v\Big\} \cap \cJ^{k,i}_{(v)} \setminus \cL_{(v)}^{\eta,[1,T-1]}\Big) \leq r_{k,T}(\wh{u},\wh{v})\rme^{-\frac{1}{2}(k+ik^\gamma)^{1-2\eta-\gamma}}\,.
	\end{equation} Combining this estimate with \eqref{eq:ex2} immediately  gives \eqref{eq:boundtppri} for $1 < i <T-1$.
	
	On the other hand, the case $i=1$ follows at once from \eqref{eq:closeu} if $\gamma$ and $\eta$ are chosen small enough so that $\frac{3}{2}\gamma < \frac{1}{2}-\eta$.
	
	Finally, to deal with the case $i=T-1$, we first note that, by mimicking the proof of~\eqref{eq:boundintegral1}, one can show that
	\begin{equation}
		\bbP \Big( \Big\{ \sqrt{ \wh{T}_{(v)}(1)} = u \,,\,\Big|\sqrt{\wh{L}_{(v)}(T-2)} - v\Big| > k^{\frac{3}{2}\gamma}\Big\} \setminus \cL_{(v)}^{\eta,T-2}\Big) \leq r_{k,T}(\wh{u},\wh{v})\rme^{-\frac{1}{4}k^{2\gamma}}\,,
	\end{equation} and thus it will suffice to show that
	\begin{equation}\label{eq:prob2}
		\bbP \Big( \Big\{ \sqrt{ \wh{T}_{(v)}(1)} = u \,,\,\sqrt{\wh{L}_{(v)}(T-2)} \geq v-k^{\frac{3}{2}\gamma}\Big\} \cap \cJ^{k,T-1}_{(v)} \Big) \leq r_{k,T}(\wh{u},\wh{v})\rme^{-2k^\gamma} 
	\end{equation} holds uniformly over all $u \in \cR^\eta_k(1)$ and $v \in \cR^\eta_k(T)$, provided that $\eta$ is chosen sufficiently small, $T \in [k,\rme^{\frac{1}{2}k^\gamma}]$ and $k$ is taken sufficiently large. Now, by repeating the decomposition in \eqref{eq:split1}, we can write the probability in~\eqref{eq:prob2} as
	\begin{equation}\label{eq:decompint2}
		\int  \sum_{\substack{w \in \cN_k \\ |w-\theta| > q_{k,T}(T-1)}}  r^{(T-1)}_1(w)r^{(T-1)}_2(w,\theta)r^{(T-1)}_3(\theta) \,\rmd \theta\,,
	\end{equation} where $r_2^{(T-1)}$ and $r_3^{(T-1)}$ are as in \eqref{eq:defr2} and \eqref{eq:defr3} respectively (with $T-1$ in place of $i$), and $r_1^{(T-1)}$ is given by 
	\begin{equation}
		r_1^{(T-1)}:=\bbP\Big( \sqrt{\wh{T}_{(v)}(1)}= u\,,\,\sqrt{\wh{L}_{(v)}(T-2)} \geq v-k^{\frac{3}{2}\gamma}\,\Big|\, \sqrt{\wh{T}_{(v)}(T-1)}=w\Big)\,. 
	\end{equation} As in \eqref{eq:boundr1f}--\eqref{eq:h1bound} and since $v-k^{\frac{3}{2}\gamma}\geq u$ for all $k$ large enough, we can bound 
	\begin{equation}
		\sup r_1^{(T-1)} \leq \sup_{\zeta \geq v - k^{\frac{3}{2}\gamma}}\bbP \Big( \sqrt{ \wh{T}_{(v)}(1)} \leq u  \,\Big|\,  \sqrt{\wh{L}_{(v)}(T-2)} = \zeta\Big) \leq 2\rme^{-2(T-1)k^\gamma -2\sqrt{2}(\wh{v}-\wh{u}-k^{\frac{3}{2}\gamma})-\frac{\wh{v}^2}{Tk^\gamma}} 	
	\end{equation} provided that $\gamma$ and $\eta$ are chosen small enough and $k$ taken large enough. Moreover, as~in~\eqref{eq:h2bound} and since $T \geq 2k$, for all $k$ large enough we have that 
	\begin{equation}
		\sum_{w \in \cN_k} 1_{\{ |w-\theta| > q_{k,\eta}(T-1)\}} r_2^{(T-1)}(w,\theta) \leq 2\rme^{-\frac{(q_{k,\eta}(T-1))^2}{k^\gamma}}\leq 2\rme^{-(k+(T-1)k^\gamma)^{1-2\eta-\gamma}} \leq 2\rme^{-k^{(1+\gamma)(1-2\eta-\gamma)}}
	\end{equation} uniformly over all $\theta > 0$. If $\gamma$ and $\eta$ are chosen small enough then, upon inserting these last two estimates in \eqref{eq:decompint2}, \eqref{eq:prob2} now immediately follows.
\end{proof}

\section{Phase B: Proof of Theorem~\ref{t:2b}} \label{sec:proofb}

In this section we prove Theorem~\ref{t:2b}. The proof is similar to the proof of Theorem~2.4 in~\cite{Tightness}, in which asymptotics are derived for the cover time, but not the last visited vertex of the clustered set. We recall that $\tau_\rmA$ and $\ol{\tau}_\rmA$ are the first hitting (real) time of $\rmA$ and the first return (real) time to that set, by the walk $\bfX$, as defined in~\eqref{e:7.1i}, and 
that $\bbP_x$ is the probability measure under which the walk $\bfX$ starts from vertex $x \in \rmD_n$. The following three lemmas were shown in~\cite{Tightness}.
\begin{lem}[Lemma 7.6 in \cite{Tightness}]
	\label{l:8.2}
	For all $n$ large enough, $x \in \rmD_n^\circ$ and $y \in \partial \rmB(x; r_n)$, 
	\begin{equation}
		\bbP_y \big(\tau_{\partial} < \wc{\bfT}_{\rmB(x;r_n)} \big) \leq n^{-\eta_0/3}\,,
	\end{equation}
	where $\eta_0$ is as in~\eqref{e:2.3}.
\end{lem}

\begin{lem}[Lemma 7.7 in~\cite{Tightness}]
	\label{l:8.3}
	Let $\rmA \subseteq\rmD_n^\circ$ be an $(r_n,n-r_n)$-clustered set and $\bfA$ a minimal $r_n$-cover of it in the sense of~\eqref{e:2.10t} such that $|\bfA|\leq n^{1/2+\eta_0/4}$, where $\eta_0$ is as in~\eqref{e:2.3}. Then, as $n \to \infty$, 
	\begin{equation}
		\bbP_\partial \Big(\tau_{\partial \rmB(x;r_n)} < \ol{\tau}_\partial \wedge \min_{y \in \bfA \setminus \{x\}}
		\tau_{\partial \rmB(y;r_n)} \Big)
		= \frac{2\pi}{\deg(\partial) n}\Big(1+o(n^{-\eta/4})\Big) \,,
	\end{equation}
	uniformly over all such sets $\rmA$ and $x \in \bfA$.
\end{lem}

\begin{lem}[Lemma 3.2 in~\cite{Tightness}]
	\label{l:8.3n}
	Let $\rmA$, $\bfA$ be as in Lemma~\ref{l:8.3}, $z \in \bfA$ and $x \in \rmB(z;r_n)$. Then,%\ToO{Was $y \in \bfA \setminus \{x\}$, but it didn't seem correct.}
	\begin{equation}
		\bbP_x \Big(\min_{y \in \bfA \setminus \{z\}}
		\tau_{\partial \rmB(y;r_n)} 
		\leq \tau_\partial 
		\Big)
		\leq \frac{2|\bfA| r_n}{n} \,,%\ToO{Lemma 3.3 in the Tightness paper uses that $x$ must be closer to the corresponding $y$ than to the boundary of the domain, which may not be true here. However, I think the bound still holds by the strong Markov property. Perhaps we should say something?}
	\end{equation}
\end{lem}

\begin{lem}[Lemma 3.3 in~\cite{Tightness}]
	\label{lem:nhprob}
	There exists a constant $C=C(\mathrm{D})>0$ such that, for all $x \in \rmD_n$, $t>0$,
	\begin{equation} \label{eq:form2}
		\bbP(L_t(x)=0) = \mathrm{e}^{-\frac{t}{G_{\rmD_n}(x,x)}} \leq \rme^{-\frac{t}{n+C}}.
	\end{equation} 
\end{lem}

We are now ready for:

\begin{proof}[Proof of Theorem~\ref{t:2b}]
	Beginning with the first statement, henceforth we denote by $\bfA$ an $r_n$-cover of $\rmA$. By monotonicity of $\rmA \mapsto \wc{\bfT}_\rmA$ with respect to set inclusion, we may and will assume without loss of generality that 
	$|\bfA| < n^{1/2+\eta_0/8}$. Henceforth, we shall view the random walk as a Poisson Process of excursions at rate $(2\pi)^{-1}\deg(\partial)$. In each of these excursions, the law of the walk is that of a random walk starting from $\partial$ and killed upon returning to~$\partial$. 
	By definition $\wc{T}_A$ is the time of the first excursion after which all vertices in $A$ have been visited. Clearly $\wc{X}_A$ must be a vertex visited in this excursion.
	
	Next, for each $x \in \bfA$, we define the stopping time $\wh{T}^\circ_{\rmB(x;r_n)}$ as the $\partial$-time of the first excursion when $\rmB(x;r_n)$ was hit before any other $\rmB(y;r_n)$ for $y \in \bfA \setminus \{x\}$. We then set,
	\begin{equation}
		\wc{T}^\circ_\bfA := \max_{x \in \bfA}\, \wh{T}^\circ_{\rmB(x;r_n)} \,
		\quad ; \qquad 
		\wc{X}^\circ_\bfA := \argmax_{x \in \bfA}\, \wh{T}^\circ_{\rmB(x;r_n)} \,.
	\end{equation}
	This is the time of the first excursion after which all balls $\rmB(x;r_n)$ for $x \in \bfA$ have been visited, and the center of the ball hit in the last excursion -- if for each excursion we only consider the first ball hit (if any) during this excursion and disregard the rest.
	
	Now, for each $x \in \bfA$, let us say that $x$ was \textit{properly covered} in a given excursion if, during this excursion, $\rmB(x;r_n)$ was the only ball hit among all balls $\rmB(x';r_n)$ for $x' \in \bfA$ and, in addition, $\rmB(x;r_n)$ was completely covered. In contrast, we will say that $x$ was \textit{poorly covered} in a given excursion if, during this excursion, $\rmB(x;r_n)$ was hit first among all balls $\rmB(x';r_n)$ for $x' \in \bfA$ but then either $\rmB(x;r_n)$ was not completely covered or another ball $\rmB(y;r_n)$ with $y \in \bfA \setminus \{x\}$ was hit. Let 
	$\cB_n(x)$ then be the event that during $\partial$-time $t_n^B+sn$, $x$ is not properly covered  in any of the excursions until that $\partial$-time but is, nonetheless, poorly covered in at least one of the excursions. Setting 
	\begin{equation}
		\cB_n := \bigcup_{x \in \bfA} \cB_n(x)\,,
	\end{equation}
	we wish to show that $\bbP(\cB_n)$ tends to $0$ as $n \to \infty$.
	
	To this end, we shall use the following elementary fact: if one carries out ${\rm Poisson}(\tau)$-many  identical independent two-stage experiments, in which the probability of a first-stage success is $p \in (0,1)$ and, upon a successful first stage, the probability of a second-stage failure is $q \in (0,1)$, then the probability that there are no experiments which succeeded in both stages but there is at least one experiment which succeeded in the first stage but then failed in the second is exactly
	\begin{equation}\label{e:8.5n}
		(1-\rme^{-\tau p q})\rme^{-\tau p (1-q)}=(1+o(1))\tau p q\rme^{-\tau p}	
	\end{equation} as $\tau p q \to 0$.
	Putting $\cB_n(x)$ in this context, each experiment is one excursion of the walk. Success in the first stage is hitting $\rmB(x;r_n)$ first among all balls $\rmB(x';r_n)$ for $x' \in \bfA$. Failure in the second stage is either not covering $\rmB(x;r_n)$ in that excursion or hitting another ball $\rmB(y;r_n)$ with $y \in \bfA \setminus \{x\}$ before returning to $\partial$.
	
	Thanks to Lemma~\ref{l:8.2}, Lemma~\ref{l:8.3} and Lemma~\ref{l:8.3n}, in the notation above we have
	\begin{equation}
		p = \frac{2\pi}{\deg(\partial) n}\Big(1+o(n^{-\eta_0/4})\Big) 
		\  ; \quad
		q = n^{-\eta_0/3} + \frac{2|\bfA|r_n}{n} \leq n^{-\eta_0/4} \,.
		\ ; \quad 
		\tau = \frac{\deg(\partial)}{2\pi}(t_n^B + sn) \,.
	\end{equation}
	Plugging this in~\eqref{e:8.5n} we obtain, for all $n$ large enough, $\bbP(\cB_n(x)) \leq (\log n +2s) n^{-1/2-\eta_0/4}$ and, by the union bound and our assumption on the size of $\bfA$, also that $\bbP(\cB_n) \leq (\log n+2s) n^{-\eta_0/8}$, which tends to zero as $n \to \infty$ as desired.
	
	At the same time, on the complement of $\cB_n$ we have
	\begin{equation}
		\wc{T}_\rmA \leq t_n^B + sn \Longleftrightarrow \wc{T}^\circ_\rmA \leq t_n^B + sn
	\end{equation}
	and
	\begin{equation}
		1_{\{\wc{T}_\rmA \leq t_n^B + sn\}}
		\rme^{-n} \big\|\wc{X}_\rmA - \wc{X}^\circ_\rmA\big\| \underset{n \to \infty}{\overset{\bbP}{\longrightarrow}} 0 \,.
	\end{equation}
	It follows that it is sufficient to prove~\eqref{e:2.15m}
	with $\wc{T}_\rmA$, $\wc{X}_\rmA$ replaced by
	$\wc{T}^\circ_\rmA$, $\wc{X}^\circ_\rmA$.
	
	To this end, we notice that, by standard Poisson thinning and Lemma~\ref{l:8.3}, $(\wh{T}^\circ_{\rmB(x; r_n)} :\: x \in \bfA)$ are independent Exponentials with rates
	\begin{equation}
		\frac{\deg(\partial)}{2\pi} \times \frac{2\pi+o(n^{-\eta/4})}{\deg(\partial)n} = \frac{1+o(n^{-\eta/4})}{n} \,.
	\end{equation}
	Another elementary computation then gives,
	\begin{equation}
		\bbP \Big( \wc{T}^\circ_{\rmA} \leq t_n^B+sn \,,\,\, \wc{X}^\circ_\rmA = x
		\Big) = \exp \Big(-\rme^{-s} \frac{1}{\sqrt{n}} \big\|\bfXi_{\rmA, r_n}\big\|\Big) \frac{1+o(n^{-\eta/5})}{\big\|\bfXi_{\rmA, r_n}\big\|}
	\end{equation}
	Summing over all $x$ such that $\rme^{-n} x \in \rmW$ gives the desired statement.
	
	The second statement follows easily from Lemma~\ref{lem:nhprob} using the union bound. The third is a consequence of Proposition~\ref{p:1103.6} and the convergence of the second component in Proposition~\ref{p:300.4}.
\end{proof}

\section{Joint convergence of extremes and average: Proofs of Theorems~\ref{t:1.2i},~\ref{t:1.4i}}
\label{s:5n}
In this section, we prove Theorems~\ref{t:1.2i} and~\ref{t:1.4i} and also obtain Proposition~\ref{p:300.4} as a direct consequence.

\subsection{DGFF extreme value preliminaries}
We recall that the min-extremes of $h_n$ are recorded in the multi-scale (min-)extremal process $\eta_{n,r}$, as defined in~\eqref{eq:defempgff}, and that this process admits a limit as in~\eqref{e:1.10i}. Aside from such precise asymptotics, we shall also need two additional structural results, both derived in~\cite{DingZeitouni}. The first is a quantification of the clustering of the extreme values of $h_n$. 
\TodoF{Add ref to Ding and Zeitouni's paper.}
\begin{lem}[Proposition~3.1 in~\cite{biskuplouidor20}] For any $u \geq 0$,
\label{t:103.14}
\begin{equation}
	\lim_{r \to \infty}\, \limsup_{n \to \infty}\,
		\bbP \Big( \exists x,y \in \rmD_n : \min\{ h_n(x),h_n(y)\} \leq -m_n+u\,,\, r<|x-y| < \rme^n/r \Big) = 0 \,.
\end{equation}
\end{lem}

The second relates the density of extreme values in a given set to its Lebesgue measure.
\begin{lem}\label{lem:coverdgff} There exist constants $c_\rmD,\wh{c}>0$ such that, for all $u \geq 0$, $n \geq 1$ and $\rmV \subseteq \rmD_n$, 
\begin{equation}
	\bbP\Big( \min_{x \in \rmV} h_n(x) \leq -m_n+u\Big) \leq c_\rmD \rme^{\wh{c} u}\sqrt{\frac{|\rmV|}{|\rmD_n|}}.
\end{equation}	
\end{lem} 
\begin{proof} Since by construction there exists a constant $\overline{c}:=\overline{c}(\rmD) > 0$ such that $\rmD_n \subseteq (-\overline{c} \rme^n,\ol{c}\rme^n)^2$, the result now follows at once from \cite[Lemma~B.12]{BL3} by choosing $N:=2\ol{c}\rme^{n}$, $t=1$ and $s=u$.
\end{proof}

\subsection{The Zero Average DGFF: Proof of Theorem~\ref{t:1.4i}}
Recall that the Dirichlet inner product on $\bbR^{\rmD_n}$ is given by
\begin{equation}
	\la f,\, g \ra_{\Delta} := \la f,\, \Delta g \ra \,,
\end{equation}
for $f,g \in \bbR^{\rmD_n}$ and $\Delta \equiv \Delta_{\rmD_n}$ is the (negative) discrete Laplacian on $\rmD_n$ with zero boundary conditions on $\rmD_n^\rmc$. 
The corresponding norm will be denoted by $\|\cdot\|_\Delta$. We shall take the normalization of $\Delta_{\rmD_n}$ to be such that the Green Function $G_n$ associated with $h_n$  is its inverse.

With $\bfone$ being the constant $1$ function on $\rmD_n$, define
\begin{equation}
	\ol{\psi}_n := \frac{1}{|\rmD_n|}\, G_n \bfone 
	\quad \Longleftrightarrow \quad \ol{\psi}_n(\cdot) = \frac{1}{|\rmD_n|} \sum_{x \in \rmD_n} G_n(x,\cdot) \,,
\end{equation}
and let $h^\psi_n$ and $\wh{h}_n$ denote the projection of $h$ under the Dirichlet inner product onto the space spanned by $\ol{\psi}_n$ and its complement (as subspaces of $\bbR^{\rmD_n}$) respectively. It follows that
\begin{equation}
\label{e:1104.1}
h_n = h^\psi_n \oplus \wh{h}_n\,,
\end{equation} 
where $\oplus$ denotes a sum of two independent quantities and that
\begin{equation}
	h^\psi_n = \frac{\ol{\psi}_n}{\| \ol{\psi}_n \|_\Delta^{2}} \langle h,\, \ol{\psi}_n \rangle_\Delta = 
\psi_n	\ol{h}_n
	 \quad ; \qquad 
\wh{h}_n \, \overset{\rmd}=\, h_n \, \big|\,\{ \la h ,\, \ol{\psi}_n \ra_\Delta = 0\} 
\,\overset{\rmd}=\, h_n \, \big|\, \{\ol{h}_n = 0\} \,,
\end{equation}
where 
\begin{equation}
\psi_n	 := \frac{\ol{\psi}_n}{\|\ol{\psi}_n\|_\Delta^2}
= \frac{\frac{1}{|\rmD_n|} \sum_{x \in \rmD_n} G_n(x, \cdot)}
	 	{\frac{1}{|\rmD_n|^2} \sum_{x,y \in \rmD_n} G_n(x,y)} \,,
\end{equation}
and $\ol{h}_n$ is the field average, as in~\eqref{e:1.17i}. Thanks to the independence in~\eqref{e:1104.1}, to obtain asymptotics for the joint law of $\ol{h}_n$ and $\eta_{n,r}$ it will be sufficient to derive the asymptotic laws of $h^\psi_n$ and $\wh{\eta}_{n,r}$, the extremal process of $\wh{h}_n$, separately. Asymptotics for $h^\psi_n$ are easy to get, as this field is essentially ``one dimensional'' with an exactly solvable variance. This is the content of Lemma~\ref{l:1104.1} below. On the other hand, deriving a limit in law for $\wh{\eta}_{n,r}$, namely proving Theorem~\ref{t:1.4i}, requires non-trivial work, and the remainder of this subsection is devoted to this task.

\begin{lem}
\label{l:1104.1}
The following holds:
\begin{enumerate}
	\item 
Uniformly on $\rmD$,
\begin{equation}
	\psi_n \big(\lfloor \rme^{n} \cdot \rfloor \big) 
	\underset{n \to \infty}\longrightarrow\
		\psi_\rmD \,,
\end{equation}
where $\psi_\rmD$ is as in~\eqref{e:1.9p}. As $\psi_\rmD$ is a continuous function on $\ol{\rmD}$, the functions $\{\psi_n\}$ are uniformly bounded in the supremum norm and asymptotically equicontinuous, in the sense that 
\begin{equation}
    \lim_{\delta \downarrow 0} \limsup_{n\to \infty} \sup_{|x-y| < \delta \rme^n} \big|\psi_n(x) - \psi_n(y) \big| = 0 \,.
\end{equation}
\item With $\ol{h}_\rmD$ as in~\eqref{e:1.7p},
\begin{equation}
	\ol{h}_n \underset{n \to \infty} \Longrightarrow \ol{h}_\rmD \,.
\end{equation}
\item 
With respect to the supremum norm on $\ol{\rmD}$, 
\begin{equation}
\label{e:1104.13}
h^\psi_n(\lfloor \rme^n \cdot \rfloor) \underset{n \to \infty}\Longrightarrow h^\psi_\rmD
\quad ; \qquad 
h^\psi_\rmD := \ol{h}_\rmD \psi_\rmD \,.
\end{equation}
The field $h_\rmD^\psi$ is Gaussian with mean zero and covariance given by
\begin{equation}
	\bbE h_\rmD^\psi(x)h_\rmD^\psi(y) = \frac{\int_{\rmD \times \rmD} G_\rmD(z,x) G_\rmD(z', y) \rmd z \rmd z'}{\int_{\rmD \times \rmD} G_\rmD(z,z') \rmd z \rmd z'} \,.
\end{equation}
\end{enumerate}
\end{lem}

To prove Lemma~\ref{l:1104.1}, we shall need the following auxiliary result.

\begin{lem}\label{lem:hitting} There exist $C,\delta > 0$ such that, given $x \in \rmD_n$ and $R > \rmd(x,\rmD_n^c)$, 
\begin{equation}\label{eq:bq1}
	\bbP_x(\tau_{\partial \rmB(x;\log R)} < \tau_\partial) \leq C \left(\frac{\rmd(x,\rmD_n^c)}{R}\right)^\delta.
\end{equation} In particular, given any $\varepsilon > 0$ and $x,y \in \rmD$ with $\| x-y\| \geq \varepsilon$, $\rmd(y,\rmD^c)>2\varepsilon$ and $\lfloor \rme^n x\rfloor, \lfloor \rme^n y\rfloor \in \rmD_n$, 
\begin{equation}\label{eq:bq2}
\bbP_{\lfloor \rme^n x\rfloor}(\tau_{\lfloor \rme^n y\rfloor} < \tau_\partial) \leq 	C_{\varepsilon,\rmD} \frac{\left(\rmd(x,\rmD^c)\right)^\delta }{n}
\end{equation} for some constant $C_{\varepsilon,\rmD} > 0$ and all $n$ large enough (depending only on $\varepsilon$ and $\rmD$).
\end{lem}

\begin{proof} For  $r > 0$ and $m \geq 0$, consider the annuli $A_{m,r}(x):=\{ y : \rme^m r < \|x-y\| \leq \rme^{m+1}r\}$ together with $\gamma_{m,r}(x)$, the center curve in $A_{m,r}(x)$, given by $\gamma_{m,r}(x):=\{ y: \|x-y\| = \frac{(1+\rme)}{2}\rme^mr\}$. Let~$G_{m,r}(x)$ be the event that Brownian motion on $\bbR^2$ started from $\gamma_{m,r}(x)$ performs a full loop around $A_{m,r}(x)$ before exiting $A_{m,r}(x)$. It is clear that $\bbP(G_{0,r}(x))$ is positive for any $r > 0$ and, by recurrence and scale invariance of Brownian motion, that it is so uniformly in $r > 0$, so that $\inf_{m \in \bbN_0\,,r>0} \bbP(G_{m,r}(x)) > 0$. Hence, by Donsker's invariance principle, 
\begin{equation}
\liminf_{n \to \infty} \left(\inf_{m \in \bbN_0\,,\,r>0}\bbP(H_{n,m,r}(x))\right) > 0
\end{equation} where $H_{n,m,r}(x)$ is
the event that a simple symmetric random walk on $\bbZ^2$ started from $\gamma_{m,r\rme^n}(x)$ performs a full loop around $A_{m,r\rme^n}(x)$ before exiting $A_{m,r\rme^n}(x)$. 

Now, if we choose $r:=\rmd(x,\rmD_n^c)$, then, for any $1 \leq m \leq \log(\frac{R}{r}) -1$,  on the event $H_{n,m,r}(x)$ the random walk exists $\rmD_n$ before reaching $\partial \rmB(x;\log R)$. Therefore, by successive applications of the strong Markov property at the hitting times of the center curves $\gamma_{m,r\rme^n}(x)$, we conclude that, for some $q \in (0,1)$, 
\begin{equation}
\bbP_x(\tau_{\partial \rmB(x;\log R)} < \tau_\partial) \leq q^{\log(\frac{R}{r})-2} \leq  C_q \left(\frac{r}{R}\right)^\delta
\end{equation} for $C=\frac{1}{q^2}$  and $\delta:=-\log q>0$, which is precisely \eqref{eq:bq1}.

To show \eqref{eq:bq2}, we notice that, if we take $R:=(\frac{1}{3}\| \lfloor \rme^n x \rfloor - \lfloor \rme^n y \rfloor\|)\wedge \rmd(\lfloor \rme^n y\lfloor,\rmD_n^c)$ in \eqref{eq:bq1} then, for all $n$ large enough we have $R \geq \frac{\varepsilon}{2}\rme^n$, so that, since $\rmB(\lfloor \rme^n x\rfloor;\log R)$ and $\rmB(\lfloor \rme^n y\rfloor,\log R)$ are disjoint by choice of $R$, by the strong Markov property at~$\tau_{\rmB(\lfloor \rme^n y\rfloor;\log R)}$ we obtain that
\begin{equation}
\bbP_{\lfloor \rme^n x\rfloor}(\tau_{\lfloor \rme^n y\rfloor} < \tau_\partial) \leq 	C \left( \frac{\rmd(\lfloor \rme^n x \rfloor,\rmD_n^c)}{\frac{\varepsilon}{2}\rme^n}\right)^\delta \sup_{z \in \partial_i \rmB(\lfloor \rme^n y\rfloor;\log R)} \bbP_z( \tau_{\lfloor \rme^n y\rfloor} < \tau_\partial) \leq C_{\varepsilon} (\rmd(x,\rmD^c))^{\delta}  %C_\varepsilon \frac{\left(\rmd(x,\rmD^c)\right)^\delta }{n + c_{\varepsilon,\rmD}}	
\end{equation}
where we recall that, given a set $A \subseteq \bbZ^2$ we write $\partial_i A := \partial A^c$ and, to obtain the last inequality, we have used that, for $c_\rmD > 0$ large enough so that $\rmD_n \subseteq \rmB(w;n+c_\rmD)$ for all $w \in \rmD_n$,  
\begin{equation}
\sup_{z \in \partial_i \rmB(\lfloor \rme^n y\rfloor;\log R)} \bbP_z( \tau_{\lfloor \rme^n y\rfloor} < \tau_\partial) \leq \sup_{z \in \partial_i \rmB(\lfloor \rme^n y\rfloor;\log R) } \bbP_{z}(\tau_{\lfloor \rme^n y\rfloor} < \tau_{\partial \rmB(y;n+c_\rmD)})	
\end{equation} where the probability on the right-hand side is with respect to a simple symmetric random walk on all of $\bbZ^2$ (not just on $\wh{\rmD}_n$), which by Lemma~\ref{lem:ll1} can be bounded from above by 
\begin{equation}
\frac{n+c_\rmD - n - \log(\tfrac{\varepsilon}{2})+\tfrac{1}{\varepsilon}O(\rme^{-n})}{n+c_\rmD}(1+O(\tfrac{1}{n})) \leq \frac{C_{\varepsilon,\rmD}}{n} 
\end{equation} for all $n$ large enough (depending only on $\varepsilon$ and $\rmD$). Now, upon combining the last three displays, \eqref{eq:bq2} immediately follows. 
\end{proof}

We are now ready to prove Lemma~\ref{l:1104.1}.

\begin{proof}[Proof of Lemma~\ref{l:1104.1}]
For Part 1, notice that it will suffice to show that, as $n \to \infty$,
\begin{equation}
\ol{\psi}_n(\lfloor \rme^n x\rfloor) \longrightarrow \ol{\psi}(x):=	 \frac{1}{\mathrm{Leb}(\rmD)} \int_{} G_\rmD (x,y)\,\rmd y
\end{equation} uniformly over $\rmD$, since all assertions from Part 1 follow at once from this by standard arguments. To this end, note
 that, for any $x \in \rmD$,  
\begin{equation}
\ol{\psi}_n(\lfloor \rme^n x\rfloor) =\frac{(1+o_n(1))}{\textrm{Leb}(\rmD)}\int_\rmD G_n(\lfloor \rme^n x\rfloor, \lfloor \rme^n y\rfloor)\,\rmd y
\end{equation} with $o_n(1) \to 0$ as $n \to \infty$ uniformly over $x$. In addition, by standard Green's function estimates (see \cite[Lemma~3.1]{Tightness}), we have
\begin{equation}
G_n(\lfloor \rme^n x\rfloor, \lfloor \rme^n y\rfloor) \leq  n + C
\end{equation} uniformly over $x, y \in \rmD$ and
\begin{equation}\label{eq:boundGreen1}
G_n(\lfloor \rme^n x\rfloor, \lfloor \rme^n y\rfloor) \leq - \log\|x-y\| + C'
\end{equation} uniformly over all $x,y \in \rmD$ with $\|x-y\| > \frac{2}{\rme^n}$. It follows from this that, given $\varepsilon > 0$,  
\begin{equation}\label{eq:asympphin}
\ol{\psi}_n(\lfloor \rme^n x\rfloor) =\frac{1}{\textrm{Leb}(\rmD)}\int_{\|x-y\| > \varepsilon \atop \rmd(y,\rmD^\rmc) > \varepsilon} G_n(\lfloor \rme^n x\rfloor, \lfloor \rme^n y\rfloor)\,\rmd y + o_{n,\varepsilon}(1),
\end{equation} where $o_{n,\varepsilon}(1) \to \infty$ in the limit as $n \to \infty$ followed by $\varepsilon \to 0$ uniformly in $x \in \rmD$. Also, since~\eqref{eq:boundGreen1} implies that 
$
G_\rmD(x,y) \leq -\log \| x-y \| + C'
$ for all $x,y \in \rmD$, by a similar argument we have
\begin{equation}\label{eq:asympginfty}
\overline{\psi}_\rmD(x) := \frac{1}{\mathrm{Leb}(\rmD)} \int_{} G_\rmD (x,y)\,\rmd y=\frac{1}{\mathrm{Leb}(\rmD)} \int_{\|x-y\| > \varepsilon \atop \rmd(y,\rmD^c) > \varepsilon} G_\rmD (x,y)\,\rmd y + o_{\varepsilon}(1).
\end{equation} where $o_{\varepsilon}(1) \to 0$ as $\varepsilon \to 0$ uniformly in $x \in \rmD$. Now, since $G_n$ converges to $G_\rmD$ uniformly over compact~subsets of $\{ (x,y) \in \rmD^2 : x \neq y\}$ (see e.g. \cite[Theorem~1.17]{biskuppims}), we have, as $n \to \infty$,
\begin{equation}
\sup_{x:\rmd(x,\rmD^c)> \varepsilon} \left[\int_{\|x-y\| > \varepsilon \atop \rmd(y,\rmD^\rmc) > \varepsilon} |G_n(\lfloor \rme^n x\rfloor, \lfloor \rme^n y\rfloor)-G_\rmD(x,y)|\,\rmd y\right]\longrightarrow 0,
\end{equation} which, in combination with \eqref{eq:asympphin} and \eqref{eq:asympginfty}, yields that  
\begin{equation}\label{eq:convbulk}
\lim_{\varepsilon \to 0} \limsup_{n \to \infty} \left[\sup_{x:\rmd(x,\rmD^\rmc) > \varepsilon} |\overline{\psi}_n(x)-\overline{\psi}_\rmD(x)|\right]=0.
\end{equation} 
Finally, to handle the uniform convergence near the boundary of $\rmD$, we notice that it will suffice to show that, for any $x,y \in \rmD$ with $\rmd(x,\rmD^c) \leq \varepsilon$ and $\rmd(y,\rmD^c)>2\varepsilon$, one has, for some~$C_{\varepsilon,\rmD} > 0$, that
\begin{equation}\label{eq:boundHbound}
	G_n(\lfloor \rme^n x\rfloor ,\lfloor \rme^ny\rfloor) \leq C_{\varepsilon,\rmD}  (\rmd(x,\rmD^c))^\delta+\wt{o}_{n,\varepsilon,\rmD}(1)
\end{equation} 
where $\wt{o}_{n,\varepsilon,\rmD} (1) \to 0$ as $n \to \infty$ uniformly over all $x,y \in \rmD$ as above for any fixed $\varepsilon > 0$. Indeed, \eqref{eq:boundGreen1}--\eqref{eq:asympphin} combined with \eqref{eq:boundHbound} then imply that, for any $x$ with $\rmd(x,\rmD^c) \leq \varepsilon':= \big(\frac{\varepsilon}{C_{\varepsilon,\rmD} }\big)^{\frac{1}{\delta}}$, 
\begin{equation}
\ol{\psi}_n(\lfloor \rme^n x\rfloor) \leq C_{\varepsilon,\rmD}  (\rmd(x,\rmD^c))^\delta+g_{n,\varepsilon,\rmD}(x) \leq 2\varepsilon + g_{n,\varepsilon,\rmD}(x)
\end{equation} where $\sup_{x:\rmd(x,\rmD^c) \leq \varepsilon} g_{n,\varepsilon,\rmD}(x) \to 0$ as $n \to \infty$ followed by $\varepsilon \to 0$, so that 
\begin{equation}\label{eq:converror1}
\lim_{\varepsilon \to 0} \limsup_{n \to \infty}\left[\sup_{x: \rmd(x,\rmD^\rmc) \leq \varepsilon'} \overline{\psi}_n(\lfloor \rme^n x\rfloor)\right]=0.
\end{equation} In addition, by taking the limit as $n \to \infty$ in \eqref{eq:boundHbound} we obtain that $G_\rmD(x,y) \leq C_{\varepsilon,\rmD} ((\rmd(x,\rmD^c))^{\delta}$ for~any $x,y \in \rmD$ with $\rmd(x,\rmD^c) \leq \varepsilon$ and $\rmd(y,\rmD^c) > 2\varepsilon$, which combined with \eqref{eq:asympginfty} and the bound $G_\rmD(x,y) \leq - \log \|x-y\| + C'$ yields that 
\begin{equation}\label{eq:converror2}
\lim_{\varepsilon \to 0}\left[\sup_{x: \rmd(x,\rmD^c) \leq \varepsilon'} \overline{\psi}_\rmD(x)\right]=0,
\end{equation} from where the uniform convergence of $\ol{\psi}_n$ to $\ol{\psi}_\rmD$ on $\rmD$ will then readily follow in light of~\eqref{eq:convbulk} (used for $\varepsilon'$ in place of $\varepsilon$) and \eqref{eq:converror1}. Therefore, it only remains to show \eqref{eq:boundHbound}. 

To this end, fix $x,y \in \rmD$ with $\rmd(x,\rmD^c) \leq \varepsilon$, $\rmd(y,\rmD^c) > 2\varepsilon$ and write $x_n:=\lfloor \rme^n x\rfloor$, $y_n:=\lfloor \rme^n y\rfloor$ for simplicity. We may assume that $x_n,y_n \in \rmD_n$, since otherwise the bound is immediate. Then, by the strong Markov property of the walk, we have that
\begin{equation}\label{eq:greenprod}
G_n(x_n, y_n) = \bbP_{x_n}(\tau_{y_n} < \tau_\partial)G_n(y_n,y_n).
\end{equation}
By standard Green Function estimates (see \cite[Lemma~3.1]{Tightness}), we have $G_n(y_n,y_n) \leq n + C_\rmD$~for some $C_\rmD > 0$ uniformly in $y \in \rmD$. On the other hand, since $\rmd(y;\rmD^c) -\rmd(x,\rmD^c) > \varepsilon$, by Lemma~\ref{lem:hitting} we have that for all $n$ large enough,
\begin{equation}
\bbP_{x_n}(\tau_{y_n} < \tau_\partial) \leq C_{\varepsilon,\rmD} \frac{(\rmd(x,\rmD^c))^\delta}{n}.	
\end{equation}
Combining both estimates, \eqref{eq:boundHbound} (and hence \eqref{eq:boundGreen1}) immediately follows in light of \eqref{eq:greenprod}.

Turning to Part 2, the desired statement follows immediately from \cite[Theorem~1.27]{biskuppims} with the particular choice $f \equiv \frac{1}{\mathrm{Leb(\rmD)}}$ therein, upon noticing that, by our conditions on $\rmD$, we have $|\rmD_n|=\rme^{2n}\mathrm{Leb}(\rmD)(1+o(1))$ for some error term $o(1) \to 0$ as $n \to \infty$.   
Part 3 follows immediately from Parts 1 and 2.
\end{proof}

Our standing goal now is to prove Theorem~\ref{t:1.4i}, namely to derive asymptotics for $\wh{\eta}_{n,r}$. We recall that the latter was defined in~\eqref{e:1.14j}. The proof will follow from the following several lemmas, the first of which, states the required tightness. 
\begin{lem}
\label{l:5.8n} 
Let $\{r_n\}$ be as in~\eqref{e:1.11i}. Then, the sequence $\{\wh{\eta}_{n,r_n}\}$ is tight.
\end{lem}
\begin{proof}
Immediate from the tightness of $\{\eta_{n,r_n}\}$ and the tightness of $\{h^\psi_n\}$ in the supremum norm as implied by Lemma~\ref{l:1104.1}.
\end{proof}

Next, towards identifying the possible subsequential limits of $\wh{\eta}_{n,r_n}$, we have the following relation in law which each such limit must obey.
\begin{lem}
\label{l:1104.3}
Let $\wh{\eta}_\rmD$ be limit of $\{\wh{\eta}_{n,r_n}\}$ along some subsequence, with $\{r_n\}$ as in~\eqref{e:1.11i}. Then, there exists a coupling between $\wh{\eta}_\rmD$, the random measure $\eta_\rmD$ from~\eqref{e:1.10i}, and $h_\rmD^\psi$ from~\eqref{e:1104.13}, such that $\wh{\eta}_\rmD$ and $h_\rmD^\psi$ are independent and 
\begin{equation}
\label{e:1104.14}
\wh{\eta}_\rmD \circ \tau^{-1}_{h^{\psi}_\rmD} = \eta_\rmD \,,
\end{equation}
where $\tau_{\psi}: \rmD \times \bbR \times \bbR^{\bbZ^2} \to \rmD \times \bbR \times \bbR^{\bbZ^2}$ is the translation operator
\begin{equation}
	\tau_{\psi}(x,s,\omega) =(x,s + \psi(x), \omega) \,.
\end{equation}
Moreover, along the above subsequence we have
\begin{equation}
	 \big(\eta_{n,r_n}, \wh{\eta}_{n,r_n}, h_n^\psi\big) \Longrightarrow \big(\eta_\rmD, \wh{\eta}_\rmD, h^\psi_\rmD\big) \,.
\end{equation}
\end{lem}
\begin{proof}
Employing the decomposition in~\eqref{e:1104.1}, since $h_n^\psi$ tends weakly to $h_\rmD^\psi$ as $n \to \infty$ w.r.t. the $\rmL^\infty$ norm on $\ol{\rmD}$, and is  independent of $\wh{\eta}_{n,r_n}$, it follows by standard arguments that $\wh{\eta}_{n,r_n} \circ \tau^{-1}_{h_n^\psi}$ tends vaguely-weakly to $\wh{\eta}_\rmD \circ \tau^{-1}_{h_\rmD^\psi}$ along the same subsequence w.r.t. which $\wh{\eta}_\rmD$ is defined. Here, we have also used the a.s. continuity of $h_\rmD^\psi$ on $\ol{\rmD}$. In view of~\eqref{e:1.11i}, we thus need to show that the vague distance between $\wh{\eta}_{n,r_n} \circ \tau^{-1}_{h_n^\psi}$ and 
$\eta_{n,r_n}$ tends to $0$ in probability with $n$. Equivalently, for any $K < \infty$ and $\varphi: \ol{\rmD} \times \bbR \times \bbR^{\bbZ^2}$ - a continuous function, which depends on $\ol{\rmD} \times \bbR \times \bbR^{\rmQ(K)}$ and supported on $\ol{\rmD} \times [-K, K] \times [-K, K]^{\rmQ(K)}$, we need to show that
\begin{equation}
\label{e:1104.16}
	\big| \la \wh{\eta}_{n,r_n} \circ \tau^{-1}_{h^\psi_n},\, \varphi \ra -
	\la \eta_{n, r_n} ,\, \varphi \ra \big|
\end{equation}
tends to $0$ in probability with $n$.

Recall that we write $f_n$ for $h_n + m_n$ and set also $\wh{f}_n := \wh{h}_n + m_n$. By~\eqref{e:1.10i}, Lemma~\ref{t:103.14} and the fact the marginals of the cluster law $\nu$ on $\bbR^{\bbZ^2 \setminus 0}$ are absolutely continuous w.r.t. the Lebesgue measure, by choosing $M,r < \infty$ large enough, $\delta > 0$ small enough and finally $n$ large enough, we can ensure that with probability arbitrarily close to $1$, 
there are at most $M$ vertices $x \in \rmD_n$ such that $f_n(x) \leq K$, and that if such an $x$ is also an $r_n$-local minimum of $f_n$, then $f_n(z) > K+1$ for all $z \in \rmQ(x;r_n) \setminus \rmQ(x;r)$ and $f_n(z)-f_n(x) > \delta$ for all $z \in \rmQ(x;r)$. Thanks to Lemma~\ref{l:1104.1}, for all $\delta' > 0$, we can also ensure that  with probability arbitrarily close to $1$ we have $|\psi_n(z)-\psi_n(x)| \leq \delta'$ for all $z \in \rmQ(x;r_n)$ and all $x \in \rmD_n$, by taking, again, $n$ large enough.

When the last two events occur and $\delta' < \delta$, if $x$ is such that $f_n(x) \leq K$ then it is an $r_n$-local minimum of $f_n$ if and only if it is an $r_n$-local minimum of $\wh{h}_n$, so that the non-zero terms in the sums realizing both integrals in~\eqref{e:1104.16} involve exactly the same vertices. Moreover, for all such $r_n$-local minimum $x$, we we have $\wh{f}_n(x) + h^\psi_n(x) = f_n(x)$ and if $y \in \rmQ(K)$, also
\begin{equation}
	\big|\big(\wh{f}_n(x+y) - \wh{f}_n(x)\big) - \big(f_n(x+y) - f_n(x)\big) \big| \leq 2\delta' \,.
\end{equation}
Taking $\delta' > 0$ small enough (depending on $\varphi$ and $M$) we can thus ensure that the sum over all such $r_n$-local minima $x$ of 
\begin{equation}
\Big|\big(\varphi \circ \tau_{h_n^\psi}\big) \Big(x/N,\, \wh{f}_n(x),\, \big(\wh{f}_n(x+y) - \wh{f}_n(x)\big)_{y \in \rmQ(K)}\Big) - 
\varphi \Big(x/N,\, f_n(x),\, \big(f_n(x+y) - f_n(x)\big)_{y\in \rmQ(K)}\Big) \Big|
\end{equation}
is arbitrarily small. But this bounds the difference in~\eqref{e:1104.16}.
\end{proof}

Uniqueness of the solution to~\eqref{e:1104.14} and, as such, of all subsequential limits of $\wh{\eta}_{n,r_n}$, will follow from the following general lemma, which is essentially contained in~\cite{Marekexceptional}.
\begin{lem}
\label{l:104.4}
Let $W$ be a standard Gaussian, $Y, Y'$ be random variables taking values in some measurable space $\cY$ and independent of $W$, and 
$F: \cY \times \bbR \to \bbR$ a bounded measurable function. Then 
\begin{equation}
 \bbE F(Y,W+z) = \bbE F(Y',W+z) \,,\,\, \forall z \in \bbR \quad \Longrightarrow \quad \bbE F(Y,0) = \bbE F(Y',0) \,.
\end{equation}
\end{lem}
\begin{proof}
Let $\phi_Y(t,z) := \bbE F(Y,\sqrt{t}W+z)$. Using $u(t,x)$ for the Gaussian density with mean $0$ and variance $t$, we have
\begin{equation}
\phi_{Y}(t,z) = \bbE \int F(Y,x) u(t,z-x) \rmd x  
\end{equation}
Since $u$ satisfies the Heat Equation $u_t = \frac12 u_{xx}$ on $\bbR_+ \times \bbR$, it follows by the Dominated Convergence Theorem that so do $\phi_Y$. 
The backward uniqueness of this equation then shows that $\Phi_Y(0,0)$ is uniquely determined by the value of $\Phi_Y(1,z)$ for all $z \in \bbR$. This directly implies the desired statement.
\end{proof}

\begin{proof}[Proof of Theorem~\ref{t:1.4i}]
Let $(\eta_\rmD,  \wh{\eta}_\rmD, h_\rmD^\psi)$ be a sub-sequential limit of $(\eta_{n,r_n}, \wh{\eta}_{n,r_n}, h^\psi_n)_{n \geq 1}$. It is an elementary fact which follows, e.g., from the Law of Large Numbers, that if $\cN$ is a PPP  with (deterministic) intensity measure $\mu$, the the limit $\cN(A_k)/\mu(A_k) \to 1$ almost-surely whenever $(A_k)_{k \geq 1}$ is an increasing sequence whose measure under $\mu$ tends to $\infty$. Used in our setting, for any measurable set $A \subset \rmD$, we have almost-surely
\begin{equation} 
\label{e:1104.22}
	\cZ_\rmD(A) = \lim_{u \to \infty} \frac{\eta_\rmD \big(A \times [-u,\infty)\big)}{\alpha^{-1} \rme^{\alpha u}}
\end{equation}
where $\cZ_\rmD$ is the LQGM measure in a realization of $\eta_\rmD$ as in~\eqref{e:1.10i}. Taking the collection of all rectangles in $\rmD$ with rational coordinates and using that the latter is a $\pi$-system which generated the Borel set on $\rmD$, we may almost-surely recover $\cZ_\rmD$ as the unique extension of this set function as a measure on $\rmD$, so that $\cZ_\rmD$ is measurable w.r.t. (the sigma-algebra generated by) $\eta_\rmD$.

Next, define the (random) measure $\wh{\cZ}_\rmD$ via
\begin{equation}
\label{e:1104.25}
	\wh{\cZ}_\rmD(A) := \int_A \rme^{-\alpha h^\psi_\rmD(x)} \cZ_\rmD(\rmd x) 
\end{equation}
We wish to claim for all measurable sets $A$, we almost-surely have
\begin{equation} 
\label{e:1104.23}
\wh{\cZ}_\rmD(A) = \lim_{u \to \infty} \frac{\wh{\eta}_\rmD \big(A \times [-u,\infty)\big)}{\alpha^{-1} \rme^{\alpha u}}
\end{equation}
Indeed, let $\wh{\cZ}_+(A)$ and $\wh{\cZ}_-(A)$ be as in~\eqref{e:1104.23} 
with the limit replaced by limit superior and limit inferior, respectively. Since,
\begin{equation}
\eta_\rmD \big(A \times [-u + \sup_A h^\psi_\rmD ,\infty)\big) 
\leq \wh{\eta}_\rmD \big(A \times [-u,\infty)\big)
\leq \eta_\rmD \big(A \times [-u + \inf_A h^\psi_\rmD ,\infty)\big) \,,
\end{equation}
it follows from~\eqref{e:1104.22} that 
\begin{equation}
\cZ_\rmD(A) \rme^{-\alpha \sup_A h^\psi_\rmD} \leq \wh{\cZ}_-(A) \leq \wh{\cZ}_+(A) \leq \cZ_\rmD(A) \rme^{-\alpha \inf_A h^\psi_\rmD}
\end{equation}

Next, for each $\epsilon > 0$, let $\{A_{\epsilon, k}\}_k$ be a finite partition of $A$ such that the diameter of each set in the partition is at most $\epsilon$. Since $\wh{\eta}_\rmD$ is a measure, we thus have
\begin{equation}
\label{e:1104.27}
\int_A \rme^{-\alpha h^{\psi, \epsilon,+}_\rmD(x)} \cZ_\rmD(\rmd x)
\leq \sum_k \wh{\cZ}_-(A_{\epsilon, k}) \leq \wh{\cZ}_-(A)
\leq \wh{\cZ}_+(A) \leq \sum_k \wh{\cZ}_+(A_{\epsilon, k}) \leq 
\int_A \rme^{-\alpha h^{\psi, \epsilon,-}_\rmD(x)} \cZ_\rmD(\rmd x)
\end{equation}
where
\begin{equation}
h^{\psi, \epsilon,+}_\rmD = \sum_k \big(\sup_{A_{\epsilon, k}} h^\psi_\rmD\big) 1_{A_{\epsilon, k}}
\quad, \qquad
h^{\psi, \epsilon,-}_\rmD = \sum_k \big(\inf_{A_{\epsilon, k}} h^\psi_\rmD\big) 1_{A_{\epsilon, k}}
\end{equation}
The almost-sure continuity of $h_\rmD^\psi$ on $\ol{\rmD}$ implies that $\|h^\psi - h^{\psi, \epsilon,\pm}_\rmD\|_\infty$ tends to $0$ as $\epsilon \to 0$ and thus both integrals in~\eqref{e:1104.27} tend to the right hand side of~\eqref{e:1104.25}. 

As in the case for $\cZ_\rmD$, since $\wh{\cZ}_\rmD$ is determined by its values on the $\pi$-system of rational rectangles in $\rmD$, 
it follows from~\eqref{e:1104.23} that $\wh{\cZ}_\rmD$ is measurable w.r.t. (the sigma-algebra generated by) $\wh{\eta}_\rmD$. Independence between the latter and $h^\psi_\rmD$, as shown by Lemma~\ref{l:1104.3} thus implies the independence between $\wh{\cZ}_\rmD$ and $h^\psi_\rmD$. In particular, 
\begin{equation}
\label{e:5.42n}
		\wh{\cZ}_\rmD \rme^{\alpha h^\psi_\rmD}  = \cZ_\rmD 
		\quad \text{a.s.} \,.
\end{equation}  
This shows existence of a random measure $\wh{\cZ}_\rmD$ satisfying~\eqref{e:5.42n}. Uniqueness in law follows from~\cite[Theorem~3.1]{abe2023exceptional}.

Now if $\wh{\eta}_\rmD$ has law as in~\eqref{E:1.9n} then $\wh{\eta}_\rmD \circ \tau^{-1}_{h^\psi_\rmD}$ is a PPP with random intensity given by
\begin{equation}
\label{e:5.43n}
	\wh{\cZ}_\rmD(\textd x) \otimes \rme^{-\alpha (s-h^\psi_\rmD(x))}\textd s\otimes\nu(\omega)\bigr) 
	= \cZ_\rmD(\rmd x) \otimes \rme^{-\alpha s}\textd s\otimes\nu(\omega) \,,
\end{equation}
so that such $\wh{\eta}_\rmD$ is a solution to~\eqref{e:1104.14}. 
To show that this is the only solution, let $f:\rmD \times \bbR \times \bbR_+^{\bbZ_2} \to \bbR$ be any bounded measurable function and consider the function $F:\cM(\rmD \times \bbR \times \bbR_+^{\bbZ_2}) \times \bbR \to \bbR$, given by
\begin{equation}
	F(\eta, w) := \exp  \Big(- \big\la \eta  ,\, f \circ \tau_{w \ol{\sigma}_\rmD \psi^{1/2}_\rmD} \big\ra\Big) \,.
\end{equation}
If $\eta_\rmD$, $\wh{\eta}_\rmD$ and $h^\psi_\rmD$ satisfy~\eqref{e:1104.14} then 
\begin{equation}
	\bbE F(\wh{\eta}_\rmD, 0) = \bbE \exp \big(-\la \wh{\eta}_\rmD, f \ra \big) 
	\quad ; \quad \bbE F(\wh{\eta}_\rmD, W+z) = \bbE \exp \Big(- \big \la \eta_\rmD, f \circ \tau_{z \ol{\sigma}_\rmD \psi^{1/2}_\rmD} \big \ra \Big) \,,
\end{equation}
whenever $W$ is a standard Gaussian independent of $\wh{\eta}_\rmD$ and $z \in \bbR$. Since the last expectation does not depend on the law of $\wh{\eta}_\rmD$, Lemma~\ref{l:104.4} implies that the same holds for the first. As functional Laplace transform determines the law of a random measure, this shows that all subsequential limits of $\wh{\eta}_{n,r_n}$ have the law of $\wh{\eta}_\rmD$ from~\eqref{E:1.9n}. Thanks to Lemma~\ref{l:5.8n} this is also the limit along the full sequence.
\end{proof}

%\begin{lem}
%	\label{l:A.1}
%	For all $u \in \bbR$, $n > 0$ and $\rmV \subset \rmD_n$, 
%	\begin{equation}
%		\limsup_{n \to \infty}
%		\bbP \Big(\min_{x \in \rmV} h(x) + m_n < u \big) \leq C \bigg(\frac{|\rmV|}{|\rmD_n|}\bigg)^{1/2} \rme^{c u} \,.
%	\end{equation}
%\end{lem}

\subsection{Proof of Theorem~\ref{t:1.2i} and Proposition~\ref{p:300.4}}
It remains to give the proofs of Theorem~\ref{t:1.2i} and then Proposition~\ref{p:300.4}.
\begin{proof}[Proof of Theorem~\ref{t:1.2i}]
Existence of the joint limit holds thanks to Theorem~\ref{t:1.4i} and Lemma~\ref{l:1104.3}. To get the measurability of $\cZ_\rmD$ w.r.t. $\eta_\rmD$ and the independence of $\eta_\rmD$ of $\ol{h}_\rmD$ given $\cZ_\rmD$, we proceed as in the proof of Theorem~\ref{t:1.4i} and recover $\cZ_\rmD$ and $\wh{\cZ}_\rmD$ from $\eta_\rmD$ and $\wh{\eta}_\rmD$ respectively, in the coupling of Lemma~\ref{l:1104.3}. Then by~\eqref{e:1104.14},~\eqref{e:5.42n} and~\eqref{e:5.43n}, 
\begin{equation}
	\eta_\rmD \,\big|\, \cZ_\rmD, \ol{h}_\rmD \ \overset{\rmd}= \ 
	\wh{\eta}_\rmD \circ \tau^{-1}_{h^{\psi}_\rmD}\,\Big|\, \wh{\cZ}_\rmD, h^{\psi}_\rmD
	\ \sim\  \text{\rm PPP}\bigl(\cZ_\rmD(\textd x)\otimes\texte^{-\alpha s}\textd s\otimes\nu(\rmd \omega)\bigr) \,,
\end{equation}
which gives the desired independence.
\end{proof}

\begin{proof}[Proof of Proposition~\ref{p:300.4}]
Recall that $f_n = h_n + m_n$ and for $r,M > 0$ set,
\begin{equation}
	\wt{\rmS}^M_{n,r}(u) := \Big\{ x \in \rmD_n :\: f_n(x) = \min_{\rmB(x; r_n)} f_n \in [-M, \sqrt{u}]
	\ , \quad  \exists y \in \rmB(r) :\: f_n(x+y) \in [-\sqrt{u}, \sqrt{u}] \Big\} \,.
\end{equation}
We first claim that with probability tending to one as $n \to \infty$ followed by $M,r \to \infty$, there is a bijection between $x \in \wt{\rmS}^M_{n,r}(u)$ and $z \in \bfS_n(u)$ such that, uniformly in such $x$ and $z$,
\begin{equation}
\label{e:5.9mm}
\rme^{-n} \|x-z\| \underset{n\to\infty}\longrightarrow 0 \,.
\end{equation}
 
 Indeed, thanks to Lemma~\ref{t:103.14}, values in $\rmS_n(u)$ must lie in an $r$-neighborhood of an $r_n$-local minimum $x$ of $f_n$ with probability tending to $1$ as $n \to \infty$. Obviously, such local minimum cannot be larger than $\sqrt{u}$ and by tightness of the centered minimum, cannot be smaller than $-M$ with probability tending to $1$ as $n \to \infty$ followed by $M \to \infty$. 
Lemma~\ref{t:103.14} and Lemma~\ref{lem:coverdgff} also imply that the $r$-neighborhood of every such $r_n$ local-minimum $x$ must lie entirely in one and only one box $\rmQ(z; n-r_n)$ for some $z \in \bbX_{n-r_n}$ and that every such $(n-r_n)$-box can include at most one such $r$-neighborhood, again with high probability under the same limits. This gives the bijection. The convergence in~\eqref{e:5.9mm} holds since the side-length of the $(n-r_n)$-box is $o(\rme^{-n})$.

It is therefore sufficient to show~\eqref{e:3.14n} with the point measure
\begin{equation}
	\wt{\bfGamma}_{n,u}^{M,r} := \sum_{x \in \wt{\rmS}^M_{n,r}(u)} \delta_{\rme^{-n}x}
\end{equation}
replacing $\bfGamma_{n,u}$, provided we take $n \to \infty$ followed by $r,M \to \infty$. 
This measure is the composition of $\eta_{n,r_n}$ with the mapping
\begin{equation}
	(x,s,\omega) \mapsto x 1_{\cA_{u,M,r}}(s,\omega)
\end{equation}
where
\begin{equation}
\label{e:403.67}
	\cA_{u,M,r} := \big\{ (s,\omega) \in \bbR \times \bbR_+^{\bbZ^2} :\: s \in [-M, \sqrt{u}] \,,\,\,
		\exists y \in \rmB(r) :\: s + \omega(y) \in [-\sqrt{u}, \sqrt{u}] \big\} \,.
\end{equation}
Observe that $\cA_{u,r,M}$ is a $\mu$-null set, where $\mu$ is the measure given by $\mu(\rmd s \times \rmd \omega)=\rme^{-\alpha s} \otimes \nu (\rmd \omega)$. Indeed, the boundary of $\cA_{u,M,r}$ is included in the set
\begin{equation}
	\Big\{ (x,s,\omega) \in \rmD \times \bbR_+ \times \bbR^{\bbZ^2} :\: 
	s \in \{-M, \sqrt{u}\} \ \ \text{or}\ \ 
		\exists y \in \rmB(r) :\: s + \omega(r) \in \{-\sqrt{u}, \sqrt{u}\} \Big\} \,,
\end{equation}
whose $\mu$-measure is equal to 
\begin{equation}
\int 1_{\{s \in \{-M, \sqrt{u}\}\}} \rme^{-\alpha s} \rmd s + 
\int \nu(\rmd \omega) \int 1_{\{\exists y \in \rmB(r) :\: s+\omega(y) \in \{-\sqrt{u}, \sqrt{u}\}\}} \rme^{-\alpha s} \rmd s = 0,
\end{equation}
where, we have used Fubini-Tonelli to exchange the integrals in the first term. Therefore, thanks to the weak convergence in Theorem~\ref{t:1.2i}, we obtain that \eqref{e:3.14n} holds for $\wt{\bfGamma}^{M,r}_{n,u}$ in place of $\bfGamma_{n,u}$ and $C_{u,M,r}:=\mu(\cA_{u,M,r})$ in place of $C_u$, in the limit as $n \to \infty$ followed by $M,r \to \infty$. From~this, a standard argument shows that $\lim_{M,r \to \infty} C_{u,M,r}$ exists and, moreover, that \eqref{e:3.14n} holds for 
\begin{equation}
C_u:=\lim_{M,r \to \infty} C_{u,M,r}	.
\end{equation} Finally, the fact that $C_u \in (0,\infty)$ follows from the tightness and non-degeneracy of the set $\rmS_n(u)$ for $u$ sufficiently large, see for example \cite[Proposition~3.14]{Tightness}. 
\end{proof}

\appendix
\section{Proofs of auxiliary results}\label{sec:auxproof}

We now give the proofs of all auxiliary results throughout the article. We separate these results into different subsections, for convenience of the reader.

\subsection{Disregarding low local time clusters on sets of small proportion}

We begin by stating a general lemma which be instrumental in the proof of Proposition~\ref{p:3.8nn}. 

\begin{lem}
	\label{l:1103.7}
	Fix $u \geq 0$ and, for each $n \geq 1$, some deterministic subset $\wt{\rmD}_n \subseteq \bbZ^2$. In addition, let us set
	\begin{equation}
		\rmW^{\wt{\rmD}_n}_n(u) = \rmW_n(u) \cap \wt{\rmD}_n
		\quad ; \qquad 
		\bfW^{\wt{\rmD}_n}_n(u) := \big\{ z \in \bbX_{n-r_n} :\: \rmW^{\wt{\rmD}_n} \cap \rmQ(z;n-r_n) \neq \emptyset\} \,.
	\end{equation}
	Then,
	\begin{equation}
		\lim_{\delta \downarrow 0} \limsup_{n \to \infty}
		\bbP \bigg( \frac{1}{\sqrt{n}}|\bfW^{\wt{\rmD}_n}_n(u)| > \delta^{-1} \sqrt{|\rmD_n^\circ \cap \wt{\rmD}_n|/|\rmD_n|} \bigg) = 0 \,.
	\end{equation}
\end{lem}

We devote this subsection to the proof of Lemma~\ref{l:1103.7}. For this, we will need two auxiliary results. The first is an estimate on the minimum of the DGFF on small sets.

The second auxiliary result we shall need is a mild version of the Thinning Lemma from \cite{Tightness}. We state this version here for completeness.

\begin{lem}\label{lem:thinning} For each $n \geq 1$, let $\rmY_n$ be a random subset of $\rmD_n$ which is $L_{t^A_n}(\rmD_n)$-measurable. Given $u \geq 0$, define the sets $\rmZ_n(u):=\rmW_n(u) \cap \rmY_n$, $\rmX_n(u+1):=\rmS_n(u+1) \cap \rmY_n$ and 
	\begin{equation}
		\bbZ_n(u):=\{ z \in \bbX_{n-r_n} : \rmZ_n(u) \cap \rmQ(z;n-r_n) \neq \emptyset\}.
	\end{equation} Then, there exist constants $C,c > 0$ (which depend only on $u$) such that, if $n$ is sufficiently large, for each $s > 0$ we have
	\begin{equation}
		\bbP\left(|\bbZ_n(u)|>s \sqrt{n}\right) \leq C(1+s^{-1})\bbP(|\rmX_n(u+1)|>cs). 
	\end{equation}
\end{lem}

\begin{proof} Direct consequence of \cite[Lemma~4.1]{Tightness} upon choosing $\cA_k:=\bbR$ in this lemma (with only the minor difference that here we are considering $\bbX_{n-r_n}$ instead of $\bbX_{\lfloor n -r_n\rfloor}$ as done in \cite{Tightness}, which, nevertheless, does not affect the argument of the proof.)
\end{proof}

\begin{proof}[Proof of Lemma~\ref{l:1103.7}] Fix any $\delta > 0$. Then, since $\bfW_n^{\wt{\rmD}_n}(u) \subseteq \bfW_n(u)$, whenever $|\rmD_n^\circ \cap \wt{\rmD}_n| > \delta|\rmD_n|$ we have
	\begin{equation}\label{eq:dc0}
		\bbP\left( \frac{1}{\sqrt{n}}|\bfW^{\wt{\rmD}_n}_n(u)| > \delta^{-1} \sqrt{|\rmD_n^\circ \cap \wt{\rmD}_n|/|\rmD_n|}\right) \leq \bbP\left( \frac{1}{\sqrt{n}}|\bfW_n(u)| > \frac{1}{\sqrt{\delta}}\right).	
	\end{equation}
	On the other hand, whenever $|\rmD_n^\circ \cap \wt{\rmD}_n| \leq \delta|\rmD_n|$, by taking $Y_n := \wt{\rmD}$ in Lemma~\ref{lem:thinning}, we see that, if $n$ is large enough,
	\begin{equation}\label{eq:dc1}
		\bbP\left( \frac{1}{\sqrt{n}}|\bfW^{\wt{\rmD}_n}_n(u)| > \delta^{-1} \sqrt{|\rmD_n^\circ \cap \wt{\rmD}_n|/|\rmD_n|}\right) \leq C\left(1+\delta \sqrt{\frac{|\rmD_n|}{|\rmD_n^\circ \cap \wt{\rmD}_n|}}\right)\bbP( \rmS_n(u+1) \cap \wt{\rmD}_n \neq \emptyset)
	\end{equation} where, in addition, by Lemma~\ref{lem:coverdgff} we have the bound 
	\begin{equation}\label{eq:dc2}
		\bbP( \rmS_n(u+1) \cap \wt{\rmD}_n \neq \emptyset) \leq \bbP\Big( \min_{x \in \rmD_n^\circ \cap \wt{\rmD}} f_n(x) \leq \sqrt{u+1}\Big) \leq c_\rmD \rme^{\wh{c} \sqrt{u+1}} \sqrt{\frac{|\rmD_n^\circ \cap \wt{\rmD}_n|}{|\rmD_n|}}, 
	\end{equation} so that
	\begin{equation}\label{eq:dc3}
		\bbP\left( \frac{1}{\sqrt{n}}|\bfW^{\wt{\rmD}_n}_n(u)| > \delta^{-1} \sqrt{|\rmD_n^\circ \cap \wt{\rmD}_n|/|\rmD_n|}\right) \leq C_\rmD \big(\sqrt{\delta}+\delta\big)\rme^{\wh{c}\sqrt{u+1}}.	
	\end{equation} for some constant $C_\rmD > 0$. In light of \eqref{eq:dc0} and \eqref{eq:dc3}, the result now immediately follows from the upper tightness of $|\bfW_n(u)|/\sqrt{n}$ (implied by Theorem~\ref{p:300.3} and Theorem~\ref{t:2.1o}). 
\end{proof}

\subsection{Proofs of Section~\ref{sec:iid} preliminaries}

Here we give the proofs of all auxiliary results from Section~\ref{sec:prelimiid}. We begin with Proposition~\ref{prop:GR-1}.

\begin{proof}[Proof of Proposition~\ref{prop:GR-1}] It will be enough to show that, if $k$ is large enough, for any $u,v \in \cN_k$ with $0 \leq \wt{u}_k \leq \rme^{ck^\gamma}$ and $1 \leq \wt{v}_k \leq \rme^{ck^\gamma}$, one has that
	\begin{equation}\label{eq:GR4}
		\bbP\Big( \sqrt{\wh{N}_t[x;i]} \leq u \,\Big|\, \sqrt{\wh{N}_t[\wt{x};j]} \geq v\,;\cF[\wt{x};j]\Big) \leq 2\bbP \Big( \sqrt{\wh{T}^{[i-1,j]}_{(v)}(i)} \leq u\Big) \end{equation} and 
	\begin{equation}\label{eq:GR5}
		\bbP\Big( \sqrt{\wh{N}_t[x;i]} \geq u \,\Big|\, \sqrt{\wh{N}_t[\wt{x};j]} \leq v\,;\cF[\wt{x};j]\Big) \leq 2 \bbP \Big( \sqrt{\wh{T}^{[i-1,j]}_{(v)}(i)} \geq u\Big),
	\end{equation} 
	for $\wh{T}^{[i-1,j]}_{(v)}(i)$ defined as in Section~\ref{sec:prooftransfer}. Indeed, in light of the bounds in \eqref{eq:GR4}--\eqref{eq:GR5},  the result is now a straightforward consequence of Lemma~\ref{lem:compound}, since 
	\begin{equation}
		\frac{1}{k^\gamma} \wh{T}_{(v)}^{[i-1,j]}(i) \sim \Bigeo\left(\wt{v}_k,\frac{1}{j-i+1},\frac{1}{j-i+1}\right).
	\end{equation}	
	Thus, we focus on the proof of \eqref{eq:GR4}--\eqref{eq:GR5}. We first notice that \eqref{eq:GR4} is a direct consequence of~\eqref{eq:GR-3}. Indeed, since by \eqref{eq:ntrep} we have that, on the event $\Big\{\sqrt{\wh{N}_t[\wt{x};j]} \geq v\Big\}$,
	\begin{equation}
		N_t[x;i] \geq \sum_{r=1}^{\wt{v}_k} V_r=N^*_{\wt{v}_k}[x;i], 
	\end{equation} where $N^*_{m}[x;i]$ is the number of $[x;i]$-excursions during the first $m$ $[\wt{x};j]$-excursions of the walk, it follows that
	\begin{equation}
		\bbP\Big( \sqrt{\wh{N}_t[x;i]} \leq u \,\Big|\, \sqrt{\wh{N}_t[\wt{x};j]} \geq v\,;\cF[\wt{x};j]\Big) \leq \sup_{\overline{z} \in \textrm{E}_{\wt{v}_k}[j]}
		\bbP\Big( \sqrt{\wh{N}_t[x;i]} \leq u \,\Big|\, Z_t[\wt{x};j]=(v,\ol{z})\Big).
	\end{equation}  Since $\sum_{\overline{y} \in \textrm{E}_{n}[i]} h_{k,i}(\overline{y})=1$ for all $n \in \N$ because $\amalg_{\rmB^-[0;i]}$ and $\Pi_{\rmB^+[0;i]}$ are probability distributions, if $w \in \mathcal{N}_k$ is such that $\wt{w}_k \leq \rme^{ck^\gamma}$ then, by summing \eqref{eq:GR-3} over all $\overline{y} \in \textrm{E}_{\wt{w}_k}[i]$, we obtain that
	\begin{equation}
		\sup_{\ol{z} \in \textrm{E}_{\wt{v}_k}[j]} \bbP\Big( \sqrt{\wh{N}_t[x;i]} = w \,\Big|\, Z_t[\wt{x};j]=(v,\ol{z})\Big) = 	\bbP \Big( \sqrt{\wh{T}^{[i-1,j]}_{(v)}(i)} = w\Big)(1+O(\rme^{-ck^\gamma})), 
	\end{equation} from where \eqref{eq:GR4} now follows by summing over all $w \leq u$. On the other hand, by \eqref{eq:ntrep} again, on the event $\Big\{\sqrt{\wh{N}_t[\wt{x};j]} \leq v\Big\}$ we have
	\begin{equation}
		N_t[x;i] \leq \sum_{r=1}^{\wt{v}_k} V_r=N^*_{\wt{v}_k}[x;i], 
	\end{equation} so that
	\begin{equation}
		\bbP\Big( \sqrt{\wh{N}_t[x;i]} \geq u \,\Big|\, \sqrt{\wh{N}_t[\wt{x};j]} \geq v\,;\cF[\wt{x};j]\Big) \leq \sup_{\overline{z} \in \textrm{E}_{\wt{v}_k}}
		\bbP\Big( \sqrt{\wh{N}_t[x;i]} \geq u \,\Big|\, Z_t[\wt{x};j]=(v,\ol{z})\Big).
	\end{equation} In particular, to obtain \eqref{eq:GR5}, it will be enough to show that
	\begin{equation}\label{eq:GR-5b}
		\bbP\Big( \sqrt{\wh{N}_t[x;i]} \geq u \,\Big|\, Z_t[\wt{x};j]=(v,\ol{z})\,;\cF[\wt{x};j]\Big)=\bbP\left( \sqrt{ \wh{T}_{(v)}^{[i-1,j]}(i)} \geq u\right)(1+O(\rme^{-ck^\gamma}))
	\end{equation} holds uniformly over $t > 0$ and $(u,\ol{y}) \in \cZ_k[i]$, $(v,\ol{z}) \in \cZ_k[j]$ with $0 \leq \wt{u}_k \leq \rme^{ck^\gamma}$ and $1 \leq v \leq \rme^{ck^\gamma}$. But this can be established by proceeding as in the proof of \eqref{eq:GR-3}. Indeed, in the case $\wt{v}_k=1$, \eqref{eq:GR-5b} is immediate whenever $\wt{u}_k=0$, so that one only needs to check the case $\wt{u}_k \geq 1$, for which the analogue of~\eqref{eq:new-c1} is now 
	\begin{equation}\label{eq:new-c1b}
		\bbP_{z,z'}(N^*_{1}[x;i] \geq \wt{u}_k) =\left(\frac{1}{j-i+1}\right) \left(\frac{j-i}{j-i+1}\right)^{\wt{u}_k-1}(1+\wt{u}_kO(\rme^{-2c_1k^\gamma })).
	\end{equation} To this end, by proceeding as in the proof of \eqref{eq:new-c1} one may check that
	\begin{equation}\label{eq:new-c1b2}
		\bbP_{z,z'}(N^*_{1}[x;i] \geq \wt{u}_k\,,\,\overline{Y}_{(\wt{u}_k)}[x;i]=\overline{y}) = \left(\frac{1}{j-i+1}\right) \left(\frac{j-i}{j-i+1}\right)^{\wt{u}_k-1}h_{k,i}(\ol{y})(1+\wt{u}_kO(\rme^{-2c_1k^\gamma}))
	\end{equation} uniformly over $1 \leq \wt{u}_k \leq \rme^{c_1 k^\gamma}$, $\ol{y} \in (\partial_i \rmB^-[0;i] \times \partial \rmB^+[0;i])^{\wt{u}_k}$ and $(z,z') \in \partial_i \rmB^-[\wt{x};j] \times \partial \rmB^+[\wt{x};j]$. Indeed, the argument is the same, but one needs to replace the term $\bbP_{w',z'}(N^*_1[x;i]=0)$ in~\eqref{eq:new-decoup1} by $\bbP_{w',z'}(N^*_1[x;i]\geq 0)=1$ and \eqref{eq:new-decouple7} by 
	\begin{equation}
		\bbP_{w',z'}(N^*_{1}[x;i] \geq \wt{u}_k-1, \ol{Y}_{(\wt{u}_k-1)}[x;i]=(\ol{y}_2,\dots,\ol{y}_{\wt{u}_k}))\,. 	
	\end{equation}
	From this, by summing \eqref{eq:new-c1b2} over all $\ol{y} \in \rmE_{\wt{u}_k}[i]$, we obtain \eqref{eq:new-c1b}. Finally, the case for $\wt{v}_k > 1$ can be handled exactly as in the proof of \eqref{eq:GR-3}. We omit the details.
\end{proof} 

Proposition~\ref{prop:LT-1} is the analogue of \cite[Proposition~3.8]{Tightness}, albeit in a slightly different setting. However, the method of proof remains unchanged. The same applies to Proposition~\ref{prop:LT-2}, which is the analogue in our present setting of  \cite[Proposition~3.7]{Tightness}. Hence, we refer the reader to \cite{Tightness} for their proofs and omit the details.

We now continue with the proof of the Resampling Lemma. The strategy of proof goes along the lines of that of \cite[Lemma 4.3]{Tightness}, but there are, nevertheless, a few differences stemming from the fact that we are working here with a more general setting.

\begin{proof}[Proof of Lemma~\ref{lem:resampling}] Notice that, given $z \in \bfW_n(u)$, $\bbY^{i}_n(u;z)$ may have more than one element but, if we write $k_i:=k+(i-2)k^\gamma$ to simplify the notation and recalling the subgrids $\bbX_{m_1,m_2}$ introduced in Section~\ref{sec:exist1}, for each $j \in J_{k_i,k_i+2k^\gamma}$ there will be at most one $y_j \in \bbX_{k_i,k_i+2k^\gamma}(j)$ satisfying $\rmW_n(u) \cap \rmQ(z; n - r_n) \cap \rmB^-[y_j(z);i-1] \neq \emptyset$ if $k$ is large enough (depending only on $\gamma$).  Therefore, if we define 
	\begin{equation}
		\begin{split}
			\bbY_n^i(u;j) := 
			\Big \{  z \in \bfW_n(u)  :\,  &
			\exists! \,y_j=y_{j}(z) \in \,\bbX_{k_i,k_i+2k^\gamma}
			(j) \text{ s.t. } \\ &\rmB^-[y_j(z);i-1] \cap \rmW_n(u) \cap \rmQ\big(z;n-r_n\big) \neq \emptyset\,,\,\sqrt{\wh{N}_{t}[y_{j}(z);i]} \in M_{k,n}^{i}\Big\}
		\end{split}
	\end{equation} together with
	\begin{equation}
		|\bbZ_n^i(u;j)|:= \sum_{ z \in \bbY_n^i(u;j)}  f^i_{n,u,y,\sqrt{\wh{N}_t[y;i]}}(L_t(\ol{\rmB^+[y_j(z);i-1]})), 
	\end{equation} then we obtain the inequality
	\begin{equation}
		|\bbZ^{i}_n(u)| \leq 
		\sum_{j=1}^{|J_{k_i,k_i+2k^\gamma}|} |\bbZ_n^i(u;j)|.
	\end{equation} 
	Now, if we consider $\mathfrak{Y}^{i}_n(u;j):=\{ y_j(z) : z \in  \bbY_n^i(u;j)\}$, then $\bbY_n^i(u;j)$ is a measurable function of $\mathfrak{Y}^{i}_n(u;j)$ and moreover, since the balls $(\ol{\rmB^+[y; i]} : y \in \bbX_{k_i, k_i+2k^\gamma}(j))$ are disjoint for $k$ large~enough, conditional on $\mathfrak{Y}^{m}_n(u;i)$ the random variables 
	\begin{equation}
		\left( \Big| f^i_{n,u,y,\sqrt{\wh{N}_t[y;i]}}(L_t(\ol{\rmB^+[y_j(z);i-1]}))\Big| : z \in \bbY_n^i(u;j)\right)	
	\end{equation}
	have expectation uniformly bounded by $b_{k,n}^{(i)}(u)$ from \eqref{eq:defbn.0}.
	Since $\bbY^{i}_n(u;j) \subseteq \bfW_n(u)$, it follows from Tchebychev's inequality (conditionally on $\mathfrak{Y}^i_n(u;j)$) that, for any $\kappa > 0$, 
	\begin{equation}\label{eq:boundDn}
		\bbP\bigg(  |\bbZ_n^i(u;j)| > \frac{\kappa}{|J_{k_i,k_i+2k^\gamma}|}|\bfW_n(u)|\bigg) \leq \kappa^{-1}|J_{k_i,k_i+2k^\gamma}|b^{(i)}_{n,k}(u).
	\end{equation} Since $|J_{k_i,k_i+2k^\gamma}|\leq C \rme^{4k^\gamma}$ for some universal constant $C>0$, summing over all $j \in J_{k_i,k_i+2k^\gamma}$ and applying the union bound immediately gives \eqref{eq:resamplebound1}. 
	
	To obtain \eqref{eq:resamplebound2}, we fix $\delta > 0$ and choose $\kappa:=\delta \rme^{-(k+ik^\gamma)^{\frac{1}{2}-\eta}}$ above, so that, since $\gamma < \frac{1}{2}-\eta$, for all $k$ large enough we obtain 
	\begin{equation}
		\limsup_{n \to \infty} \sup_{i \in \cC_{k,n}} \bbP\Big( |\bbZ_n^i(u)| > \delta \rme^{-(k+ik^\gamma)^{\frac{1}{2}-\eta}}|\bfW_n(u)|\Big) \leq 2C\delta^{-1}\rme^{-(k+ik^\gamma)^{\frac{1}{2}-\eta}}.
	\end{equation} At the same time, by the union bound we have that, for any $\delta > 0$, 
	\begin{multline}
		\bbP \Big( \exists i \in \cC_{k,n} :\: |\bbZ^i_n(u)| | > 
		\rme^{-(k+ik^\gamma)^{\frac{1}{2}-2\eta}} \sqrt{n} \Big) \\
		\leq 
		\bbP \Big( \big|\bfW_n(u)\big| > \delta^{-1} \sqrt{n} \Big) +
		\sum_{i \in \cC_{k,n}}
		\bbP \Big( |\bbZ^i_n(u)| > \delta \rme^{-(k+ik^\gamma)^{\frac{1}{2}-2\eta}} \big|\bfW_n(u)\big| \Big) \,.
	\end{multline} Hence, 
	by first taking $n \to \infty$, then $k \to \infty$ and finally $\delta \to 0$, the lower tightness of $|\bfW_n(u)|/\sqrt{n}$ (implied by Theorem~\ref{t:2.1o}) and \eqref{eq:boundDn} combined yield \eqref{eq:resamplebound2} by a straightforward computation.
\end{proof}

Finally, we give the proof of Lemma~\ref{lem:incball.1}. 

\begin{proof}[Proof of Lemma~\ref{lem:incball.1}] The result is immediate if $i=2$ since, in that case, we may take~$y:=x$. For the case $i \geq 3$, it is enough to show that, given $x \in \bbX_k$ with $\ol{\rmB^-[x;i+1]}\subseteq \rmD_n$, there exists $y \in \bbX_{k+(i-2)k^\gamma}$ such that $\rmB^-[x;i-2] \subseteq \rmB^-[y;i-1]$, $ \ol{\rmB^+[y;i]} \subseteq \rmD_n$ and 
	\begin{equation}
		\rmB(y;k+(i-1)k^\gamma -\rme^{-k^\gamma}) \subseteq \rmB^-[x;i] \subseteq \rmB(y;k+(i-1)k^\gamma +\rme^{-k^\gamma})
	\end{equation}
	\begin{equation}
		\rmB(y;k+ik^\gamma -5\rme^{-k^\gamma}) \subseteq \rmB^+[x;i] \subseteq \rmB(y;k+ik^\gamma -3\rme^{-k^\gamma}).
	\end{equation} But, if we choose any $y \in \bbX_{k+(i-2)k^\gamma}$ such that $\|y-x\| \leq \frac{1}{\sqrt{2}}\rme^{k+(i-2)k^\gamma}$ (which always exists by definition of $\bbX_{k+(i-2)k^\gamma}$), then it is straightforward to check that, whenever $k$ is large enough (depending only on $\gamma$), $y$ satisfies $\rmB^-[x;i-2] \subseteq \rmB^-[y;i-1]$ and, also, that for any $r \geq k+(i-1)k^\gamma$, 
	\begin{equation}
		\rmB(y;r-\rme^{-k^\gamma}) \subseteq \rmB(x;r) \subseteq \rmB(y;r+\rme^{-k^\gamma}),	
	\end{equation} from where the rest of the required conditions immediately follow.
\end{proof}

\subsection{Proofs of Proposition~\ref{p:3.8nn}, Proposition \ref{prop:rest} and Proposition~\ref{prop:rest2}}\label{sec:prest}

We aim in this subsection  to give the proofs of Propositions~\ref{p:3.8nn}, \ref{prop:rest} and \ref{prop:rest2}. As a matter of fact, since $\bfR^{k,\epsilon}_n(u) \subseteq \bfR^{k}_n(u) \subseteq \bfW^{k}_n(u) \subseteq \bfW_n(u)$ and also $\bfR^{k,\epsilon}_n(u) \subseteq \bfW^k_{n,\epsilon}(u) \leq \bfW_n(u)$, 
all three propositions will be an immediate consequence of the following stronger result.

\begin{prop}\label{prop:rest3} If $\gamma$ is chosen small enough, then, given any $0 < \eta' < \eta < \frac{1}{2}$ small enough \mbox{(depending only on $\gamma$),} for any fixed $u \geq 0$ we have
	\begin{equation}
		\frac{|\bfW_n(u) \setminus \bfR^{k,\epsilon}_n(u)|}{\sqrt{n}}\overset{\bbP}{\longrightarrow} 0	
	\end{equation} in the limit as $n \to \infty$, followed by $k \to \infty$ and finally $\epsilon \to 0$.
\end{prop}

Therefore, we shall devote this subsection to the proof Proposition~\ref{prop:rest3}. Our first task will~be to show that the restriction of $\bfW_n(u)$ to $\bfW^{k,\epsilon}_n(u)$ is asymptotically harmless.

\begin{lem}\label{l:300.3.10} For any fixed $u \geq 0$, 
	\begin{equation}
		\frac{|\bfW_n(u) \setminus \bfW_n^{k,\epsilon}(u)|}{\sqrt{n}} \overset{\bbP}{\longrightarrow}0,  
	\end{equation}
	in the limit as $n \to \infty$, followed by $k \to \infty$ and finally $\epsilon \to 0$.	
\end{lem}

\begin{proof} By Theorem~\ref{p:300.3}, it suffices to show that
	\begin{equation}
		\frac{|\bfW_n^{k}(u) \setminus \bfW_n^{k,\epsilon}(u)|}{\sqrt{n}} \overset{\bbP}{\longrightarrow}0,  
	\end{equation}
	in the limit as $n \to \infty$, followed by $k \to \infty$ and finally $\epsilon \to 0$. To this end, we first observe that, in the notation of Lemma~\ref{l:1103.7}, we have
	\begin{equation}
		\frac{|\bfW_n^{k}(u) \setminus \bfW_n^{k,\epsilon}(u)|}{\sqrt{n}} \leq |\bfW_n^{\wt{\rmD}}(u)|,	
	\end{equation} where
	$\wt{\rmD}:=\rmA^{k,\epsilon} \cup \rmA^{n-r_n,\epsilon}
	$ with $\rmA^{m,\epsilon}$ for $m \in \{k,n-r_n\}$ given by 
	\begin{equation}
		\rmA^{m,\epsilon}:= \bigcup_{x \in \bbX_m} 	\rmQ(x;m) \setminus \rmQ(x;m-\epsilon).
	\end{equation}
	Now, notice that, since $|\rmQ(m)\setminus \rmQ(m-\epsilon)| \leq (\epsilon++o_m(1))|\rmQ(m)|$ for some $o_m(1) \to \infty$ as $m \to 1$ and, in addition, $\rmQ(z;n-r_n) \subseteq \rmD_n$ holds for all $z \in \rmD_n^\circ \cap \wt{\rmD}$ whenever $n$ is large enough, we~have  $|\rmD_n^\circ \cap \wt{\rmD}| \leq (2\epsilon+o_k(1)+o_n(1))|\rmD_n|$. The result now follows at once from Lemma~\ref{l:1103.7}.%\footnote{Adjust the statement of Lemma~\ref{l:1103.7} so that it fits with this. Need to take general $\wt{\rmD}$ and then the quotient $|\wt{\rmD}\cap \rmD_n^\circ|/|\rmD_n|$ in the probability.}
\end{proof}

We are now ready to prove Proposition~\ref{prop:rest3}.

\begin{proof}[Proof of Proposition~\ref{prop:rest3}] In light of Lemma~\ref{l:300.3.10}, it will suffice to show that, for any fixed $u \geq 0$ and $\epsilon \in (0,1)$, we have
	\begin{equation}\label{eq:rest3.1}
		\frac{|\bfW_n^{k,\epsilon}(u) \setminus \bfR_n^{k,\epsilon}(u)|}{\sqrt{n}} \overset{\bbP}{\longrightarrow}0,  
	\end{equation} in the limit as $n \to \infty$ followed by $k \to \infty$. This will follow from an application of the Thinning and Resampling Lemmas from \cite{Tightness}.
	
	Indeed, for any $z \in \bfW_n^{k,\epsilon}(u)$ there exists (a unique) $x(z) \in \bbX_k$ with $\rmQ(x(z);k) \subseteq \rmQ(z;n-r_n)$ and $\rmW_n(u) \cap \rmQ(x(z);k) \neq \emptyset$. Since $\rmW_n(u) \cap \rmQ(z;n-r_n) \subseteq \rmQ(z;n-r_n-\eps)$, the former implies that, if $n$ is large enough (depending only on $k$ and $\eps$), in fact $\rmW_n(u) \cap \rmQ(z;n-r_n) \subseteq \rmQ(x(z);k)$.
	On the other hand, if in addition $z \notin \bfR_n^{k,\epsilon}(u)$, then, since $\overline{\rmB^+[x(z);\cT_k+6]}\subseteq \rmD_n$ for all $n$ large enough (depending on $k$ and $\gamma)$, there must exist some $i \in \{1,\dots,\cT_k+5\}$ such that 
	\begin{equation}
		\sqrt{\wh{N}_{t^A_n}[x(z);i]} \notin \begin{cases}\cR^{\frac{\eta'}{2}}_k(i) & \text{ if $i \geq 2$}, \\ \\ \cR^{\eta'}_k(i) & \text{ if $i=1$,}\end{cases} 
	\end{equation} where above we have used that $\cR^{\frac{\eta'}{2}}_k(i)\subseteq \cR^{\eta'}_k(i) \subseteq \cR^{\eta}_k(i)$ because $\frac{\eta'}{2} < \eta' < \eta$. Whenever $i \geq 2$, by Lemma~\ref{lem:incball.1}, there exists $y \in \bbX_{k+(i-2)k^\gamma}$ such that 
	\begin{equation}	
		\rmQ(x(z);k) \subseteq \rmB^-[x(z),\max\{i-2,1\}] \subseteq \rmB^-[y;i-1]
	\end{equation}
	and 
	\begin{equation}
		\wh{N}^{\textrm{in}}_{t^A_n}[y;i] \geq \wh{N}_{t^A_n}[x(z);i]	 \geq \wh{N}^{\textrm{out}}_{t^A_n}[y;i].	
	\end{equation}
	On the other hand, if $i=1$ and $\sqrt{\wh{N}_{t^A_n}[x(z);2]} \in \cR^{\frac{\eta'}{2}}_k(2)$, then, for all $k$ large enough (depending only on $\gamma$ and $\eta'$),
	\begin{equation}
		\left|\sqrt{\wh{N}_{t^A_n}[x(z);2]}-\sqrt{\wh{N}_{t^A_n}[x(z);1]}\right| \geq k^{\frac{1}{2}-\eta'}.	
	\end{equation}
	It follows that if we define 
	\begin{equation}
		\begin{split}
			\mathbf{M}^{k,*}_n(u):=\Big\{ z \in \bfW_n(u) : \exists x \in \bbX_{k} \text{ s.t. }& \rmW_n(u) \cap \rmQ(z;n-r_n) \subseteq \rmQ(x;k), 	\\& \sqrt{\wh{N}_{t^A_n}[x(z);2]} \in \cR^{\frac{\eta'}{2}}_k(2),\\ & \left|\sqrt{\wh{N}_{t^A_n}[x(z);2]}-\sqrt{\wh{N}_{t^A_n}[x(z);1]}\right| \geq k^{\frac{1}{2}-\frac{\eta'}{2}}\Big\} 
		\end{split}
	\end{equation}
	and, for $i=2,\dots,\cT_k+5$, 
	\begin{equation}
		\begin{split}
			\mathbf{M}^{k,+,i}_n(u):=\Big\{ z \in \bfW_n(u) : \exists y \in \bbX_{k+(i-2)k^\gamma}\text{ s.t. } & \rmW_n(u) \cap \rmQ(z;n-r_n) \subseteq \rmB^-[y;i-1], 	\\& \wh{N}^{\textrm{in}}_{t^A_n}[y;i] \geq \sqrt{2}(k+(i-1)k^\gamma)+(k+ik^\gamma)^{\frac{1}{2}+\frac{\eta'}{2}}\Big\}, 
		\end{split}
	\end{equation}together with
	\begin{equation}
		\begin{split}
			\mathbf{M}^{k,-,i}_n(u):=\Big\{ z \in \bfW_n(u) : \exists y \in \bbX_{k+(i-2)k^\gamma}\text{ s.t. } & \rmW_n(u) \cap \rmQ(z;n-r_n) \subseteq \rmB^-[y;i-1],	\\& \wh{N}^{\textrm{out}}_{t^A_n}[y;i] \leq \sqrt{2}(k+(i-1)k^\gamma)+(k+ik^\gamma)^{\frac{1}{2}-\frac{\eta'}{2}}\Big\} 
		\end{split}
	\end{equation}
	then $\bfW^{k,\epsilon}_n(u) \setminus \bfR^{k,\epsilon}_n(u) \subseteq \mathbf{M}^{k,*}_n(u) \cup \mathbf{M}_n^{k,+,[2,\cT_k+5]}(u) \cup \mathbf{M}_n^{k,-,[2,\cT_k+5]}$, where, as usual, we extended the definition of $\mathbf{M}_n^{k,\pm,i}(u)$ to $[2,\cT_k+5]$ in place of $i$ by taking the union over all $i=2,\dots,\cT_k+5$. Thus, in order to conclude the result it will suffice to show that, for any $\delta > 0$,
	\begin{equation}\label{eq:rl1.1}
		\lim_{k \to \infty} \lim_{n \to \infty} \bbP( |\mathbf{M}^{k,*}_n(u)| > \delta \sqrt{n}) = 0	
	\end{equation} and
	\begin{equation}\label{eq:tlrl1.2}
		\lim_{k \to \infty} \lim_{n \to \infty} \bbP( |\mathbf{M}^{k,\pm,[2,\cT_k+5]}_n(u)| > \delta \sqrt{n}) = 0.	
	\end{equation} To prove \eqref{eq:rl1.1} one can use the Lemma~\ref{lem:resampling} the only key input being the fact that, if we define 
	\begin{equation}
		b_{k,n}^{(2)}(u):= \sup_{x,(v,\ol{z})} \bbP\bigg( \Big| \sqrt{\wh{N}_{t^A_n}[x;1]} - v\Big| \geq k^{\frac{1}{2}-\eta'}\, \bigg|\, Z_{t^A_n}[x;2]=(v,\ol{z})  \,,\, \rmW_n(u) \cap \rmQ(x;k)\neq \emptyset\bigg),
	\end{equation}  where the supremum is over all $(v,\ol{z}) \in \cZ^{\frac{\eta'}{2}}_k[2]$ and $x$ satisfying $\ol{\rmB^+[x;2]} \subseteq \rmD_n$, then, by virtue of Proposition~\ref{prop:GR-1} and Proposition~\ref{prop:LT-1}, we have, for some constant $C_u>0$ and all $k$ large enough, 
	\begin{equation}
		b_{k,n}^{(2)}(u) \leq 4\rme^{-\frac{1}{2}k^{1-2\eta'-\gamma} +C_uk^{\frac{1+\eta'}{2}}} \leq \rme^{-\frac{1}{4}k^{1-2\eta'-\gamma}}
	\end{equation} if $\eta'$ is small enough (depending only on $\gamma$). On the other hand, in order to establish \eqref{eq:tlrl1.2}, one needs to combine the Thinning and Resampling Lemmas as done in the proof of \cite[Lemma~7.2]{Tightness}, with the only (minor) difference that here one needs to work with $\wh{N}^{\textrm{in}}_t[y;i]$ and $ \wh{N}^{\textrm{out}}_t[y;i]$ instead of the downcrossings $N_t(y;m)$ from \cite{Tightness}. We omit the details.
\end{proof}

\subsection{Proofs of Section~\ref{s:barrier} preliminaries} \label{sec:proofprelimballot}

We now give the proofs of all auxiliary results from Section~\ref{sec:prelimballot}. We begin with the gambler's~ruin estimates for simple random walk on $\bbZ^2$.  

\begin{proof}[Proof of Lemma~\ref{lem:ll1}] Follows immediately from Proposition~6.4.1 and Proposition~6.4.3 in~\cite{LL}. 
\end{proof}

Then, we have the asymptotic estimates for the Poisson kernal and harmonic measure of~balls.

\begin{proof}[Proof of Lemma~\ref{lem:kernelfrominf}] The estimate \eqref{eq:miss1} is essentially given by \cite[Lemma~A.2]{Tightness}, but we include a short proof here for completeness. By \cite[Theorem~6.3.8]{LL} there exists $C> 0$ such that
	\begin{equation}\label{eq:qpk1}
		1-C\rme^{-(l+r')} \leq \frac{\Pi_{\rmB(\wt{x};l+r')}(v,z)}{\Pi_{\rmB(\wt{x};l+r')}(u,z)} \leq 1+C\rme^{-(l+r')}
	\end{equation} uniformly over all pairs of nearest neighbors $u,v \in \rmB(\wt{x};l)$ and all $z \in \partial \rmB(\wt{x};l+r')$. Indeed, this follows from applying \cite[Equation~(6.19)]{LL} to the map $f_{u,z}(\cdot):=\Pi_{\rmB(\wt{x};l+r')}(\cdot+u,z)$ which is harmonic on $\rmB(0;l+r'-1)$ if $r' \geq 1$ since in that case we have $\rmB(u;l+r'-1) \subseteq \rmB(\wt{x};l+r')$. 
	Hence, since any $y \in \rmB(\wt{x};l+r')$ can be joined with $\wt{x}$ by a path having at most $\rme^{l}$ steps, by iterating \eqref{eq:qpk1} over all pairs of consecutive neighbors on such path, a straightforward computation then yields \eqref{eq:miss1} for all $r'$ large enough.
	
	To show \eqref{eq:miss2}, we notice that, if we consider the hitting event $\rmH:=\{ \tau_{\partial \rmB(\wt{x};l+r')} <  \overline{\tau}_{\rmB(x;k)}\}$, then, since $\bbP_y(\rmH)>0$ for all $y \in \rmB(\wt{x};l)\setminus (\rmB(x;k) \setminus \partial_i \rmB(x;k))$, \eqref{eq:miss1} implies that 
	\begin{equation}\label{eq:pmiss2a}
		\Pi_{\rmB(0;l+r')}(0,z)(1+O(\rme^{-\ol{c}r'}))=\bbP_{y} ( Z_{\tau_{\partial \rmB(\wt{x};l+r')}}=\wt{x}+z \,|\, \rmH) \bbP_y(\rmH) + \bbP_y(  \{ Z_{\tau_{\partial \rmB(\wt{x};l+r')}}=\wt{x}+z \} \cap \rmH^c).
	\end{equation} Now, since $\rmB(x;k) \subseteq \rmB(\wt{x};l)$, by \eqref{eq:miss1} together with the strong Markov property at time $\tau_{\rmB(x;k)}$ we obtain that
	\begin{equation}
		\bbP_y(  \{ Z_{\tau_{\partial \rmB(\wt{x};l+r')}}=\wt{x}+z \} \cap \rmH^c) =  \Pi_{\rmB(0;l+r')}(0,z)(1+O(\rme^{-\ol{c}r'}))\bbP(\rmH^c), 
	\end{equation} which, in combination with \eqref{eq:pmiss2a}, immediately yields \eqref{eq:miss2}. 
	
	We conclude by proving~\eqref{eq:miss3}. To this end, we first notice that, by arguing as in the proof of the last-exit decomposition from~\cite[Proposition~4.6.4]{LL}, for any $y \in \rmB(\wt{x};l+r') \setminus \rmB(x;k+r)$ and $\wt{w} \in \partial_i \rmB(x;k)$, whenever $r' \geq 1$ (so that $\partial \rmB(x;k+r) \subseteq \rmB(\wt{x};l+r') \setminus \rmB(x;k)=:A$) we have
	\begin{equation}\begin{split}
			\bbP_y(Z_{\tau_{\rmB(x;k)}}=\wt{w}\,,\,\tau_{\rmB(x;k)} < \tau_{\partial \rmB(\wt{x};l+r')}) &=\bbP_y( Z_{\tau_{A^c}}=\wt{w}\,,\,\tau_{\partial \rmB(x;k+r)} < \tau_{A^c})\\ & = \sum_{z \in \partial \rmB(x;k+r)} G_{A}(y,z)\bbP_z( Z_{\tau_{A^c}}=\wt{w}\,,\, \tau_{A^c}<\ol{\tau}_{\partial\rmB(x;k+r)}),
		\end{split}
	\end{equation} where, for the first equality, we have used the fact that, starting from $y \in \rmB(\wt{x};l+r') \setminus \rmB(x;k+r)$, the walk must first hit $\partial \rmB(x;r+k)$ in order to reach $\rmB(x;k)$.
	On the other hand, by reverting the path of the walk and \eqref{eq:miss2} (used for $k+r$ in place of $l+r'$), we obtain 
	\begin{equation}
		\begin{split}
			\bbP_z( Z_{\tau_{A^c}}=\wt{w}\,,\, \tau_{A^c}<\ol{\tau}_{\partial \rmB(x;k+r)}) & = \bbP_{\wt{w}}(Z_{\tau_{\partial\rmB(x;k+r)}}=z\,,\,\tau_{\partial \rmB(x;k+r)} < \overline{\tau}_{\rmB(x;k)}) \\
			& = \Pi_{\rmB(0;k+r)}(0,z-x)\bbP_{\wt{w}}(\tau_{\partial \rmB(x;k+r)} < \overline{\tau}_{\rmB(x;k)})(1+O(\rme^{-\ol{c}r}))
		\end{split}
	\end{equation} for all $r$ large enough, from which a straightforward computation shows that
	\begin{equation}\label{eq:miss3cond}
		\bbP_y(Z_{\tau_{\rmB(x;k)}}=\wt{w}\,|\,\tau_{\rmB(x;k)} < \tau_{\partial \rmB(\wt{x};l+r')}) =\frac{\bbP_{\wt{w}}(\tau_{\partial \rmB(x;k+r)} < \overline{\tau}_{\rmB(x;k)})}{\sum_{\wt{w}' \in \partial_i \rmB(x;k)}\bbP_{\wt{w}'}(\tau_{\partial \rmB(x;k+r)} < \overline{\tau}_{\rmB(x;k)})}(1+O(\rme^{-\ol{c}r})).
	\end{equation} 
	On the other hand, by \eqref{eq:amalgdef} and the fact that $g_A \equiv 0$ on $A$, it follows that, for any $w \in \partial_i \rmB(0;k)$, 
	\begin{equation}\label{eq:miss3condb}
		\begin{split}
			\amalg_{\rmB(0;k)}(w)&=\sum_{z \in \partial \rmB(0;k+r)} g_{\rmB(0;k)}(z)\bbP_{w}(Z_{\tau_{\partial\rmB(0;k+r)}}=z\,,\,\tau_{\partial \rmB(0;k+r)} < \overline{\tau}_{\rmB(0;k)})\\
			&=r \bbP_w(\tau_{\partial \rmB(0;k+r)} < \overline{\tau}_{\rmB(0;k)})(1+O(\rme^{-k}))
		\end{split}
	\end{equation}
	where the last inequality holds since $g_{\rmB(0;k)}(z)=r+O(\rme^{-k})$ uniformly over all $z \in \partial \rmB(0;k+r)$. Since $\sum_{w \in \partial_i \rmB(0;k)} \amalg_{\rmB(0;k)}(w)=1$ by definition, the former implies that
	\begin{equation}\label{eq:miss3condc}
		\sum_{w \in \partial_i \rmB(0;k)} \bbP_w(\tau_{\partial \rmB(0;k+r)} < \overline{\tau}_{\rmB(0;k)}) = \frac{1}{r}(1+O(\rme^{-k}))
	\end{equation} which, by translation invariance and in light of \eqref{eq:miss3cond}, together with \eqref{eq:miss3condb} immediately yields the result.
\end{proof}

We next prove Lemma~\ref{lem:idlaw}.

\begin{proof}[Proof of Lemma~\ref{lem:idlaw}] (i) follows from the fact that, by standard gambler's ruin estimates, each (downward) excursion from $T-i$ reaches $T-j-1$ with probability $\frac{1}{j-i+1}$ and, whenever it does, the probability of repeating a downcrossing from $T-j$ to $T-j-1$ before returning to $T-i$ is again $\frac{1}{j-i+1}$. 
	
	On the other hand, (ii) is a consequence of the fact that, again by standard gambler's ruin estimates, in each of these $\wt{v}_k$ independent excursions away from~$T$, there is a probability $\frac{1}{i}$ of reaching $T-i$ and, conditional on hitting $T-i$, the local time spent at $T-i$ has distribution $\textrm{Exponential}(\frac{1}{i})$. 
	
	Finally, we obtain (iii) upon noticing that, conditional on spending $\wt{\theta}_k$ local time at site $T-i$, by standard gambler's ruin estimates, we have that the number of (downward) excursions away from $T-i$ that reach $T-j$ is $\textrm{Poisson}(\frac{\wt{\theta_k}}{j-i+1})$, while the number of downcrossings from $T-j$ to $T-j-1$ in each of these excursions is $\textrm{Geometric}(\frac{1}{j-i+1})$. 	
\end{proof}

Finally, we have the proof of the tail estimates for compound distributions.

\begin{proof}[Proof of Lemma~\ref{lem:compound}]
	The left tail estimate from (i) is contained in \cite[Lemma 4.6]{belius2017subleading}, while those from (ii) and (iii) are in \cite[Lemma A.3]{Tightness}. The right tail estimates all follow by similar arguments. Hence, we refer the reader to \cite{belius2017subleading,Tightness} and omit the details.
\end{proof}

\section*{Acknowledgments}
Both authors would like to thank Marek Biskup and Amir Dembo for many stimulating discussions. The work of O.L. 
was supported by ISF grants no.~2870/21 and~3782/25. The work of S.S. was partially supported by Fondecyt grant 1240848.

\bibliographystyle{abbrv}
\bibliography{BoxCoverTime}
\end{document}